\documentclass[12pt,letterpaper,twoside,openright,openany]{book}
\usepackage[left=3cm,right=3cm, top=3.5cm, bottom=3.5cm]{geometry}
\usepackage[utf8]{inputenc}
\usepackage[T1]{fontenc} 
\usepackage{fancyhdr}

\usepackage{adjustbox,lipsum}

\usepackage{titlesec}
\titleformat{\chapter}[display]   
{\normalfont\huge\bfseries}{\chaptertitlename\ \thechapter}{20pt}{\LARGE}   
\titlespacing*{\chapter}{0pt}{-50pt}{40pt}

\usepackage{chngcntr}

\usepackage{amsthm}
\newtheorem{theorem}{Theorem}[section]
\newtheorem{lemma}[theorem]{Lemma}

\newtheorem{definition}[theorem]{Definition}
\newtheorem{example}[theorem]{Example}

\newtheorem{remark}{Remark}

\usepackage{wrapfig}
\usepackage{amsfonts, amssymb}
\usepackage{graphicx}
\usepackage{algpseudocode}
\usepackage{algorithm}
\usepackage{algorithmicx}

\usepackage{url}            
\usepackage{booktabs}       
\usepackage{nicefrac}       
\usepackage{microtype}      
\usepackage{natbib}
\usepackage{xcolor}

\usepackage{mathtools}

\usepackage{mathrsfs}

\usepackage{enumitem}
\usepackage{latexsym}
\usepackage{epsfig}
\usepackage{graphicx} 
\usepackage{dsfont}

\usepackage{makecell}
\usepackage{caption}
\usepackage{subcaption}

\usepackage{amsmath, amsfonts, amssymb}

\usepackage{bbm, fancyhdr}
\usepackage{bm}
\usepackage{tikz}
\usepackage{tkz-euclide}
\usepackage{chngcntr}
\usepackage{verbatim}

\usepackage[breaklinks=true,colorlinks=true,linkcolor=red, citecolor=blue, urlcolor=blue]{hyperref}%
\renewcommand{\cite}[1]{\citep{#1}}

\def \W{\mathcal W}

\def\X{\mathcal X}
\def\Y{\mathcal Y}

\def\R{\mathbb R}
\def\E{\mathbb E}

\def\PP{\mathbb P}
\def\Prob{\mathbb P}

\def\U{\mathcal{U}}
\def\V{\mathcal{V}}
\def\Z{\mathcal{Z}}

\def\e{\varepsilon}
\def\la{\langle}
\def\ra{\rangle}
\def\vp{\varphi}
\def\y{\mathbf {y}}

\def\x{\mathbf {x}}
\def\z{\mathbf {z}}

\def\u{\mathbf {u}}
\def\v{\mathbf {v}}

\def\A{\boldsymbol{A}}

\def\c{\boldsymbol{c}}
\def\one{{\mathbf 1}}

\def\WW{\mathbf {W}}
\def\u{\mathbf {u}}
\def\b{\mathbf {b}}
\def\v{\mathbf {v}}

\def\Pp{\mathcal P}

\def\Blm{\boldsymbol{\lambda}}
\def\Bzeta{\boldsymbol{\zeta}}

\def\Beta{\boldsymbol{\eta}}

\def\Blm{\boldsymbol{\lambda}}

\def\pp{\mathtt{p}}
\def\qq{\mathtt{q}}
\def\p{\mathbf {p}}
\def\q{\mathbf {q}}

\def\lm{\lambda}

\def\dm#1{{#1}}
\def\ag#1{{#1}} 
\def\dd#1{{#1}} 

\def\ga#1{{#1}}
\def\gav#1{{#1}} 
\def\avg#1{{#1}}

\def\g#1{{#1}}

\def\dm#1{{#1}}

\title{
\vspace{-3.5cm}\textbf{ \huge
Decentralized Algorithms for Wasserstein Barycenters }\\
~\\
{\large DISSERTATION}
}
\author{
\large zur Erlangung des akademischen \large Grades \\
\large Doctor rerum naturalium (Dr. rer. nat.) \\
\large im Fach Mathematik
~\\
\\
\large eingereicht an der\\
\large Mathematisch-Naturwissenschaftlichen
\large Fakultät \\
\large der Humboldt-Universität zu \large Berlin\\
~\\
\large  von\\
~\\
\large  M.Sc. Darina Dvinskikh\\
\large geboren am 17.11.1993 in Russland
}
\date{  
 \begin{flushleft}
\large Präsidentin der Humboldt-Universität zu Berlin:\\
\large \hspace{1cm} Prof. Dr.-Ing. Dr. Sabine Kunst\\
\vspace{0.5cm}
\large Dekan der Mathematisch-Naturwissenschaftlichen Fakultät:\\
\large \hspace{1cm} Prof. Dr. Elmar Kulke\\
\vspace{0.5cm}
\large Gutachter: \hspace{0.46cm} 1. \large Prof. Dr. Vladimir Spokoiny \\
\hspace{3cm} 2. Prof. Dr. Bernhard Schmitzer\\
\hspace{3cm} 3. Prof. Dr. Jonathan Niles-Weed\\
~\\
\large Tag der mündlichen Prüfung: 18. August 2021
\end{flushleft}}

\begin{document}

 \maketitle

\section*{Abstract}
In this thesis, we consider the Wasserstein barycenter problem of discrete probability measures as well as the population Wasserstein barycenter problem  given by a Fr\'{e}chet mean
from  computational and statistical sides.

The statistical focus is estimating the sample size of measures needed to calculate an approximation of a Fr\'{e}chet mean (barycenter) of probability distributions with a given precision.
For empirical risk minimization approaches, the question of the regularization  is also studied along with proposing a new regularization which contributes to the better complexity bounds in comparison with the quadratic regularization. 

The computational focus is  developing decentralized algorithms for calculating Wasserstein barycenters.
The motivation for dual approaches  is closed-forms for the dual formulation of entropy-regularized Wasserstein distances and their derivatives, whereas the primal formulation has a closed-form expression only  in some cases, e.g., for Gaussian measures.
Moreover, the dual oracle returning the gradient of the dual representation for entropy-regularized Wasserstein distance  can be computed for a cheaper price in comparison with the primal oracle returning the gradient of the (entropy-regularized) Wasserstein distance. The number of dual oracle calls in this case will  be also  less, i.e., the square root of the number of primal oracle calls.
Furthermore,
in contrast to the primal objective, the dual objective has
 Lipschitz continuous  gradient due to the strong convexity of regularized Wasserstein distances. Hence,  accelerated gradient descent-based method for the Lipschitz smooth objective can be used, which is optimal in terms of the number of iterations and oracle calls.
Moreover, we study saddle-point formulation of the non-regularized  Wasserstein barycenter problem which leads to the bilinear saddle-point problem.  Hence, mirror prox algorithm can be used. This approach also allows us to get  optimal complexity bounds  and it can be easily presented in a decentralized setup.
\\

\textbf{Keywords:} optimal transport,  Wasserstein barycenter, stochastic optimization,  decentralized optimization, distributed optimization, primal-dual methods, first-order oracle.

\newpage
\section*{Zusammenfassung}
 In dieser Arbeit beschäftigen wir uns mit dem Wasserstein Baryzentrumproblem diskreter Wahrscheinlichkeitsmaße sowie mit dem population Wasserstein Baryzentrumproblem gegeben von a Fr\'{e}chet Mittelwerts
von der rechnerischen und statistischen Seiten.

Der statistische Fokus liegt auf der Schätzung der Stichprobengröße von Maßen zur Berechnung einer Annäherung  des Fr\'{e}chet Mittelwerts (Baryzentrum) der Wahrscheinlichkeitsmaße mit einer bestimmten Genauigkeit.
Für empirische Risikominimierung (ERM) wird auch die Frage der Regularisierung untersucht zusammen mit dem Vorschlag einer neuen Regularisierung, die zu den besseren Komplexitätsgrenzen im Vergleich zur quadratischen Regularisierung beiträgt.

Der Rechenfokus liegt auf der Entwicklung von dezentralen Algorithmen zur Berechnung von Wasserstein Baryzentrum.
Die Motivation für duale Optimierungsmethoden ist
geschlossene Formen für die duale Formulierung von entropie-regulierten Wasserstein Distanz und ihren Derivaten, während, die primale Formulierung nur in einigen Fällen einen Ausdruck in geschlossener Form hat, z.B. für Gauß-Maße.
Außerdem  kann das duale Orakel, das den Gradienten der dualen Darstellung für die entropie-regulierte Wasserstein Distanz zurückgibt, zu einem günstigeren Preis berechnet werden als das primale Orakel, das den Gradienten der (entropie-regulierten) Wasserstein Distanz zurückgibt. 
Die Anzahl der dualen Orakelrufe ist in diesem Fall ebenfalls weniger, nämlich die Quadratwurzel der Anzahl der primalen Orakelrufe.
Im Gegensatz zum primalen  Zielfunktion,
 hat das duale  Zielfunktion
Lipschitz-stetig Gradient aufgrund der starken Konvexität regulierter Wasserstein Distanz. 
 Deshalb können wir beschleunigte Gradientenverfahren Algorithmus für das Zielfunktion mit Lipschitz-stetig  Gradienten verwendet, die optimal in Bezug auf
 der Anzahl der Iterationen und Orakelaufrufe  sind.  Außerdem
 untersuchen wir die Sattelpunktformulierung des (nicht regulierten) Wasserstein Baryzentrum, die zum Bilinearsattelpunktproblem führt.  Deshalb können wir   
 Spiegel Prox Algorithmus
 verwendet.
 Dieser Ansatz ermöglicht es uns auch, optimale Komplexitätsgrenzen zu erhalten, und kann einfach in einer dezentralen Weise präsentiert werden.\\

\textbf{Stichwörter:}
optimaler Transport,  Wasserstein Baryzentrum, stochastische Optimierung, dezentrale Optimierung, primal-duale Optimierungsmethoden erster Ordnung, Orakel erster Ordnung.

\newpage

\vspace{8cm}
\hspace{10cm}
\textit{To my family}

\newpage

\section*{Acknowledgements}
First of all, I would thank my advisor, Vladimir Spokoiny,
for  his support, his advice and warm research meetings during my PhD at Weierstrass Institute in Berlin. Especially, I am  grateful for his trust and mathematical freedom which he gave me in choosing a research direction provided me with valuable advice. It was an honor for me to work in his  research group and and attend the group seminars where we  got to know the research field of each group member   and prominent visiting scientists from all over the world. 

I thank  the  employees and researchers in the Weierstrass Institute who are always ready to help with organizational working  issues. Especially,  I thank Pavel Dvurechensky for his helpful advice and rewarding meeting that we have together. Many thanks to Franz Besold for his help with teaching assistants in the statistical seminars at Humboldt-Universität zu Berlin.

I am also very grateful to Alexander Gasnikov for his fruitful ideas and explanations  which had a great influence on the content of this thesis.

I also want to thank warmly
all the researches,  I worked with, for the opportunity to carry out researches jointly. In particular, thanks to Angelia Nedi\'{c},  C\'{e}sar A. Uribe, Daniil Tiapkin, Eduard Gorbunov, Alexander Rogozin.

The research of Chapter \ref{ch:population}  was 
supported by the Russian Science Foundation (project 18-71-10108), \url{https://rscf.ru/project/18-71-10108/}. 
The research of Chapter \ref{ch:WB}  was 
supported by the Ministry of Science and Higher Education of the Russian Federation (Goszadaniye) No. 075-00337-20-03, project No. 0714-2020-0005. 
The research of Chapter \ref{ch:decentralized} was funded by RFBR 19-31-51001.

\tableofcontents
\listoffigures
\listoftables

\chapter*{Notations}
\begin{itemize}
    \item $\Delta_n  = \{ a \in \mathbb{R}_+^n  \mid \sum_{l=1}^n a_l =1 \}$ is the probability simplex.
    \item  $I_{n\times n}$ is the identity matrix of size $n\times n$. 
    \item  $0_{n\times n}$ is zeros matrix of size $n\times n$. 
    \item  $\boldsymbol{1}_{n}$ is the vector of ones of size $n$. 
    \item $[n]$ is the sequence of integer number from 1 to $n$.
    \item Capital symbols, e.g., $A,B$, are used for matrices.
    \item Bold capital  symbols, e.g., $\mathbf A, \mathbf B$, are used  for block-matrices.
    \item Bold small symbol, e.g., $\x = (x_1^\top,\cdots,x_m^\top)^\top \in \mathbb{R}^{mn}$ is the column vector of vectors $x_1,...,x_m\in \R^n$.
    \item We refer to the $i$-th component of  vector $\x$ as $ x_i\in \R^n$.
    \item $[x]_j$ is $j$-th component of  vector $x$.
    \item $\la \cdot~, \cdot \ra$ is the usual Euclidean dot-product between vectors. For
two matrices of the same size $A$ and $B$, $\la A, B\ra =  \rm tr(A B)$ is the Frobenius dot-product.
\item   $ \|s\|_{*} = \max_{x\in X} \{ \la x,s \ra : \|x\|\leq 1 \} $ is the dual norm  for some norm $\|x\|$, $x\in X$. In particular, 
for the $\ell_p$-norm, its dual norm is $\ell_q$-norm, where $\frac{1}{p}+\frac{1}{q}=1$.
\item For two vectors $x,y$ (or matrices $A,B$ ) of the same size, $x/y$ ($A/B$) and $x \odot y$ ($A \odot B$)  stand for the element-wise product and  element-wise division respectively. When used on vectors, functions such as $\log$ or $\exp$ are always applied element-wise.
\item For prox-function $d(x)$, the corresponding  Bregman divergence is $B(x, y) = d(x) -d(y) - \la  \nabla d(y), x - y \ra$.
\item $\lm_{\max}(W)$ is  the maximum eigenvalue of a symmetric matrix $W$
\item $\lm^+_{\min}(W)$ is  the  minimal
non-zero eigenvalue of a symmetric matrix $W$
\item $\chi(W) = \frac{\lm_{\max}(W)}{\lm_{\min}^+(W)}$ is   the condition number
of matrix $W$
\item  $ O(\cdot)$  is the notation for an upper bound on the growth rate  hiding constants.
\item $\widetilde O(\cdot)$  is the notation for an upper bound on the growth rate hiding   logarithms.
\end{itemize}

\chapter*{Mathematical Preliminaries}

\begin{definition}[$M$-Lipschitz]
A function $f:X \times \Xi \rightarrow \R$ is $M$-Lipschitz continious with respect to $x\in X$ in  norm $\|\cdot\|$ if it satisfies
\begin{equation}\label{def:MLipsch}
{|}f(x,\xi)-f(y,\xi){|}\leq M\|x-y\|, \qquad \forall x,y \in X,~ \forall \xi \in\Xi. \end{equation}
\end{definition}
From Eq. \eqref{def:MLipsch} it follows that
\[
\|\nabla_x f(x,\xi)\|_{*} \leq M,  \qquad \forall x \in X,~ \forall \xi \in \Xi,
\]
where $\nabla_x f(x,\xi)$ is a subgradient of $f(x,\xi)$ with respect to $x$ \cite{shapiro2014lectures}.

\begin{definition}[$L$-smoothness]
A function $f: X \times \Xi \rightarrow \R$ is $L$-Lipschitz smooth, or has  $L$-Lipschitz continuous gradient, with respect to norm $\|\cdot\|_{X}$ if $f(x,\xi)$ is continuously differentiable with respect to $x$ and its gradient satisfies Lipschitz condition 
\begin{equation}\label{def:LipschSmooth}
\|\nabla_x f(x,\xi) - \nabla_y f(y,\xi) \|_{*} \leq L \|x-y\|, \quad \forall x,y \in X, ~ \forall \xi \in \Xi.
\end{equation}
\end{definition}
From  Eq. \eqref{def:LipschSmooth} it follows that
\begin{equation}
f(y,\xi) \leq f(x,\xi) + \la \nabla_x f(x,\xi) , y-x \ra + \frac{L}{2} \|x-y\|^2, \quad \forall x,y \in X, ~\forall
 \xi \in \Xi.
\label{eq:nfLipDef}
\end{equation}

\begin{definition}[$\gamma$-strong convexity]
A function $f:X\times \Xi \rightarrow \R$ is $\gamma$-strongly convex with respect to $x$ in  norm $\|\cdot\|_X$ if it is continuously differential and it satisfies
\[f(x, \xi)-f(y, \xi)- \la\nabla f(y, \xi), x-y\ra\geq \frac{\gamma}{2}\|x-y\|^2, \qquad \forall x,y \in  X,~ \forall \xi \in \Xi.\]
\end{definition}
\begin{definition}[Dual Function]
The Fenchel--Legendre conjugate for a function $f:(X, \Xi) \rightarrow \R$ is \[f^*(u,\xi) \triangleq \max_{x \in X}\{\la x,u\ra - f(x,\xi)\}, \qquad  \forall \xi \in \Xi.\]
\end{definition}

\begin{theorem}\citep[Theorem 6 (Strong/Smooth Duality)]{kakade2009duality}\label{th:primal-dual}
 Assume that  $f$ is a closed and convex function on $X=\R^n$. Then $f$ is $\gamma$-strongly convex w.r.t. a norm $\|\cdot\|_X$ if and only if $f^*$ is $\frac{1}{\gamma}$--Lipschitz smooth
w.r.t. the dual norm $\|\cdot\|_{X^*}$.
\end{theorem}

\begin{theorem}\citep[Theorem 1]{nesterov2005smooth}\label{th:primal-dualNes}
  Assume that function $f(x)$ is continuous and $\gamma$-strongly convex w.r.t. a norm $\|\cdot\|$. Then $\vp(u) = \max\limits_{x \in X}\{\la Ax,u\ra - f(x)\}$ is $\frac{\lm_{\max}(A^\top A)}{\gamma}$--Lipschitz smooth
w.r.t. the dual norm $\|\cdot\|_{*}$.
\end{theorem}

\chapter{Introduction}

\section{Background on Optimal Transport}
Optimal transport problem  is closely related to the notion of \textit{linear programming}.
{Linear programming} (LP) is  the science of theoretical and numerical analysis and solving extremal (e.g., maximization or maximization)  problems defined by systems of linear equations and inequalities.
A lot of mathematicians made contributions to the development of linear programming, including T. Koopmans, G.B. Danzig (a founder of the simplex method, 1949) and I.I. Dikin (a  founder of the interior points method, 1967), but the  priority belongs to  the Soviet mathematician and economist L. V. Kantorovich \cite{kantorovich1960mathematical},  who was the first who discovered that a wide class of the most important production problems can be described mathematically and solved numerically (1939).

Particular and important cases of linear programming problems are network flow problem,  multicommodity flow problem, and \textit{optimal transport} (OT) problem.  The history of optimal transport begins with the French mathematician  G. Monge \cite{monge1781memoire},
who proposed a complicated theory of describing    an optimal mass transportation in a geometric way. 
Inspired by the problem of  resource allocation, L.V. Kantorovich
introduced relaxations which  allowed him to formulate the transport problem as
linear programming problem, and as a consequence, to apply linear programming methods to solve it. The main relaxation was based on the refusing of deterministic nature of transportation (a mass from the source point could only be transferred to one target point) and introducing  a probabilistic transport.
To do so, a coupling matrix was introduced instead of Monge maps. Admissible couplings  (also known as transportation polytope) of all coupling matrices with marginals discrete source $\mu$ and discrete target $\nu $   can be written as follows
\[U(\mu,\nu) \triangleq\{ \pi\in \R^{n_2\times n_1}_+: \pi \one_{n_{1}} =\mu, \pi^T \one_{n_2} = \nu\}.\]
Here $\pi$ is a coupling (transport plan) ($\pi_{ij}$ describes the amount of mass moving 
from  source bin $i$  towards target bin $j$) 
 Thus, the problem of optimal transport between $\mu$ and  $\nu$ under a symmetric transportation cost matrix $C  \in \R^{n\times n}_+$, called also as the Monge--Kantorovich problem,  is formulated as follows
\begin{equation}\label{def:optimaltransport}
    \min_{\pi \in U(\mu,\nu)} \la C, \pi \ra.
\end{equation}

Moreover, Kantorovich formulated an infinite-dimension analog of optimal transport problem \eqref{def:optimaltransport} between probability measures $\mu \in \mathcal P(X)$ and $\nu \in \mathcal P (Y)$  under  transportation cost function $c(x,y)$ 
\[
\min_{\pi \in \U(\mu,\nu)} \int_{\X\times\Y} c(x,y)d\pi(x,y),
\]
where \[\U(\mu,\nu) \triangleq \{ \pi\in \mathcal  P(\X\times\Y): T_{\X\#} =\mu, T_{\Y\#} = \nu\}.\]
Here   $T_{\X\#}$ and $T_{\Y\#}$  are the push-forwards.
Furthermore, the replacement of Monge's maps by  couplings  and infinite-dimension formulation of optimal transport  allowed Kantorovich and G. S. Rubinstein  to introduce Kantorovich--Rubinstein distance in the space of probability measures. Nowadays, it is often referred to as   Wasserstein distance.  Namely, $\rho$-Wasserstein distance ($\rho\geq 1$)  between probability measures $\mu, \nu \in \mathcal P (X)$   is defined as follows
\begin{equation}\label{def:p-Wassersteindistance}
    \W_\rho(\mu,\nu) \triangleq\left( \min_{\pi \in \U(\mu,\nu)} \int_{\X\times\X} \mathtt d(x,y)^\rho d\pi(x,y) \right)^{1/\rho},
\end{equation}
where it was assumed that $\X=\Y$ and  $c(x,y) =\mathtt d(x,y)^\rho$ is a distance on $\X$.

For multivariate Gaussian measures,
the  2-Wasserstein distance  has a closed-form solution and
  is expressed through Bures metric \cite{bures1969extension} which is used  to compare quantum states in quantum physics.

Nowadays, optimal transport metric  provides a successful framework to compare objects that can be modeled  as probability measures (images, videos, texts and etc.).  Transport based distances, especially  1-Wasserstein distance (EMD), have gained popularity in various fields such as statistics \cite{ebert2017construction,bigot2012consistent}, unsupervised learning \cite{arjovsky2017wasserstein}, signal and image analysis \cite{thorpe2017transportation},  computer vision \cite{rubner1998metric}, text classification \cite{kusner2015word}, economics and finance \cite{rachev2011probability} and medical imaging \cite{wang2010optimal,gramfort2015fast}. 
A lot of 
 statistical results are known about optimal transport (Wasserstein) distances \cite{sommerfeld2018inference,weed2019sharp,klatt2020empirical}.

\section{Background on  Wasserstein Barycenters}
The  success of optimal transport  led to an increasing interest in  \textit{Wasserstein barycenters}.
In \cite{agueh2011barycenters}, the notion of {a Wasserstein barycenter} was introduced in   the Wasserstein space (space $\mathcal P_2 (\X)$ of probability measures with finite second moment supported on a convex domain $\X$) similarly to the barycenter of points in the Euclidean space by replacing  the squared Euclidean distance with the squared 2-Wasserstein distance. Namely, a Wasserstein barycenter of a set of probability measures $\nu_1,\nu_2,...,\nu_m$ is defined as follows 
\begin{equation}\label{def:Wassersteinbarycenter}
    \min_{\mu \in \mathcal P_2 (\X)} \sum_{i=1}^m \lm_i \W_2^2(\mu, \nu_i),
\end{equation}
where the $\lm_i$’s are positive weights summing to 1.

Wasserstein barycenters are used in Bayesian computations \cite{srivastava2015wasp}, texture mixing \cite{rabin2011wasserstein}, clustering  ($k$-means for probability measures) \cite{del2019robust}, shape interpolation and color transferring \cite{Solomon2015},  statistical estimation of template models \cite{boissard2015distribution} and neuroimaging \cite{gramfort2015fast}.

\section{Background on Population Wasserstein barycenter}

For random probability measures with distribution $\PP$ supported on $P_2 (\X)$, population Wasserstein barycenter is introduced through a notion of a Fr\'{e}chet mean \cite{frechet1948elements}
\begin{equation}\label{def:Freche_population}
  \min_{p\in  P_2 (\X)}\E_{q \sim \PP} W(p,q)= \min_{p\in  P_2 (\X)} \int_{P_2 (\X)} W(p,q)d\PP(q).
\end{equation}
For identically distributed measures, problem \eqref{def:Wassersteinbarycenter} can be interpreted as an  empirical counterpart of problem \eqref{def:Freche_population}.
If a solution of \eqref{def:Freche_population} exists and is unique, then it is referred to as the population barycenter of distribution $\PP$.

 \section{Overview of the Thesis}

In this thesis, we consider the Wasserstein barycenter problem of discrete probability measures as well as the population Wasserstein barycenter problem  given by a Fr\'{e}chet mean.
The main focus of this thesis is computational 
aspect of  the Wasserstein barycenter problem:  
deriving first-order methods to compute Wasserstein barycenters. Dual first-order methods  rely on the fact that
regularized  optimal transport 
by negative entropy   with $\gamma>0$, that is 
\[W_\gamma(p,q) = \min_{\pi \in U(p,q)} \left\{\la C, \pi \ra + \gamma \la \pi,\log \pi \ra  \right\},\]
has a dual closed-form representation defined by the Fenchel--Legendre transform w.r.t. $p \in \Delta_n$
\cite{agueh2011barycenters,cuturi2016smoothed}:
\begin{align}\label{eq:dradWW}
     W_{\gamma, q}^*(u) &=  \max_{ p \in \Delta_n}\left\{ \la u, p \ra - W_\gamma(p,q) \right\} \notag \\
    &= \gamma\left(-\la q,\log q\ra + \sum_{j=1}^n [q]_j \log\left( \sum_{i=1}^n \exp\left(([u]_i - C_{ji})/\gamma\right) \right)\right),
\end{align}
  where  $[q]_j$ and $[u]_i$ are the $j$-th and $i$-th components   of  $q$ and $u$ respectively, and $C_{ji}$ is the entry of matrix $C$.
The gradient of dual function $W_{\gamma, q}^*(u)$ is Lipschitz continuous and has also a closed-form solution
 \begin{align}\label{eq:primal_sol_recov3323}
 [\nabla W^*_{\gamma,q} (u)]_l = \sum_{j=1}^n [q]_j \frac{\exp\left(([u]_l-C_{lj})/\gamma\right)  }{\sum_{\ell=1}^n\exp\left(([u]_\ell-C_{\ell j})/\gamma\right)},
\end{align}
for all $ l =1,...,n$.

A saddle point approach for Wasserstein barycenter problem relies on the fact that non-regularized optimal transport \eqref{def:optimaltransport} has  a bilinear  saddle-point representation 
 \cite{jambulapati2019direct}:
\begin{equation*}
  W(p,q) =  \min_{x \in \Delta_{n^2}} \max_{y\in [-1,1]^{2n}} \left\{ \la d, x \ra +2\|d\|_\infty\left(~ y^\top Ax - \left\la \begin{pmatrix}
  p\\
  q
\end{pmatrix}, y \right\ra\right)\right\}.
\end{equation*}
Here $d$ is the vectorized cost matrix $C$, $x \in \Delta_{n^2}$ is the vectorized transport plan $\pi$, and \[A \triangleq  \begin{pmatrix}
    I_{n\times n} &\otimes &\boldsymbol{1}^\top_{n}\\
     \boldsymbol{1}^\top_{n} &\otimes &I_{n\times n}
 \end{pmatrix} = \{0,1\}^{2n\times n^2} \] is the incidence matrix. 

Decentralized formulations of the Wasserstein barycenter problem 
both for the saddle-point and dual representations 
are based on 
introducing artificial constraint
$p_1=p_2= ...=p_m \in \R^n$ which is further replaced with affine constraint  $ \WW \p=0$ (in the saddle-point approach)  and $\sqrt \WW \p = 0$ (in the dual approach), where 
$\p = (p_1^\top, ..., p_m^\top)^\top$ is column vector and 
$\WW$ is referred as the communication matrix for a decentralized system. From the definition of matrix $\WW$ it follows that 
\[\sqrt \WW \p=0 \Longleftrightarrow  \WW \p = 0 
\Longleftrightarrow p_1 = p_2 = ... = p_m.\] 
The affine constraint $ \WW \p=0$ (or $\sqrt \WW \p = 0$) is brought to the objective via the Fenchel--Legendre transform.
Thus, for the primal Wasserstein barycenter problem defined w.r.t. entropy-regularized optimal transport 
\begin{equation*}
  \min_{p\in \Delta_n} \frac{1}{m} \sum_{i=1}^m {W}_\gamma(p,q_i)=\min_{\substack{p_1=...=p_m, \\p_1,...,p_m \in \Delta_n} } \frac{1}{m} \sum_{i=1}^m {W}_\gamma(p_i,q_i) = \min_{\substack{\sqrt{\WW} \p=0, \\p_1,...,p_m \in \Delta_n} } \frac{1}{m} \sum_{i=1}^m {W}_\gamma(p_i,q_i), 
\end{equation*}
we can construct  the corresponding   dual Wasserstein barycenter problem:
\begin{equation}\label{eq:W_bary_regdual222244}
  \min_{\y \in \R^{nm}} {W}^*_{\gamma, \q}(\sqrt{\WW}\y)\triangleq  \frac{1}{m}\sum_{i=1}^{m} {W}^*_{\gamma, q_i}(m[\sqrt{\WW}\y]_i),
\end{equation}
where $\q = (q_1^\top, \cdots, q_m^\top)^\top$, and $\y = (y_1^\top, \cdots, y_m^\top)^\top \in \R^{nm}$ is the Lagrangian dual multiplier. 
As the primal function is strongly convex, then the dual function is  $L$-Lipschitz smooth, or has Lipschitz continuous gradient. The constant $L$ for ${W}^*_{\gamma, \q}(\sqrt{\WW}\y)$ is defined via communication matrix $\WW$ and regularization  parameter $\gamma$. Hence, accelerated gradient descent-based method  can be used, which is optimal in terms of the number of iterations and oracle calls. 
For simplicity, the decentralized procedure  solving dual problem \eqref{eq:W_bary_regdual222244} can be demonstrated on  the gradient descent as follows
\begin{equation*}
\y^{k+1}= \y^{k}- \frac{1}{L} \nabla{W}^*_{\gamma, \q}(\sqrt{\WW}\y^k) = \y^{k} -\frac{1}{L}\sqrt{\WW}\p(\sqrt{\WW}\y^k). \end{equation*}
Without change of variable, it is unclear how to execute this  procedure in a distributed fashion. Let  $\u := \sqrt{\WW}\y $, then the gradient step  multiplied by $\sqrt \WW$ can be rewritten as 
\[\u^{k+1} = \u^{k} -\frac{1}{L}\WW\p(\u^{k}),\]
 where $[\p(\boldsymbol u)]_i = p_i(u_i) = \nabla W^*_{\gamma,q_i}(u_i)$ from \eqref{eq:primal_sol_recov3323}, $i=1,...,m$.
This procedure can be performed in a decentralized manner on a distributed network.
The vector $\WW\p(\boldsymbol u)$ naturally defines communications with neighboring nodes due to the structure of communication matrix $\WW$ as  the elements of communication matrix are zero for non-neighboring nodes. 
Moreover, 
in the dual approach which is based on gradient method, the randomization of $\nabla W_{\gamma, q_i}^*(u_i)$ can be used  to reduce the complexity of calculating the true gradient, that is $O(n^2)$ arithmetic operations, by calculating its stochastic approximation of $O(n)$ arithmetic operations. 
The randomization for the true gradient \eqref{eq:primal_sol_recov3323} is achieved   by taking the $j$-th term in the sum with probability $[q]_j$ 
\[
[\nabla W_{\gamma,q}^*(u,\xi)]_l =  \frac{\exp\left(([u]_l-C_{l \xi})/\gamma\right)  }{\sum_{\ell=1}^n\exp\left(([u]_\ell-C_{\ell\xi })/\gamma\right)}, \qquad \forall l =1,...,n.
\]
where we replaced index $j$ by $\xi$ to underline its randomness.
This is the motivation for considering the first-order methods with stochastic oracle.  


For greater generality, we derive the methods for a general convex minimization problem where the objective is given by the  sum of functions, and for a general stochastic convex minimization problem where the objective is given by its  expectation. 
These two problems are  generalizations of problems \eqref{def:Wassersteinbarycenter} and \eqref{def:Freche_population}. 
The reason for this generality is obtaining the results of other interests than Wasserstein barycenter problem.

Thus, we consider a   general stochastic convex optimization problem whose  objective is given by its expectation
(problem
\eqref{def:Freche_population} is  a particular case of this problem)
\begin{equation}\label{eq:E_gener_risk_min_prob}
\min_{x\in X \subseteq \mathbb{R}^n} F(x) \triangleq \E f(x,\xi),  
\end{equation}
where $\E f(x,\xi)$ is the expectation with respect to random variable  $\xi$ from  set $\Xi$, $f(x,\xi)$ is convex  in $x$ on  convex set $X$. 
 Such kind of problems arise in many machine learning applications \cite{shalev2014understanding} (e.g., empirical risk minimization) and statistical applications \cite{spokoiny2012parametric} (e.g., maximum likelihood estimation). We will say that an output $x^N$ of an algorithm  is an $\e$-solution of problem \eqref{eq:E_gener_risk_min_prob} 
if
the following holds with  probability at least $1-\beta$
\[ F(x^N) - \min\limits_{x\in X } F(x)\leq \e.\]

 The complexity of an algorithm  is measured by the number of iterations and the number of  oracle calls.
We consider  the  (stochastic) first-order oracle, i.e., the oracle which
for a given  realization $\xi \in \Xi$, returns  the gradient (subgradient) of $f(x,\xi)$ calculated with respect to  $x \in X$. For the dual first-order methods, we use the dual (stochastic) first-order oracle returning  the gradient of the dual to $f(x,\xi)$ function given by the Fenchel--Legendre transform of $f(x,\xi)$.

 We also  consider a   general convex optimization problem whose objective is given by the sum of convex functions (problem \eqref{def:Wassersteinbarycenter} is  a particular case of this problem) 
  \begin{equation}\label{eq:primal_gener_sum_type}
      \min_{x\in X \subseteq \mathbb{R}^n}  f(x) \triangleq \frac{1}{m}\sum_{i=1}^m f_i(x).
  \end{equation}





Problems of type \eqref{eq:primal_gener_sum_type} can be effectively solved in a distributed  manner on a computational network. 
In the last decade, distributed optimization  became especially popular with the release of the book \cite{bertsekas1989parallel} and due to the emergence of big data and rapid growth of problem sizes.
The idea of distributed calculations is simple:
 every node (computational unit of  some connected undirected graph (network)), assigned  by its private function $f_i$, calculates the gradient of the private function and simultaneously communicates with its neighbors by exchanging  messages  at each communication round. 

For primal approaches, the lower and upper bounds on communications rounds and (stochastic) primal  oracle calls of $\nabla f_i$ per node $i$ are known, as well as the methods  matching these lower  bounds.  We refer to works \cite{scaman2017optimal,li2018sharp,uribe2017optimal} describing these bounds for  Lipschitz smooth deterministic objective. For non-smooth (deterministic and stochastic) objective,  we 
appeal to \cite{lan2017communication,scaman2018optimal}. In the stochastic Lipschitz smooth case, the optimal bound on the number of communication rounds was obtained in \cite{dvinskikh2021decentralized}, the optimal bound on the number of stochastic oracle calls was gained in \cite{rogozin2021accelerated}. Tables \ref{tab:DetPrimeOrIntro} and \ref{T:stoch_prima_oracleIntro} summarize the results for deterministic and stochastic primal oracles respectively.  In these tables, factor $\widetilde{O}(\sqrt\chi)$  is responsible for the consensus time, i.e.,  the number of communication rounds required to reach the consensus in the considered network; $\sigma^2$ and $\sigma^2_\psi$
are the sub-Gaussian variance for $\nabla f_i(x_i,\xi_i)$ and $\nabla \psi_i(\lm_i,\xi_i)$ respectively, where $\nabla \psi_i(\lm_i,\xi_i)$ is the dual function to $\nabla f_i(x_i,\xi_i)$ with respect to $x_i$.

	\begin{table}[ht!]
\caption {Optimal bounds on the number of  communication rounds and deterministic oracle calls of  $\nabla f_i(x_i)$ per node}
\label{tab:DetPrimeOrIntro}
\begin{adjustbox}{max width=\textwidth}
\begin{tabular}{lllll}
\toprule
 Property of $f_i$  & \makecell{  $\mu$-strongly convex,\\  $L$-smooth} 
 &   \makecell[l]{ $L$-smooth} 
 & \makecell[l]{  $\mu$-strongly convex,\\ $M$-Lipschitz} &   $M$-Lipschitz\\
\midrule
 \makecell[l]{Number of \\
 communication \\ rounds} 
 & \makecell[l]{$\widetilde O\left(\sqrt{\frac{L}{\mu}\chi} \right)$}   
 & \makecell[l]{ $ \widetilde O\left({\sqrt{\frac{LR^2}{\e} \chi}}\right)$ }
 &  \makecell[l]{  $O\left(\sqrt{\frac{M^2}{\mu\e}\chi} \right)$} 
 &  \makecell[l]{  $O\left( \sqrt{\frac{M^2R^2}{\e^2}\chi} \right) $}   \\
 \midrule
\makecell[l]{Number of \\ oracle calls of\\  $\nabla f_i(x_i)$  per node $i$} 
& $ \widetilde O\left(\sqrt{\frac{L}{\mu}} \right)$ 
&  $O\left(\sqrt{\frac{LR^2}{\e}}  \right) $ 
&   $O\left(\frac{M^2}{\mu\e} \right) $ 
&   $O\left( \frac{M^2R^2}{\e^2}\right)$\\
 \bottomrule
\end{tabular}
\end{adjustbox}
\end{table}


	\begin{table}[ht!]
\caption {Optimal bounds on the number of  communication rounds and stochastic oracle calls of    $\nabla f_i(x_i,\xi_i)$ per node} 
\label{T:stoch_prima_oracleIntro}
\begin{adjustbox}{max width=\textwidth}
\begin{tabular}{ lllll}
 \toprule
Property of $f_i$  &{\makecell{  
$\mu${-strongly convex},\\  $L$-smooth} } &  {\makecell[l]{$L$-smooth} }&  \makecell[l]{  
$\mu$-strongly convex, \\ $\E\|\nabla f_i(x_i,\xi_i)\|_2^2\le M^2$ } & $\E\|\nabla f_i(x_i,\xi_i)\|_2^2\le M^2$\\
 \midrule
\makecell[l]{Number of \\ communication \\ rounds}& \makecell[l]{ $\widetilde O\left(\sqrt{\frac{L}{\mu}\chi} \right)$}   & \makecell[l]{ $ \widetilde O\left({\sqrt{\frac{LR^2}{\e} \chi}}\right)$ }&  \makecell[l]{  $O\left(\sqrt{\frac{M^2}{\mu\e}\chi} \right)$} &  \makecell[l]{  $O\left( \sqrt{\frac{M^2R^2}{\e^2}\chi} \right) $}  \\
 \midrule
  \makecell[l]{Number of \\ oracle calls of \\ $\nabla f_i(x_i, \xi_i)$\\
 per node $i$} & {\makecell[l]{ {$\widetilde O\left(\max\left\{\frac{\sigma^2}{m\mu\e},
 \sqrt{\frac{L}{\mu}}\right\}\right)$} }}  & {\makecell[l]{{$ O\left(\max\left\{\frac{\sigma^2R^2}{m\e^2}, \sqrt{\frac{LR^2}{\e}}  \right\}\right)$}}} & { {\makecell[l]{  $O\left(\frac{M^2}{\mu\e} \right)$}}} &{\makecell[l]{  $O\left(\frac{M^2R^2}{\e^2}\right)$}} \\
 \bottomrule
\end{tabular}
\end{adjustbox}
\end{table}


For deterministic dual  oracle, the bounds are also known: \citet{scaman2017optimal} provided the results for strongly convex and smooth primal objective, the bounds for non-smooth but strongly convex primal objective were obtained in \cite{uribe2018distributed,uribe2020dual}.   Stochastic dual  oracle
was not actively studied and optimal bounds on the number of stochastic dual  oracle calls were not obtained.
We leverage this gap and  derive  primal-dual decentralized algorithms which are optimal in terms of  the number of dual (stochastic) oracle calls and the number of communication rounds.  
Table \ref{tab:distrDetCOm22}  summarizes the results for deterministic dual oracle. Table \ref{T:dist_stoch22} demonstrates one of the contributions of this thesis: optimal bounds for stochastic dual oracle. 
The case of non-smooth but strongly convex primal objective in Table \ref{T:dist_stoch22} corresponds to the Wasserstein barycenter problem defined with respect to entropy-regularized optimal transport. 
This is one of the motivation to consider the dual oracle since the  dual representation (the Fenchel--Legendre transformation) of the entropy-regularized optimal transport and its derivatives  can be presented in closed-forms.

    	\begin{table}[ht!]
\caption {The optimal bounds for dual deterministic oracle}
\label{tab:distrDetCOm22}
\begin{center}
\begin{adjustbox}{max width=\textwidth}
\begin{tabular}{ lll }
 \toprule
Property of $f_i$ & \makecell[l]{  $\mu${-strongly convex},\\  $L$-smooth} 
&  \makecell[l]{  $\mu$-strongly convex, \\ $\|\nabla f_i(x^*)\|_2 \leq M$ }  \\
 \midrule
\makecell[l]{The number of \\ communication  \\ rounds} & \makecell[l]{ $\widetilde O\left(\sqrt{\frac{L}{\mu}\chi(W)} \right)$ } &  \makecell[l]{ $O\left(\sqrt{\frac{M^2}{\mu\e}\chi(W)}\right) $ }  \\
 \midrule
\makecell[l]{  The number of \\ oracle calls of\\ $\nabla \psi_i(\lambda_i)$
 per node $i$} & \makecell[l]{ $\widetilde O\left(\sqrt{\frac{L}{\mu}\chi(W)} \right)$ } &  \makecell[l]{ $O\left(\sqrt{\frac{M^2}{\mu\e}\chi(W)}\right) $ }  \\
 \bottomrule
\end{tabular}
\end{adjustbox}
\end{center}
\end{table}

	\begin{table}[ht!]
\caption{The optimal bounds for dual stochastic (unbiased) oracle}
\label{T:dist_stoch22}
\begin{adjustbox}{max width=\textwidth}
\begin{tabular}{ lll }
 \toprule
Property of $f_i$ &{\makecell[l]{  $\mu${-strongly convex},\\  $L$-smooth} }  & \makecell[l]{  $\mu$-strongly convex,\\ $\|\nabla f_i(x^*)\|_2\le M$ } \\
 \midrule
\makecell[l]{The number of \\ communication \\ rounds} & \makecell[l]{ $\widetilde O\left(\sqrt{\frac{L}{\mu}\chi(W)} \right)$ }  & \makecell[l]{ $O\left(\sqrt{\frac{\ga{M^2}}{\mu\e}\chi(W)}\right) $ } \\
 \midrule
\makecell[l]{The number of \\ oracle calls of \\
$\nabla \psi_i(\lambda_i,\xi_i)$\\
 per node $i$}  &  \makecell[l]{$\widetilde O\left(\max\left\{{\frac{M^2\sigma_{\ga{\psi}}^2}{\e^2}\chi(W)}, \sqrt{\frac{L}{\mu}\chi(W)} \right\}\right)$ }  & \makecell[l]{ $O\left(\max\left\{ \frac{M^2\sigma_{\psi}^2}{\e^2}\chi(W), \sqrt{\frac{ M^2}{\mu\e}\chi(W)} \right\}\right)$ }   \\
 \bottomrule
\end{tabular}
\end{adjustbox}

\end{table}

\subsection{Thesis Structure}

The dissertation consists of 5 Chapters:

 In Chapter \ref{ch:population}, 
we study the two main approaches in  machine learning and optimization community for convex risk minimization problem, namely, the Stochastic Approximation (SA) and the Sample Average Approximation (SAA) also known as the Monte Carlo approach.
In terms of the oracle complexity (required number of stochastic gradient evaluations), both approaches are considered equivalent on average (up to a logarithmic factor). The total complexity depends on the specific problem, however, starting from work \cite{nemirovski2009robust} it was generally accepted that the SA is better than the SAA. 
 We show that for the Wasserstein barycenter problem, this superiority can be swapped. We provide the detailed comparison with stating the complexity bounds for  the SA and the SAA implementations calculating Fr\'{e}chet  mean defined with respect to optimal transport distances and Fr\'{e}chet  mean defined with respect to entropy-regularized optimal transport distances. 
 As a byproduct, we also construct  confidence intervals for population barycenter defined with respect to  entropy-regularized optimal transport distances in the $\ell_2$-norm. Here we propose a new regularization for the  the SAA approach which contributs to a better convergence rate in comparison with the quadratic regularization.
The preliminary results were derived for a general convex optimization problem given by the expectation 
so that they can be applied to a wider range of problems other than the Wasserstein barycenter problem.

In Chapter \ref{ch:dual},
we introduce a decentralized dual algorithm to minimize the sum of  strongly convex functions on a network of agents (nodes). This algorithm is based on accelerated gradient descent and it allows to obtain   optimal bounds on the number of communication rounds and oracle calls of dual objective per node.
The results can be naturally applied for the Wasserstein barycenter problem as
 the dual formulation of entropy-regularized Wasserstein distances 
 and their derivatives have closed-form representations. 

In Chapter \ref{ch:WB}, we provide saddle point approach    to compute unregularized Wasserstein barycenters  with no limitations in contrast to the regularized-based methods, which are numerically unstable under a small value of the regularization parameter. The approach is based on the saddle-point problem reformulation and the application of mirror prox algorithm with a specific norm. We also show how the algorithm can be executed in a decentralized manner.  The complexity of the proposed methods meets the best known results in the decentralized and non-decentralized  setting.

Chapter  \ref{ch:decentralized} has interests other than Wasserstein barycenters. The purpose of this Chapter is obtaining the optimal  bounds on the number of communication rounds and oracle calls for the gradient of the dual objective per node in the problem of minimizing the sum of strongly convex  functions with  Lipschitz continuous gradients. Thus, this Chapter complements Chapter \ref{ch:dual} for the case of additionally Lipschitz smooth (stochastic) objectives.

\section{Main Contributions}

\begin{itemize}

    \item \textit{Statistical issue:} statistical study of the Wasserstein barycenter problem 
    \begin{enumerate}[label=(\alph*)]
        \item Estimating the sample size of measures needed to calculate an approximation for a Fr\'{e}chet mean (barycenter) of a probability distribution with a given precision
        \item Proposing a new regularization for risk minimization approach (also known as the SAA approach) which contributes to  better convergence rate in comparison with  quadratic regularization
    \end{enumerate}
    \item \textit{Computational issue:} proposing  decentralized (stochastic) algorithms with optimal convergence rates
    \begin{enumerate}[label=(\alph*)]
        \item Obtaining  optimal bounds on the number of communication rounds and dual oracle calls for the gradient of the dual (stochastic) objective per node in decentralized optimization for minimizing the  sum of  strongly convex functions, possibly with  Lipschitz continuous gradients
        \item Developing  decentralized  algorithms with the best known bounds for the problem of calculating Wasserstein barycenters of a set of discrete measures
    \end{enumerate}

\end{itemize}

\section{Bibliographic Notes}
The contribution of this thesis is based on  the following papers.
\begin{itemize}
    \item Chapter  \ref{ch:population} is based on the work \cite{dvinskikh2020sa} accepted to the journal `Optimization Methods and Software'
    
\item Chapter  \ref{ch:dual} is partially based on the  results of joint paper with  Eduard Gorbunov, Alexander Gasnikov,  Pavel Dvurechensky and C{\'e}sar A. Uribe
\cite{dvinskikh2019primal}  published in the proceedings of the 58th Conference on Decision and Control (CDC, 2019 IEEE), on a part of the  results of joint paper with   Pavel Dvurechensky, Alexander Gasnikov, Angelia Nedi\'{c} and C{\'e}sar A. Uribe
\cite{dvurechensky2018decentralize}  published in the proceedings of the 32nd Conference on Neural Information Processing Systems (NeurIPS 2018), and on a part of  the  results of joint paper with  Alexey Kroshnin, Nazarii Tupitsa,  Pavel Dvurechensky, Alexander Gasnikov and C{\'e}sar A. Uribe  
\cite{kroshnin2019complexity}  published in the proceedings of the 36th International Conference on Machine Learning

\item Chapter \ref{ch:WB} partially uses  the results from joint paper with Daniil Tiapkin  
\cite{dvinskikh2020improved} published in the proceedings of the 24th International Conference on Artificial Intelligence and Statistics (AISTATS, 2021). Besides, this Chapter contains a part of the results from   arXiv preprint
\cite{rogozin2021decentralized} with Alexander Rogozin, Alexander Beznosikov,  Dmitry Kovalev, Pavel Dvurechensky and Alexander Gasnikov

\item The results of Chapter  \ref{ch:decentralized} are from   joint paper  with Alexander Gasnikov
\cite{dvinskikh2021decentralized} published in the Journal of Inverse and Ill-posed Problems, 2021

\end{itemize}


\chapter{Two Approaches: Stochastic Approximation (SA)  and  Sample Average Approximation (SAA).}\label{ch:population}
\chaptermark{The SA and the SAA Approaches}

 This Chapter is inspired by the work \cite{nemirovski2009robust} stated that the SA approach outperforms the SAA approach for  certain class of convex stochastic problems. We show that for the  Wasserstein barycenter problem, this superiority can be inverted. We provide detailed comparison with stating the complexity bounds for  the SA and the SAA implementations calculating Fr\'{e}chet mean  defined with respect to optimal transport distances and entropy-regularized optimal transport distances. 
The preliminary results are derived for a general convex optimization problem given by the expectation for  interest other than the Wasserstein barycenter problem.

\paragraph{Background on the SA and the SAA and Convergence Rates.}
We consider the stochastic convex minimization problem

\begin{equation}\label{eq:gener_risk_min}
\min_{x\in X \subseteq \mathbb{R}^n} F(x) \triangleq \E f(x,\xi),    
\end{equation}
where function $f$ is convex in $x$ ($x\in X,$ $X$ is a convex set), and $\E f(x,\xi)$ is the expectation of $f$ with respect to $\xi \in \Xi$.
Such kind of problems arise in many applications of data science 
\cite{shalev2014understanding,shapiro2014lectures} (e.g., risk minimization) and mathematical statistics \cite{spokoiny2012parametric} (e.g., maximum likelihood estimation). There are two competing  approaches based on Monte Carlo sampling techniques to solve \eqref{eq:gener_risk_min}: the Stochastic Approximation (SA) \cite{robbins1951stochastic} and the Sample Average Approximation (SAA).  
The SAA approach  replaces the objective in problem \eqref{eq:gener_risk_min} with its sample average approximation (SAA) problem
\begin{equation}\label{eq:empir_risk_min}
 \min_{x\in X} \hat{F}(x) \triangleq \frac{1}{m}\sum_{i=1}^m f(x,\xi_i), 
\end{equation}
where  $\xi_1, \xi_2,...,\xi_m$ are the realizations of a random variable  $\xi$. The number of realizations $m$ is adjusted by the desired precision.
The total working time of both approaches  to solve  problem \eqref{eq:gener_risk_min} with the average precision $\e$ in the non-optimality gap in term of the objective function (i.e., to find  $x^N$ such that $\E F(x^N) - \min\limits_{x\in X } F(x)\leq \e$),
depends on the specific problem. However, it was generally accepted \cite{nemirovski2009robust} that the SA approach is better than the SAA approach.
  Stochastic gradient (mirror) descent, an implementation of the SA approach \cite{juditsky2012first-order}, gives the following estimation for the number of iterations (that is equivalent to the sample size of $\xi_1, \xi_2, \xi_3,...,\xi_m$)
\begin{equation}\label{eq:SNSm}
m = O\left(\frac{M^2R^2}{\e^2}\right).
\end{equation}
Here we considered the minimal assumptions (non-smoothness) for the objective $f(x,\xi)$
\begin{equation}\label{M}
\|\nabla f(x,\xi)\|_2^2\le M^2, \quad \forall x \in X, \xi \in \Xi.
\end{equation}
Whereas, the application of the SAA approach requires  the following sample size \cite{shapiro2005complexity}
\[m = \widetilde{O}\left(\frac{n M^2R^2}{\e^2}\right),\]
that is $n$ times more ($n$ is the problem’s dimension) than the sample size in the SA approach. This estimate was obtained under the assumptions that problem \eqref{eq:empir_risk_min} is solved exactly. This is one of the main drawback of the SAA approach. However, if the objective $f(x,\xi)$ is $\lm$-strongly convex in $x$, the sample sizes are equal  up to logarithmic terms
\[
m = O\left(\frac{M^2}{\lm \e}\right).
\]
Moreover, in this case, for the SAA approach, it suffices to solve problem \eqref{eq:empir_risk_min}  with accuracy \cite{shalev2009stochastic}
 \begin{equation}\label{eq:aux_e_quad}
   \e' = O\left(\frac{\e^2\lm}{M^2}\right).
 \end{equation}
 Therefore, to eliminate the linear dependence on  $n$ in the SAA approach for a non-strongly convex objective, regularization $\lm =\frac{\e}{R^2}$ should be used \cite{shalev2009stochastic}.

Let us suppose that $f(x,\xi)$ in \eqref{eq:gener_risk_min} is convex but non-strongly convex in $x$ (possibly, $\lm$-strongly convex but with very small $\lm \ll \frac{\e}{R^2}$). Here  $R = \|x^1 - x^*\|_2$ is the Euclidean distance between starting point $x^1$ and the solution $x^*$ of \eqref{eq:gener_risk_min} which corresponds to the minimum of this norm (if the solution is not the only one). Then,  the problem \eqref{eq:gener_risk_min} can be replaced by 
\begin{equation}\label{eq:gener_risk_min_reg}
\min_{x\in X }  \E f(x,\xi) + \frac{\e}{2R^2}\|x - x^1\|_2^2.
\end{equation}
 The empirical counterpart of \eqref{eq:gener_risk_min_reg}  is 
\begin{equation}\label{eq:empir_risk_min_reg}
 \min_{x\in X} \frac{1}{m}\sum_{i=1}^m f(x,\xi_i) + \frac{\e}{2R^2}\|x - x^1\|_2^2,   
\end{equation}
where the sample size $m$ is defined in  \eqref{eq:SNSm}.
Thus, in the case of non-strongly convex objective, a regularization equates the sample size of both approaches.

\section{Strongly Convex Optimization Problem}

We start with preliminary results stated for a general stochastic strongly convex  optimization problem  of form
\begin{equation}\label{eq:gener_risk_min_conv}
\min_{x\in X \subseteq \mathbb{R}^n} F(x) \triangleq \E f(x,\xi),
\end{equation}
where $f(x,\xi)$ is $\gamma$-strongly convex with respect to $x$. Let us define   $ x^* = \arg\min\limits_{x\in X} {F}(x)$.

\subsection{The SA Approach: Stochastic Gradient Descent }

The classical SA algorithm for problem  \eqref{eq:gener_risk_min_conv} is presented by  stochastic gradient descent (SGD) method. We consider the SGD with inexect oracle  given by $g_\delta(x,\xi)$ such that
\begin{equation}\label{eq:gen_delta}
\forall x \in X, \xi \in \Xi, \qquad
\|\nabla f(x,\xi) - g_\delta(x,\xi)\|_2 \leq \delta.
\end{equation}
Then the iterative  formula of SGD can be written as ($k=1,2,...,N.$)
\begin{equation}\label{SA:implement_simple}
  x^{k+1} = \Pi_{X}\left(x^k - \eta_{k} g_\delta(x^k,\xi^k) \right). 
\end{equation}
Here $x^1 \in X$ is  starting point, $\Pi_X$ is the projection onto $X$, $\eta_k$ is a stepsize.
For a $\gamma$-strongly convex $f(x,\xi)$ in $x$, stepsize $\eta_k$ can be taken as $\frac{1}{\gamma k}$ to obtain  optimal rate $O(\frac{1}{\gamma N})$.

A good indicator of the success of an algorithm is the \textit{regret}
 \[Reg_N \triangleq \sum_{k=1}^{N} \left(f( x^k,  \xi^k) -  f(x^*,  \xi^k)\right).\]
It measures the value of the difference between a made decision and the optimal decision on all the rounds.
The work \cite{kakade2009generalization} gives a  bound on the excess risk of the output of an online algorithm in terms of the average regret.

\begin{theorem}\citep[Theorem 2]{kakade2009generalization} \label{Th:kakade2009generalization}
Let  $f:X\times \Xi \rightarrow [0,B]$ be $\gamma$-strongly convex and $M$-Lipschitz w.r.t. $x$. Let $\tilde x^N \triangleq \frac{1}{N}\sum_{k=1}^{N}x^k $ 
be the average of online vectors $x^1, x^2,...,x^N$.
Then with probability at least $1-4\beta\log N$
\[ F(\tilde x^N) - F(x^*) \leq \frac{Reg_N }{N} + 4\sqrt{ \frac{M^2\log(1/\beta)}{\gamma}}\frac{\sqrt{Reg_N}}{N} + \max\left\{ \frac{16M^2}{\gamma},6B \right\}\frac{\log(1/\beta)}{N}.  \]
\end{theorem}
For the update rule \eqref{SA:implement_simple} with $\eta_k = \frac{1}{\gamma k}$, this theorem can be specify as follows.

\begin{theorem}\label{Th:contract_gener}
Let  $f:X\times \Xi \rightarrow [0,B]$ be $\gamma$-strongly convex and $M$-Lipschitz w.r.t. $x$. Let $\tilde x^N \triangleq \frac{1}{N}\sum_{k=1}^{N}x^k $ be the average of outputs generated by iterative formula \eqref{SA:implement_simple} with $\eta_k = \frac{1}{\gamma k}$. Then,  with probability  
at least $1-\beta$  the following holds
\begin{align*}
    F(\tilde x^N) - F(x^*) 
    &\leq  \frac{3\delta D }{2}   + \frac{3(M^2+\delta ^2)}{N\gamma }(1+\log N)  \notag\\
    &+ \max\left\{ \frac{18M^2}{\gamma},6B + \frac{2M^2}{\gamma} \right\}\frac{\log(4\log N/\beta)}{N}.
\end{align*}
where $D =\max\limits_{x',x'' \in X}\|x'-x''\|_2 $ and $\delta$ is defined by \eqref{eq:gen_delta}.
\end{theorem}

\begin{proof}  
The proof mainly relies on  Theorem \ref{Th:kakade2009generalization} and estimating the regret for iterative formula \eqref{SA:implement_simple} with $\eta_k = \frac{1}{\gamma k}$.

From $\gamma$-strongly convexity in  $x$ of $f(x,\xi)$,  it follows for $x^k, x^* \in X$
\begin{equation*}
   f(x^*, \xi^k) \geq f(x^k,  \xi^k) + \la \nabla f(x^k, \xi^k), x^*-x^k\ra +\frac{\gamma}{2}\|x^*-x^k\|_2. 
\end{equation*}
Adding and subtracting the term  $\la g_\delta(x^k, \xi^k), x^*-x^k\ra$ we get using Cauchy–Schwarz inequality and \eqref{eq:gen_delta} 
\begin{align}\label{str_conv_W1}
   f(x^*,  \xi^k) &\geq f(x^k,  \xi^k) + \la g_\delta(x^k,\xi^k), x^*-x^k\ra +\frac{\gamma}{2}\|x^*-x^k\|_2   \notag \\
    &+ \la \nabla f(x^k,  \xi^k) -g_\delta(x^k,\xi^k), x^*-x^k\ra \notag\\
    &\geq f(x^k,  \xi^k) + \la g_\delta(x^k,\xi^k), x^*-x^k\ra + 
    \frac{\gamma}{2}\|x^*-x^k\|_2 + \delta\|x^*-x^k\|_2.
\end{align}
From the update rule \eqref{SA:implement_simple} for $x^{k+1}$  we have
\begin{align*}
   \|x^{k+1} - x^*\|_2 &= \|\Pi_{X}(x^k - \eta_{k} g_\delta(x^k,\xi^k)) - x^*\|_2 \notag\\
   &\leq \|x^k - \eta_{k} g_\delta(x^k,\xi^k) - x^*\|_2 \notag\\
   &\leq \|x^k  - x^*\|_2^2 + \eta_{k}^2\| g_\delta(x^k,\xi^k)\|_2^2 -2\eta_{k}\la g_\delta(x^k,\xi^k), x^k  - x^*\ra.
\end{align*}
From this it follows
\begin{equation*}
    \la g_\delta(x^k,\xi^k), x^k  - x^*\ra \leq \frac{1}{2\eta_{k}}(\|x^k-x^*\|^2_2 - \|x^{k+1} -x^*\|^2_2) + \frac{\eta_{k}}{2}\| g_\delta(x^k,\xi^k)\|_2^2.
\end{equation*}
Together with \eqref{str_conv_W1} we get
\begin{align*}
     f(x^k,  \xi^k) -  f(x^*,  \xi^k) &\leq  \frac{1}{2\eta_{k}}(\|x^k-x^*\|^2_2 - \|x^{k+1} -x^*\|^2_2)   \notag \\
     &-\left(\frac{\gamma}{2}+\delta\right)\|x^*-x^k\|_2+ \frac{\eta_{k}^2}{2}\| g_\delta(x^k,\xi^k)\|_2^2.
\end{align*}
Summing this from 1 to $N$,    we get using $\eta_k = \frac{1}{\gamma k}$
\begin{align}\label{eq:eq123}
 \sum_{k=1}^{N}f( x^k,  \xi^k) -  f(x^*,  \xi^k) &\leq 
     \frac{1}{2}\sum_{k=1}^{N}\left(\frac{1}{\eta_k} -  \frac{1}{\eta_{k-1}} + {\gamma} +\delta \right)\|x^*-x^k\|_2 \notag\\ &\hspace{-1cm}+\frac{1}{2}\sum_{k=1}^{N}{\eta_{k}}\| g_\delta (x^k,\xi^k)\|_2^2 \notag\\
     &\hspace{-1cm}\leq \frac{\delta}{2}\sum_{k=1}^{N}\|x^*-x^k\|_2 +\frac{1}{2}\sum_{k=1}^{N}{\eta_{k}}\| g_\delta (x^k,\xi^k)\|_2^2.
 \end{align}
 From Lipschitz continuity of $f(x,\xi)$ w.r.t. to $x$ it follows that $\|\nabla f(x,\xi)\|_2\leq M$ for all $x \in X, \xi \in \Xi$. Thus, using that for all $a,b, ~ (a+b)^2\leq 2a^2+2b^2$ it follows
 \[
 \|g_\delta(x,\xi)\|^2_2 \leq 2\|\nabla f(x,\xi)\|^2_2 + 2\delta^2 = 2M^2 + 2\delta^2
 \]
 From this and \eqref{eq:eq123} we bound the regret as follows
 \begin{align}\label{eq:reg123}
 Reg_N \triangleq \sum_{k=1}^{N}f( x^k,  \xi^k) -  f(x^*,  \xi^k)   &\leq  \frac{\delta}{2}\sum_{k=1}^{ N}\|p^*-p^k\|_2 + (M^2 +\delta^2)\sum_{k=1}^{ N} \frac{1}{\gamma k}\notag\\
     &\leq \frac{1}{2}  \delta D N + \frac{M^2+\delta ^2}{\gamma }(1+\log N). 
 \end{align}
Here the last bound takes place due to the sum of harmonic series.
Then for \eqref{eq:reg123} we can use Theorem \ref{Th:kakade2009generalization}. Firstly, we simplify it rearranging the terms using that $\sqrt{ab} \leq \frac{a+b}{2}$ 
\begin{align*}
    F(\tilde x^N) - F(x^*) 
    &\leq \frac{Reg_N }{N} + 4\sqrt{ \frac{M^2\log(1/\beta)}{N\gamma}}\sqrt{\frac{Reg_N}{N}} + \max\left\{ \frac{16M^2}{\gamma},6B \right\}\frac{\log(1/\beta)}{N}
    \notag\\
    &\leq \frac{3 Reg_N }{N} + \frac{2 M^2\log(1/\beta)}{N\gamma} + \max\left\{ \frac{16M^2}{\gamma},6B \right\}\frac{\log(1/\beta)}{N}\notag\\
    &= \frac{3 Reg_N }{N}  + \max\left\{ \frac{18M^2}{\gamma},6B + \frac{2M^2}{\gamma} \right\}\frac{\log(1/\beta)}{N}.
\end{align*}
Then we substitute \eqref{eq:reg123} in this inequality and making change $\beta = 4\beta\log N$ and get with probability at least $ 1-\beta$
\begin{align*}
    F(\tilde x^N) - F(x^*) 
    &\leq  \frac{3\delta D }{2}   + \frac{3(M^2+\delta ^2)}{N\gamma }(1+\log N)  \notag\\
    &+ \max\left\{ \frac{18M^2}{\gamma},6B + \frac{2M^2}{\gamma} \right\}\frac{\log(4\log N/\beta)}{N}.
\end{align*}
\end{proof}

\subsection{Preliminaries on the SAA Approach }
The SAA approach replaces the objective in  \eqref{eq:gener_risk_min_conv} with its sample average 
\begin{equation}\label{eq:empir_risk_min_conv}
 \min_{x\in X } \hat{F}(x) \triangleq \frac{1}{m}\sum_{i=1}^m f(x,\xi_i),  
\end{equation}
where each $f(x,\xi_i)$ is $\gamma$-strongly convex in $x$.
Let us define the empirical minimizer of \eqref{eq:empir_risk_min_conv} $\hat x^* = \arg\min\limits_{x\in X} \hat{F}(x)$, and   $\hat x_{\e'}$  such that
\begin{equation}\label{eq:fidelity}
\hat{F}(\hat x_{\e'}) - \hat{F}(\hat x^*)  \leq \e'.
\end{equation}
The next theorem  gives a  bound on the excess risk for problem \eqref{eq:empir_risk_min_conv}
in the SAA approach.
\begin{theorem}\label{Th:contractSAA}
Let  $f:X\times \Xi \rightarrow [0,B]$ be $\gamma$-strongly convex and $M$-Lipschitz w.r.t. $x$ in the $\ell_2$-norm. Let $\hat x_{\e'}$ satisfies  \eqref{eq:fidelity}  with precision $ \e' $. 
Then,  with probability at least $1-\beta$  we have
\begin{align*}
F( \hat x_{\e'}) -  F(x^*)
   &\leq \sqrt{\frac{2M^2}{\gamma}\e'}  +\frac{4M^2}{\beta\gamma m}. 
\end{align*}
Let $\e' = O \left(\frac{\gamma\e^2}{M^2}  \right)$ and $m = O\left( \frac{M^2}{\beta \gamma \e} \right)$. Then, with probability  at least $1-\beta$  the following holds
\[F( \hat x_{\e'}) -  F(x^*)\leq \e \quad \text{and} \quad \|\hat x_{\e'} - x^*\|_2 \leq \sqrt{2\e/\gamma}.\]
\end{theorem}
The proof of this theorem mainly relies on  the following theorem.
\begin{theorem}\citep[Theorem 6]{shalev2009stochastic}\label{Th:shalev2009stochastic}
Let $f(x,\xi)$ be $\gamma$-strongly convex and $M$-Lipschitz  w.r.t. $x$ in the $\ell_2$-norm.  Then, with probability at least $1-\beta$ the following holds
\[
F(\hat x^* ) - F(x^*) \leq \frac{4M^2}{\beta \gamma m},
\]
where $m$ is the sample size.
\end{theorem} 
\begin{proof}[Proof of Theorem \ref{Th:contractSAA}]

For any $x\in X$, the following holds
\begin{equation}\label{eq:exp_der}
 F(x) -  F(x^*) =    F(x) -   F(\hat x^*)+ F( \hat x^*)- F(x^*). 
\end{equation}
From  Theorem \ref{Th:shalev2009stochastic}  with probability  at least $1 - \beta$  the following holds
\begin{equation*}
F( \hat x^*) -  F(x^*) \leq  \frac{4M^2}{\beta \gamma m}.
\end{equation*}
Then from this and \eqref{eq:exp_der}  we have  with probability  at least $1 - \beta$
\begin{equation}\label{eq:eqtosub}
 F(x) -  F(x^*) \leq F(x) -   F(\hat x^*)+  \frac{4M^2}{\beta \gamma m}. 
\end{equation}
From Lipschitz continuity of $ f(x,\xi)$ it follows, that for any $x\in X, \xi\in \Xi$ the following holds
\begin{equation*}
| f(x,\xi) - f(\hat x^*,\xi)| \leq M\|x-\hat x^*\|_2.
\end{equation*}
Taking the expectation of this inequality w.r.t. $\xi$ we get 
\begin{equation*}
   \E | f(x,\xi) - f(\hat x^*,\xi)| \leq M\|x-\hat x^*\|_2.
\end{equation*}
Then we use Jensen's inequality ($g\left (\E(Y)\right) \leq \E g(Y) $) for the expectation,  convex function $g$ and  a random variable $Y$. Since the module is a convex function we get
\begin{equation*}
 | \E f(x,\xi) - \E f(\hat x^*,\xi)| =| F(x) -  F(\hat x^*)| \leq   \E | f(x,\xi) - f(\hat x^*,\xi)| \leq M\|x-\hat x^*\|_2.
\end{equation*}
Thus, we have 
\begin{equation}\label{eq:Lipsch}
   | F(x) -  F(\hat x^*)| \leq M\|x-\hat x^*\|_2.
\end{equation}
From strong convexity of $ f( x, \xi)$ in $x$, it follows that the average of $f(x,\xi_i)$'s, that is $\hat F(x)$, is also $\gamma$-strongly convex in $x$. Thus we get for any $x \in X, \xi \in \Xi$
\begin{equation}\label{eq:str}
\|x-\hat x^*\|_2 \leq \sqrt{\frac{2}{\gamma } (\hat F(x) - \hat F(\hat x^*))}.
\end{equation}
By using \eqref{eq:Lipsch} and \eqref{eq:str}  \label{eq:th6} and taking $x=\hat x_{\e'}$ in \eqref{eq:eqtosub},  we get the first statement of the theorem
\begin{align}\label{eq_to_prove_2}
 F( \hat x_{\e'}) -  F(x^*)
  &\leq   \sqrt{\frac{2M^2}{\gamma }(\hat F( \hat x_{\e'}) -  \hat F(\hat x^*))} +\frac{4M^2}{\beta\gamma m} \leq \sqrt{\frac{2M^2}{\gamma}\e'}  +\frac{4M^2}{\beta\gamma m}. 
\end{align}
Then from the strong convexity we have
\begin{align}\label{eq:con_reg_bar}
\| \hat x_{\e'} - x^*\|_2 
  &\leq  \sqrt{\frac{2}{\gamma}\left( \sqrt{\frac{2M^2}{\gamma }\e'} +\frac{4M^2}{\beta \gamma m}\right)}. 
\end{align}
Equating \eqref{eq_to_prove_2} to $\e$, we get the expressions for the sample size $m$ and auxiliary precision $\e'$. Substituting both of these expressions in  \eqref{eq:con_reg_bar} we finish the proof.

\end{proof}

\section{Non-Strongly Convex Optimization Problem}
Now we  consider non-strongly convex    optimization problem 
\begin{equation}\label{eq:gener_risk_min_nonconv}
\min_{x\in X \subseteq \mathbb{R}^n} F(x) \triangleq \E f(x,\xi),
\end{equation}
where $f(x,\xi)$ is Lipschitz continuous  in $x$. Let us define   $ x^* = \arg\min\limits_{x\in X} {F}(x)$.

\subsection{The SA  Approach: Stochastic Mirror Descent}
We consider stochastic mirror descent (MD) with inexact oracle  \cite{nemirovski2009robust,juditsky2012first-order,gasnikov2016gradient-free}.\footnote{By using dual averaging scheme~\cite{nesterov2009primal-dual} we can rewrite Alg.~\ref{Alg:OnlineMD} in online regime \cite{hazan2016introduction,orabona2019modern} without including $N$ in the stepsize policy. Note, that mirror descent and dual averaging scheme are very close to each other \cite{juditsky2019unifying}.}  For a prox-function $d(x)$ and the corresponding Bregman divergence $B_d(x,x^1)$, the proximal mirror descent  step is 
\begin{equation}\label{eq:prox_mirr_step}
   x^{k+1} = \arg\min_{x\in X}\left( \eta \left\langle g_\delta(x^k,\xi^k), x\right\rangle + B_d(x,x^k)\right).
\end{equation}
We consider the simplex setup: 
 prox-function $d(x) = \la x,\log x \ra$. Here and below, functions such as $\log$ or $\exp$ are always applied element-wise. The corresponding Bregman divergence is given by the Kullback--Leibler divergence
    \[
{\rm KL}(x,x^1) = \la x, \log(x/x^1)\ra - \boldsymbol{1}^\top(x-x^1).
\]
Then the starting point is taken as $x^1 = \arg\min\limits_{x\in \Delta_n}d(x)= (1/n,...,1/n)$.

\begin{theorem}\label{Th:MDgener}
Let $ R^2 \triangleq {\rm KL}(x^*,x^1) \leq \log n $ and $D =\max\limits_{x',x''\in \Delta_n}\|x'-x''\|_1 = 2$. Let  $f:X\times \Xi \rightarrow \R^n$ be $M_\infty$-Lipschitz w.r.t. $x$ in the $\ell_1$-norm. Let $\breve x^N \triangleq \frac{1}{N}\sum_{k=1}^{N}x^k $ be the average of outputs generated by iterative formula \eqref{eq:prox_mirr_step} with $\eta = \frac{\sqrt{2} R}{M_\infty\sqrt{N} }$. Then,  with probability  
at least $1-\beta$  we have 
\begin{equation*}
F(\breve x^N) - F(x^*)   \leq\frac{M_\infty (3R+2D \sqrt{\log (\beta^{-1})})}{\sqrt{2N}} +\delta D = O\left(\frac{M_\infty \sqrt{\log ({n}/{\beta})}}{\sqrt{N}} +2 \delta \right). 
\end{equation*} 
\end{theorem}

\begin{proof}
For  MD with prox-function function $d(x) = \la x\log x\ra $ the following holds for any $x\in \Delta_n$ \citep[Eq. 5.13]{juditsky2012first-order}
\begin{align*}
   \eta\la g_\delta (x^k,\xi^k), x^k -x \ra
 &\leq {\rm {\rm KL}}(x,x^k) - {\rm KL}(x,x^{k+1}) +\frac{\eta^2}{2}\|g_\delta(x^k,\xi^k)\|^2_\infty \\
 &\leq {\rm {\rm KL}}(x,x^k) -{\rm {\rm KL}}(x,x^{k+1}) +\eta^2M_\infty^2.
 \end{align*}
Then by adding and subtracting the terms $\la F(x), x-x^k\ra$ and $\la \nabla  f(x, \xi^k), x-x^k\ra$  in this inequality, we get using 
Cauchy--Schwarz inequality  the following
\begin{align}\label{str_conv_W2}
   \eta\la \nabla F(x^k), x^k-x\ra   
   &\leq  \eta\la  \nabla  f(x^k,\xi^k)-g_\delta(x^k,\xi^k), x^k-x\ra \notag \\
   &+ \eta\la\nabla F(x^k)- \nabla f(x^k,\xi^k)  , x^k-x\ra + {\rm KL}(x,x^k) - {\rm KL}(x,x^{k+1}) +\eta^2M_\infty^2 \notag\\
   &\leq  \eta\delta\max_{k=1,...,N}\|x^k-x\|_1 + \eta\la\nabla F(x^k)- \nabla f(x^k,\xi^k)  , x^k-x\ra \notag \\ 
   &+{\rm KL}(x,x^k) - {\rm KL}(x,x^{k+1}) +\eta^2M_\infty^2.
\end{align}
Then using convexity of $F(x^k)$ we have
\[
 F(x^k) - F(x)\leq \eta\la \nabla F(x^k), x^k-x\ra 
\]
Then we use this for \eqref{str_conv_W2} and sum  for $k=1,...,N$ at $x=x^*$
\begin{align}\label{eq:defFxx}
     \eta\sum_{k=1}^N  F(x^k) - F(x^*) &\leq 
    \eta\delta N \max_{k=1,...,N}\|x^k-x^*\|_1
     +\eta\sum_{k=1}^N\la\nabla F(x^k)- \nabla f(x^k,\xi^k)  , x^k-x^*\ra \notag\\
    &+ {\rm KL}(x^*,x^1) - {\rm KL}(x^*,x^{N+1}) + \eta^2M_\infty^2N \notag\\ 
     &\leq  \eta\delta N{D}
     +\eta \sum_{k=1}^N\la\nabla  F(x^k)- \nabla f(x^k,\xi^k)  , x^k-x^*\ra + R^2+ \eta^2M_\infty^2N. 
\end{align}
Where we used ${\rm KL}(x^*,x^1) \leq R^2 $ and $\max\limits_{k=1,...,N}\|p^k-p^*\|_1 \leq D$.
Then using convexity of $F(x^k)$ and the definition of output $\breve x^N$ in \eqref{eq:defFxx} we have
\begin{align}\label{eq:F123xxF}
    F(\breve x^N) -F(x^*)
    &\leq  \delta D +\frac{1}{N}\sum_{k=1}^N \la\nabla F(x^k)- \nabla  f(x^k,\xi^k)  , x^k-x^*\ra + \frac{R^2}{\eta N}+ \eta M_\infty^2.
\end{align}
Next we use the {Azuma--}Hoeffding's {\cite{jud08}} inequality and get for all $\beta \geq 0$
\begin{equation}\label{eq:AzumaH}
    \mathbb{P}\left(\sum_{k=1}^{N+1}\la \nabla F(x^k) -\nabla f(x^k,\xi^k), x^k-x^*\ra \leq \beta \right)\geq 1 - \exp\left( -\frac{2\beta^2}{N(2M_\infty D)^2}\right)= 1 - \beta.
\end{equation}
Here we used that $\la \nabla F(p^k) -\nabla f(x^k,\xi^k), x^*-x^k\ra$ is a martingale-difference and 
\begin{align*}
{\left|\la \nabla  F(x^k) -\nabla f(x^k,\xi^k), x^*-x^k\ra \right|} &\leq \| \nabla F(x^k) -\nabla W(p^k,q^k)\|_{\infty} \|x^*-x^k\|_1 \notag \\
&\leq 2M_\infty \max\limits_{k=1,...,N}\|x^k-x^*\|_1 \leq 2M_\infty D.
\end{align*}
Thus, using \eqref{eq:AzumaH} for \eqref{eq:F123xxF} we have that  with probability at least $1-\beta$ 
\begin{equation}\label{eq:eta}
   F(\breve x^N) -  F(x^*) \leq  \delta D +\frac{\beta}{N}+ \frac{R^2}{\eta N}+ \eta M_\infty^2.
\end{equation}
Then, expressing $\beta$ through $\beta$  and substituting $\eta = \frac{ R}{M_\infty} \sqrt{\frac{2}{N}}$ to \eqref{eq:eta}  ( such $\eta$ minimize the r.h.s. of \eqref{eq:eta}),  we get \begin{align*}
  & F(\breve x^N) -  F(x^*) \leq   \delta D + \frac{M_\infty D\sqrt{2\log(1/\beta)} }{\sqrt{N} } + \frac{M_\infty R}{\sqrt{2N}}+ \frac{M_\infty R\sqrt{2}}{\sqrt{N}} \notag \\
  &\leq \delta D + \frac{M_\infty (3R+2D \sqrt{\log(1/\beta) })}{\sqrt{2N}}. 
\end{align*} 
Using $R=\sqrt{\log n}$ and {$D = 2$}  in this inequality,
we obtain
\begin{align}\label{eq:final_est}
  F(\breve x^N) -  F(x^*)  &\leq \frac{M_\infty (3\sqrt{\log{n}} +4 \sqrt{\log(1/\beta)})}{\sqrt{2N}} +2\delta. 
\end{align} 
We raise this to the second power, use that for all $a,b\geq 0, ~ 2\sqrt{ab}\leq a+b$ and then extract the square root. We obtain the following
\begin{align*}
\sqrt{\left(3\sqrt{\log{n}} +4 \sqrt{\log(1/\beta)}\right)^2} &= \sqrt{ 9\log{n} + 16\log(1/\beta) +24\sqrt{\log{n}}\sqrt{\log(1/\beta)}  } \\
&\leq  \sqrt{ 18\log{n} + 32\log(1/\beta) }.
\end{align*}
Using this for \eqref{eq:final_est}, we get the  statement of the theorem
\begin{align*}\label{eq:final_est2}
 F(\breve x^N) -  F(x^*)  &\leq \frac{M_\infty \sqrt{18\log{n} +32 \log(1/\beta)}}{\sqrt{2N}} +2\delta  = O\left(\frac{M_\infty \sqrt{\log ({n}/{\beta})}}{\sqrt{N}} +2\delta  \right) . 
\end{align*} 
\end{proof}

\subsection{ Penalization in the SAA Approach}
In this section, we study the SAA approach for non-strongly convex problem \eqref{eq:gener_risk_min_nonconv}. We regularize this problem by 1-strongly convex w.r.t. $x$ penalty function $r(x,x^1)$ in the $\ell_2$-norm 
 \begin{equation}\label{def:gener_reg_prob}
\min_{x\in X \subseteq \mathbb{R}^n} F_\lm(x) \triangleq \E f(x,\xi) +   \lm r(x,x^1)
\end{equation}
and we prove  that the sample sizes in the SA and the SAA approaches will be equal   up to logarithmic terms. 
The empirical counterpart of problem \eqref{def:gener_reg_prob} is
 \begin{equation}\label{eq:gen_prob_empir}
 \min_{x\in X }\hat{F}_\lm(x) \triangleq \frac{1}{m}\sum_{i=1}^m f(x,\xi_i) +   \lm r(x,x^1).
\end{equation}
Let us define $ \hat x_\lm = \arg\min\limits_{x\in X} \hat{F}_{\lm}(x)$.
The next lemma  proves the statement from \cite{shalev2009stochastic}  on boundness of the population sub-optimality in terms of the square root of empirical sub-optimality.
\begin{lemma}\label{Lm:pop_sub_opt}
Let $f(x,\xi)$ be convex and $M$-Lipschitz continuous w.r.t. $x$ in the $\ell_2$-norm.  Let $r(x,x^1)$ be 1-strongly convex  and  $M_r$-Lipschitz continuous w.r.t. $x$ in the  $\ell_2$-norm.
Then for any $x \in X$ with probability at least $ 1-\beta$ the following holds
\[F_\lm(x) - F_\lm(x^*_\lm) \leq \sqrt{\frac{2M_\lm^2}{\lm} \left(\hat F_\lm(x) - \hat F_\lm(\hat x_\lm)\right)} + \frac{4M_\lm^2}{\beta \lm m},\]
 where $x^*_\lm = \arg \min\limits_{x\in X} F_\lm(x) $,
 $M_\lm \triangleq M +\lambda  {M}_r$.
\end{lemma}

\begin{proof}
Let us define $f_\lm(x, \xi) \triangleq f (x, \xi) + \lm r(x,x^1) $. As $f(x, \xi)$ is $M$-Lipschitz continuous, $f_\lm(x, \xi)$ is also Lipschitz continuous with  $M_\lm \triangleq M +\lambda  {M}_r$.
From  
 Jensen's inequality for the expectation, and the module as a  convex function,  we get that 
 $F_\lm(x)$  is also  $M_\lm$-Lipschitz continuous
\begin{equation}\label{eq:MLipscht_cont}
|F_\lm(x)-  F_\lm(\hat x_\lm) |  \leq  M_\lm\| x -\hat x_\lm\|_2, \qquad \forall x \in X.
\end{equation}
From   $\lm$-strong convexity of $f(x, \xi)$, we obtain that  $\hat F_\lm(x)$ is also $\lm$-strongly convex 
\[
\|x-\hat x_\lm\|_2^2\leq \frac{2}{\lm}\left( \hat F_\lm(x)-\hat F_\lm(\hat x_\lm)  \right), \qquad \forall x \in X. 
\]
From this and \eqref{eq:MLipscht_cont}  it follows
\begin{equation}\label{eq:sup_opt_empr}
  F_\lm(x)-  F_\lm(\hat x_\lm)  \leq \sqrt{\frac{2M_\lm^2}{\lm}\left( \hat F_\lm(x)-\hat F_\lm(\hat x_\lm) \right)}.  
\end{equation}
For any $x \in X$ and $ x^*_\lm = \arg\min\limits_{x\in X} {F}_\lm(x)$ we consider
\begin{equation}\label{eq:Fxhatx}
    F_\lm(x)-  F_\lm( x^*_\lm) = F_\lm(x)- F_\lm(\hat x_\lm) + F_\lm(\hat x_\lm) - F_\lm( x^*_\lm).
\end{equation}
From \citep[Theorem 6]{shalev2009stochastic} we have with probability at least $ 1 -\beta$
\[ F_\lm(\hat x_\lm) - F_\lm(x^*_\lm) \leq \frac{4M_\lm^2}{\beta \lm m}.\]
Using  this and \eqref{eq:sup_opt_empr} for \eqref{eq:Fxhatx} we obtain with probability at least $ 1 -\beta$
\[F_\lm(x) - F_\lm(x^*_\lm) \leq \sqrt{\frac{2M_\lm^2}{\lm}\left( \hat F_\lm(x)-\hat F_\lm(\hat x_\lm) \right)} + \frac{4M_\lm^2}{\beta \lm m}.\]
\end{proof}
The next theorem proves the eliminating the linear dependence on $n$ in the sample size of the regularized SAA approach for a non-strongly convex objective 
(see estimate \eqref{eq:SNSm}), and estimates the  auxiliary precision for the regularized SAA problem \eqref{eq:aux_e_quad}.
\begin{theorem}\label{th_reg_ERM} 
Let $f(x,\xi)$ be convex and $M$-Lipschitz continuous w.r.t. $x$ in the $\ell_2$-norm and let $r(x,x^1)$ be 1-strongly convex  and  $M_r$-Lipschitz continuous w.r.t. $x$ in the  $\ell_2$-norm.
Let $\hat x_{\e'}$ be such that
\[
\frac{1}{m}\sum_{i=1}^m f(\hat x_{\e'},\xi_i) +   \lm r(\hat x_{\e'}, x^1) - \arg\min_{x\in X}  \left\{\frac{1}{m}\sum_{i=1}^m f(x,\xi_i) +   \lm r(x,x^1)\right\} \leq \e'. 
\]
To satisfy
\[F(\hat x_{\e'}) -  F(x^*)\leq \e\]
with probability  at least $1-\beta$,
we need to take  $\lm = \e/(2\mathcal{R}^2)$ and
 \[m = \frac{ 32 M^2\mathcal{R}^2}{\beta \e^2}, \]
 where
$\mathcal{R}^2 =  r(x^*,x^1)$. The precision $\e'$ is defined as
\[\e' = \frac{\e^3}{64M^2 \mathcal{R}^2}.\]
\end{theorem}

\begin{proof}
From Lemma \ref{Lm:pop_sub_opt}
we get for $x=\hat x_{\e'}$,
\begin{align}\label{eq:suboptimality}
F_\lm(\hat x_{\e'}) - F_\lm(x^*_\lm) &\leq \sqrt{\frac{2M_\lm^2}{\lm}\left( \hat F_\lm(\hat x_{\e'})-\hat F_\lm(\hat x_\lm ) \right)} + \frac{4M_\lm^2}{\beta \lm m} =\sqrt{\frac{2M_\lm^2}{\lm}\e'} + \frac{4M_\lm^2}{\beta \lm m},
\end{align}
 where we used the definition of $\hat x_{\e'}$ from the statement of the this theorem and  $M_\lm \triangleq M +\lambda  {M}_r$.
Then we  subtract $F(x^*)$ in both sides of \eqref{eq:suboptimality} and get 
\begin{align}\label{eq:suboptimalityIm}
  F_\lm( \hat x_{\e'}) - F(x^*)  
  &\leq \sqrt{\frac{2M_\lm^2\e'}{\lambda}} +\frac{4M_\lm^2}{\beta\lambda m} +F_\lm(x^*_\lm)-F(x^*). 
\end{align}
Then we use
\begin{align*}
    F_\lm(x^*_\lm) &\triangleq \min_{x\in X}\left\{ F(x)+\lm r(x,x^1) \right\} && \notag\\
    &\leq F(x^*) + \lm r(x^*,x^1) &&  \text{The inequality holds for any } x \in X, \notag\\
    &\triangleq F(x^*) +\lm \mathcal{R}^2
\end{align*}
where $\mathcal{R}^2 \triangleq r(x^*,x^1)$.
Then from this and \eqref{eq:suboptimalityIm} and the definition of $F_\lm(\hat x_{\e'})$ in \eqref{def:gener_reg_prob}  we get 
\begin{align}\label{eq:minim+lamb}
  F( \hat x_{\e'}) -  F(x^*) &\leq \sqrt{\frac{2M_\lm^2}{\lm}\e'} + \frac{4M_\lm^2}{\beta \lm m} - \lambda r(\hat x_{\e'}, x^1)+{\lambda}\mathcal{R}^2\notag \\
  &\leq \sqrt{\frac{2M_\lm^2\e'}{\lambda}} +\frac{4M_\lm^2}{\beta\lambda m} +{\lambda}\mathcal{R}^2. 
\end{align}
Let us remind  that $M_\lm \triangleq M +\lambda  {M}_r$. Then 
assuming $M \gg \lambda M_r  $ and choosing $\lm =\e/ (2\mathcal{R}^2)$ in \eqref{eq:minim+lamb}, we get the
following
\begin{equation}\label{offline23}
  F( \hat x_{\e'}) - F(x^*) \leq   \sqrt{\frac{4M^2\mathcal R^2\e'}{\e}} +\frac{8M^2\mathcal R^2}{\beta m \e} +\e/2.
\end{equation}
Equating the first and the second terms in the r.h.s. of \eqref{offline23} to $\e/4$  respectively,  we obtain the 
 the rest  statements of the theorem including  $  F( \hat x_{\e'}) - F(x^*) \leq \e.$ 

\end{proof}

\section{Fr\'{e}chet Mean  with respect to Entropy-Regularized Optimal Transport }\label{sec:pen_bar}
\chaptermark{Fr\'{e}chet Mean}

In this section, we consider the problem of finding population barycenter of independent identically distributed random discrete  measures. We define the population barycenter of distribution $\PP$ with respect  to
 entropy-regularized  transport distances 
 \begin{equation}\label{def:populationWBFrech}
\min_{p\in  \Delta_n}
W_\gamma(p)\triangleq
\E_q W_\gamma(p,q), \qquad q \sim \PP.
\end{equation}

\subsection{Properties of  Entropy-Regularized Optimal Transport}
Entropic regularization of transport distances \cite{cuturi2013sinkhorn}  improves their statistical properties  \cite{klatt2020empirical,bigot2019central} and reduces their computational complexity. Entropic regularization has shown good results in generative models \cite{genevay2017learning}, 
multi-label learning \cite{frogner2015learning}, dictionary learning \cite{rolet2016fast}, image processing  \cite{cuturi2016smoothed,rabin2015convex}, neural imaging \cite{gramfort2015fast}. 

Let us firstly remind   optimal transport problem (introduced in Eq. \eqref{def:optimaltransport})  between histograms $p,q \in \Delta_n$  with cost matrix $C\in \R_{+}^{n\times n}$
\begin{equation}\label{eq:OTproblem}
 W(p,q) \triangleq   \min_{\pi \in U(p,q)} \la C, \pi \ra,
\end{equation}
where
\[U(p,q) \triangleq\{ \pi\in \R^{n\times n}_+: \pi \one =p, \pi^T \one = q\}.\]

\begin{remark}[Connection with the $\rho$-Wasserstein distance]
When for $\rho\geq 1$, $C_{ij} =\mathtt d(x_i, x_j)^\rho$  in \eqref{eq:OTproblem}, where $\mathtt d(x_i, x_j)$ is a distance on support points $x_i, x_j$, then $W(p,q)^{1/\rho}$ is known as the $\rho$-Wasserstein distance.
\end{remark}
Nevertheless, all the results of this thesis  are based only on the assumptions that the  matrix $C \in \R_+^{n\times n}$ is symmetric and non-negative. Thus, optimal transport  problem defined in \eqref{eq:OTproblem} is    a more general  than the  Wasserstein distances.

Following \cite{cuturi2013sinkhorn}, we introduce  entropy-regularized optimal transport problem  
\begin{align}\label{eq:wass_distance_regul}
W_\gamma (p,q) &\triangleq \min_{\pi \in  U(p,q)} \left\lbrace \left\langle  C,\pi\right\rangle - \gamma E(\pi)\right\rbrace,
\end{align}
 where $\gamma>0$ and $E(\pi) \triangleq -\la \pi,\log \pi \ra $ is the  entropy. Since $E(\pi)$ is 1-strongly concave on $\Delta_{n^2}$ in the  $\ell_1$-norm, the objective in \eqref{eq:wass_distance_regul} is $\gamma$-strongly convex  with respect to $\pi$ in the $\ell_1$-norm on $\Delta_{n^2}$, and hence problem \eqref{eq:wass_distance_regul} has a unique optimal solution. Moreover, $W_\gamma (p,q)$ is $\gamma$-strongly convex with respect to $p$ in the $\ell_2$-norm on $\Delta_n$ \citep[Theorem 3.4]{bigot2019data}.

One particular advantage of the entropy-regularized  optimal transport is a closed-form representation for its dual function~\cite{agueh2011barycenters,cuturi2016smoothed} defined by the Fenchel--Legendre transform of $W_\gamma(p,q)$ as a function of $p$
\begin{align}\label{eq:FenchLegdef}
     W_{\gamma, q}^*(u) &=  \max_{ p \in \Delta_n}\left\{ \la u, p \ra - W_{\gamma}(p, q) \right\} =  \gamma\left(E(q) + \left\langle q, \log (K \beta) \right\rangle  \right)\notag\\
    &= \gamma\left(-\la q,\log q\ra + \sum_{j=1}^n [q]_j \log\left( \sum_{i=1}^n \exp\left(([u]_i - C_{ji})/\gamma\right) \right)\right)
\end{align}
  where  $\beta = \exp( {u}/{\gamma}) $, \mbox{$K = \exp( {-C}/{\gamma }) $} and $[q]_j$ is $j$-th component of  vector $q$. Functions such as $\log$ or $\exp$ are always applied element-wise for vectors.
Hence, the gradient of dual function $W_{\gamma, q}^*(u)$ is also represented in  a closed-form \cite{cuturi2016smoothed}
\begin{equation*}
\nabla W^*_{\gamma,q} (u)
= \beta \odot \left(K \cdot {q}/({K \beta}) \right) \in \Delta_n,
\end{equation*}
where symbols $\odot$ and $/$ stand for the element-wise product and element-wise division respectively.
This can be also written as
\begin{align}\label{eq:cuturi_primal}
\forall l =1,...,n \qquad [\nabla W^*_{\gamma,q} (u)]_l = \sum_{j=1}^n [q]_j \frac{\exp\left(([u]_l-C_{lj})/\gamma\right)  }{\sum_{i=1}^n\exp\left(([u]_i-C_{ji})/\gamma\right)}.
\end{align}
The dual representation of $W_\gamma (p,q) $    is
\begin{align}\label{eq:dual_Was}
  W_{\gamma}(p,q) &= \min_{\pi \in U(p,q) }\sum_{i,j=1}^n\left(  C_{ij}\pi_{i,j}  + \gamma \pi_{i,j}\log \pi_{i,j} \right)  \notag \\
  &=\max_{u, \nu \in \R^n} \left\{  \la u,p\ra + \la\nu,q\ra  - \gamma\sum_{i,j=1}^n\exp\left( ([u]_i+[\nu]_j -C_{ij})/\gamma -1  \right) \right\} \\
  &=\max_{u \in \R^n}\left\{ \la u,p\ra -
  \gamma\sum_{j=1}^n [q]_j\log\left(\frac{1}{[q]_j}\sum_{i=1}^n\exp\left( ([u]_i -C_{ij})/\gamma \right)  \right)
  \right\}. \notag 
\end{align}
Any solution $\begin{pmatrix}
 u^*\\
 \nu^*
 \end{pmatrix}$ of  \eqref{eq:dual_Was}  is a subgradient 
 of $  W_{\gamma}(p,q)$  \citep[Proposition 4.6]{peyre2019computational}  \begin{equation}\label{eq:nabla_wass_lagrang}
 \nabla W_\gamma(p,q) = \begin{pmatrix}
 u^*\\
 \nu^*
 \end{pmatrix}.
\end{equation} We consider  $u^*$ and $\nu^*$ such that
  $\la u^*, \one\ra = 0$ and $\la \nu^*, \one\ra = 0$  ($u^*$ and $\nu^*$ are determined up to an additive constant).

The next theorem \cite{bigot2019data} describes  the Lipschitz continuity of $W_\gamma (p,q)$ in $p$ on probability simplex $\Delta_n$  restricted to
\[\Delta^\rho_n = \left\{p\in \Delta_n : \min_{i \in [n]}p_i \geq \rho \right\},\]
where $0<\rho<1$ is an arbitrary small constant. 
\begin{theorem}\citep[Theorem 3.4, Lemma 3.5]{bigot2019data}
\label{Prop:wass_prop}
\begin{itemize}
\item For any $q \in \Delta_n$, 
$W_\gamma (p,q)$ is $\gamma$-strongly convex w.r.t. $p$ in the $\ell_2$-norm
\item For any $q \in \Delta_n$, $p \in \Delta^\rho_n$ and $0<\rho<1$, 
$\|\nabla_p W_\gamma (p,q)\|_2 \leq M$, where 
\[M = \sqrt{\sum_{j=1}^n\left(   2\gamma\log n  +\inf_{i\in [n]}\sup_{l \in [n]} |C_{jl} -C_{il}|  -\gamma\log \rho \right)^2}. \]
\end{itemize}
\end{theorem}
We roughly take $M = O(\sqrt n\|C\|_\infty)$ since for all $i,j\in [n], C_{ij} > 0$, we get
\begin{align*}
M &\stackrel{\text{\cite{bigot2019data}}}{=}  O\left(\sqrt{\sum_{j=1}^n\left(  \inf_{i\in [n]}\sup_{l \in [n]} |C_{jl} -C_{il}|   \right)^2} \right) \\
&= O \left(\sqrt{\sum_{j=1}^n \sup_{l \in [n]} C_{jl}^2} \right)= O\left( \sqrt{n} \sup_{j,l \in [n]} C_{jl} \right)
= O\left( \sqrt{n}\sup_{j\in [n]}\sum_{l \in [n]}C_{jl} \right)= O \left( \sqrt{n}\|C\|_\infty\right).
\end{align*}
Thus, we suppose that $W_\gamma(p,q)$ and $W(p,q)$ are Lipschitz continuous with almost the same Lipschitz constant $M$ in the $\ell_2$-norm on $\Delta_n^\rho$. 
Moreover, 
by the same arguments,  
for the Lipschitz continuity in the $\ell_1$-norm: $\|\nabla_p W_\gamma (p,q)\|_\infty \leq M_\infty$,  we can roughly estimate $M_\infty = O(\|C\|_\infty)$ by taking maximum instead of the square root of the sum.

In what follows, we use Lipshitz continuity of $W_\gamma(p,q)$ and $W(p,q)$  for measures  from $\Delta_n$ keeping in mind that   adding some noise and normalizing the measures makes them belong to $\Delta_n^\rho$. We also notice that if the measures are from  the interior of $\Delta_n$ then their barycenter will be also from  the interior of $\Delta_n$. 






\subsection{The SA Approach: Stochastic Gradient Descent}



For  problem \eqref{def:populationWBFrech}, as a particular case of problem \eqref{eq:gener_risk_min}, stochastic gradient descent method can be used. From Eq. \eqref{eq:nabla_wass_lagrang}, it follows that an approximation for the gradient of $W_\gamma(p,q)$ with respect to $p$ can be calculated by Sinkhorn algorithm \cite{altschuler2017near-linear,peyre2019computational,dvurechensky2018computational}
through the computing dual variable $u$ with $\delta$-precision
 \begin{equation}\label{inexact}
 \|\nabla_p W_\gamma(p,q) - \nabla_p^\delta W_\gamma(p,q) \|_2\leq \delta, \quad \forall q\in \Delta_n. 
\end{equation}
Here denotation  $\nabla_p^\delta W_\gamma(p,q)$ means an inexact stochastic subgradient of $ W_\gamma(p,q)$ with respect to $p$.
Algorithm \ref{Alg:OnlineGD} combines stochastic gradient descent
given by iterative formula \eqref{SA:implement_simple}   for $\eta_k = \frac{1}{\gamma k}$ with Sinkhorn algorithm (Algorithm \ref{Alg:SinkhWas})
and Algorithm \ref{Alg:EuProj} making the projection onto the simplex $\Delta_n$.


\begin{algorithm}[ht!]
\caption{Sinkhorn's algorithm \cite{peyre2019computational} for calculating $ \nabla_p^\delta W_\gamma(p^k,q^k) $}
\label{Alg:SinkhWas}  
\begin{algorithmic}[1]
\Procedure{Sinkhorn}{$p,q, C, \gamma$}
\State $a^1 \gets  (1/n,...,1/n)$,  $ b^1 \gets  (1/n,...,1/n)$  
\State $K \gets \exp (-C/\gamma)$
\While{\rm not converged}
 \State $a \gets {p}/(Kb)$
\State $ b \gets {q }/(K^\top a)$
\EndWhile
 \State \textbf{return} $  \gamma \log(a)$\Comment{Sinkhorn scaling $  a = e^{u/\gamma}$} 
  \EndProcedure
\end{algorithmic}
\end{algorithm}

\begin{algorithm}[ht!]
\caption{Euclidean Projection $\Pi_{\Delta_n}(p) = \arg\min\limits_{v\in \Delta_n}\|p-v\|_2$ onto  Simplex   $\Delta_n$ \cite{duchi2008efficient}}
\label{Alg:EuProj}  
\begin{algorithmic}[1]
\Procedure{Projection}{$w\in \R^n$}
\State Sort components of $w$ in decreasing manner: $r_1\geq r_2 \geq ... \geq r_n$.
 \State Find  $\rho = \max\left\{ j \in [n]: r_j - \frac{1}{j}\left(\sum^{j}_{i=1}r_i - 1\right) \right\}$
\State Define
    $\theta = \frac{1}{\rho}(\sum^{\rho}_{i=1}r_i - 1)$
    \State For all $i \in [n]$, define $p_i = \max\{w_i - \theta, 0\}$.
 \State \textbf{return} $p \in \Delta_n$ 
  \EndProcedure
\end{algorithmic}
\end{algorithm}

\begin{algorithm}[ht!]
\caption{Projected  Online Stochastic Gradient Descent for WB (PSGDWB)} 
\label{Alg:OnlineGD}   
\begin{algorithmic}[1]
  \Require starting point $p^1 \in \Delta_n$, realization $q^1$,  $\delta$, $\gamma$.
                \For{$k= 1,2,3,\dots$}
                \State  $\eta_{k} = \frac{1}{\gamma k}$
                \State  $\nabla_p^\delta W_\gamma(p^k,q^k)  \gets$ \textsc{Sinkhorn}$(p^k,q^k, C, \gamma)$ or the accelerated Sinkhorn \cite{guminov2019accelerated}
                \State $p^{(k+1)/2} \gets p^k - \eta_{k} \nabla_p^\delta W_\gamma(p^k,q^k)$
                \State $p^{k+1} \gets$ \textsc{Projection}$(p^{(k+1)/2})$
                  \State Sample $q^{k+1}$ 
                  \EndFor
    \Ensure $p^1,p^2, p^3...$
\end{algorithmic}
 \end{algorithm}


 For Algorithm \ref{Alg:OnlineGD} and problem \eqref{def:populationWBFrech},  Theorem \ref{Th:contract_gener} can be specified as follows 
\begin{theorem}\label{Th:contract}
Let $\tilde p^N \triangleq \frac{1}{N}\sum_{k=1}^{N}p^k $ be the average of  $N$ online outputs of Algorithm  \ref{Alg:OnlineGD} run with $\delta$. Then,  with probability  
at least $1-\beta$  the following holds
\begin{equation*}
W_{\gamma}(\tilde p^N) - W_{\gamma}(p^*_{\gamma})
 = O\left(\frac{M^2\log(N/\beta)}{\gamma N} + \delta  \right),  
\end{equation*}
where $p^*_{\gamma} \triangleq \arg\min\limits_{p\in  \Delta_n}
W_\gamma(p)$.

Let  Algorithm \ref{Alg:OnlineGD} run with  $\delta = O\left(\e\right)$ and  $
N =\widetilde O \left( \frac{M^2}{\gamma \e} \right) = \widetilde O \left( \frac{n\|C\|_\infty^2}{\gamma \e} \right)
$. Then,  with probability  
at least $1-\beta$  
\begin{equation*}
 W_{\gamma}(\tilde p^N) -W_{\gamma}(p^*_{\gamma}) \leq \e \quad \text{and} \quad \|\tilde p^N - p^*_{\gamma}\|_2 \leq \sqrt{2\e/\gamma}.
\end{equation*}
The total complexity of 
 Algorithm  \ref{Alg:OnlineGD} 
is
\begin{align*}
   \widetilde O\left(  \frac{n^3\|C\|_\infty^2}{\gamma\e}\min\left\{ \exp\left( \frac{\|C\|_{\infty}}{\gamma} \right) \left( \frac{\|C\|_{\infty}}{\gamma} + \log\left(\frac{\|C\|_{\infty}}{\kappa \e^2} \right) \right),  \sqrt{\frac{n \|C\|^2_{\infty}}{ \kappa\gamma \e^2}} \right\} \right),
\end{align*}
where $ \kappa \triangleq \lm^+_{\min}\left(\nabla^2   W_{\gamma, q}^*(u^*)\right)$. 

\end{theorem}
\begin{proof}
We estimate the co-domain (image) of $W_\gamma(p,q)$
\begin{align*}
\max_{p,q \in \Delta_n}  W_\gamma (p,q) 
&= \max_{p,q \in \Delta_n} \min_{ \substack{\pi \in \R^{n\times n}_+, \\ \pi \one =p, \\ \pi^T \one = q}} ~ \sum_{i,j=1}^n (C_{ij}\pi_{ij}+\gamma\pi_{ij}\log \pi_{ij})\notag\\
&\leq \max_{\substack{\pi \in \R^{n\times n}_+, \\ \sum_{i,j=1}^n \pi_{ij}=1}}  \sum_{i,j=1}^n(C_{ij}\pi_{ij}+\gamma\pi_{ij}\log \pi_{ij}) \leq \|C\|_\infty.
\end{align*}
Therefore, $W_\gamma(p,q): \Delta_n\times \Delta_n\rightarrow \left[-2\gamma\log n, \|C\|_\infty\right]$.
Then we apply Theorem \ref{Th:contract_gener} with $B =\|C\|_\infty $  and $D =\max\limits_{p',p''\in \Delta_n}\|p'-p''\|_2 = \sqrt{2}$, and we sharply get
\begin{equation*}
W_{\gamma}(\tilde p^N) - W_{\gamma}(p^*_{\gamma})
 = O\left(\frac{M^2\log(N/\beta)}{\gamma N} +  \delta  \right),  
\end{equation*}
Equating each terms in the r.h.s. of this equality to $\e/2$ and using $M=O(\sqrt n \|C\|_\infty)$,  we get the expressions for $N$ and $\delta$. The statement 
$\|\tilde p^N - p^*_{\gamma}\|_2 \leq \sqrt{2\e/\gamma}$
follows directly from strong convexity of $W_\gamma(p,q)$ and  $W_\gamma(p)$.

The proof of algorithm complexity  follows from the complexity
of the  Sinkhorn's algorithm.
To state the complexity of the Sinkhorn's \avg{algorithm} we firstly define \avg{$\tilde\delta$ as the accuracy in function value of the inexact solution $u$ of  maximization problem in  \eqref{eq:dual_Was}.}
Using this 
we formulate the number of iteration of the Sinkhorn's 
\cite{franklin1989scaling,carlier2021linear,kroshnin2019complexity,stonyakin2019gradient} 
\begin{align}\label{eq:sink}
 \widetilde O \left( \exp\left( \frac{\|C\|_{\infty}}{\gamma} \right) \left( \frac{\|C\|_{\infty}}{\gamma} + \log\left(\frac{\|C\|_{\infty}}{\tilde \delta} \right) \right)\right). 
\end{align}
The number of iteration for the accelerated Sinkhorn's can be improved \cite{guminov2019accelerated}
\begin{equation}\label{eq:accel}
\widetilde{O} \left(\sqrt{\frac{n \|C\|^2_\infty}{\gamma \e'}} \right). 
\end{equation}
Here $\e'$ is the accuracy in the function value, which is the expression 
$ \la u,p\ra + \la\nu,q\ra  - \gamma\sum_{i,j=1}^n\exp\left( {(-C_{ji}+u_i+\nu_j)}/{\gamma} -1  \right)$ under the maximum in \eqref{eq:dual_Was}.
From strong convexity of this objective on the space orthogonal to eigenvector $\boldsymbol 1_n$ corresponds to the eigenvalue $0$ for this function, it follows that 
\begin{equation}\label{eq:str_k}
  \e'\geq \frac{\gamma}{2}\|u - u^*\|^2_2 = \frac{\kappa}{2}\delta,  
\end{equation}
 where $\kappa \triangleq \lm^+_{\min}\left(\nabla^2   W_{\gamma, q}^*(u^*)\right)$.  From \citep[Proposition A.2.]{bigot2019data}, for the eigenvalue of $\nabla^2 W^*_{\gamma,q}(u^*)$ it holds that $0=\lm_n\left(\nabla^2   W_{\gamma, q}^*(u^*)\right) < \lm_k\left(\nabla^2   W_{\gamma, q}^*(u^*)\right)  \text{ for all } k=1,...,n-1$. Inequality \eqref{eq:str_k}   holds due to  $\nabla^\delta_p W_\gamma(p,q) := u$ in Algorithm \ref{Alg:OnlineGD} and $\nabla_p W_\gamma(p,q) \triangleq u^*$ in \eqref{eq:nabla_wass_lagrang}. 
Multiplying both of estimates \eqref{eq:sink} and \eqref{eq:accel} by
the complexity of each iteration of the (accelerated) Sinkhorn's algorithm $\avg{O}(n^2)$ and 
\ag{the} number of  iterations $
N =\widetilde O \left( \frac{M^2}{\gamma \e} \right)$ \ag{(measures)} of Algorithm \ref{Alg:OnlineGD},
and
taking the minimum,  we get the last statement of the theorem.
\end{proof}

 Next, we study the practical convergence of projected stochastic gradient descent (Algorithm  \ref{Alg:OnlineGD}).
 Using the fact that the true Wasserstein barycenter of one-dimensional Gaussian measures 
has closed form expression for the mean and the variance \cite{delon2020wasserstein}, we study the convergence to the true barycenter of
 the generated  truncated Gaussian measures.  Figure \ref{fig:gausbarSGD} illustrates the convergence in the $2$-Wasserstein distance within 40 seconds.
\begin{figure}[ht!]
\centering
\includegraphics[width=0.45\textwidth]{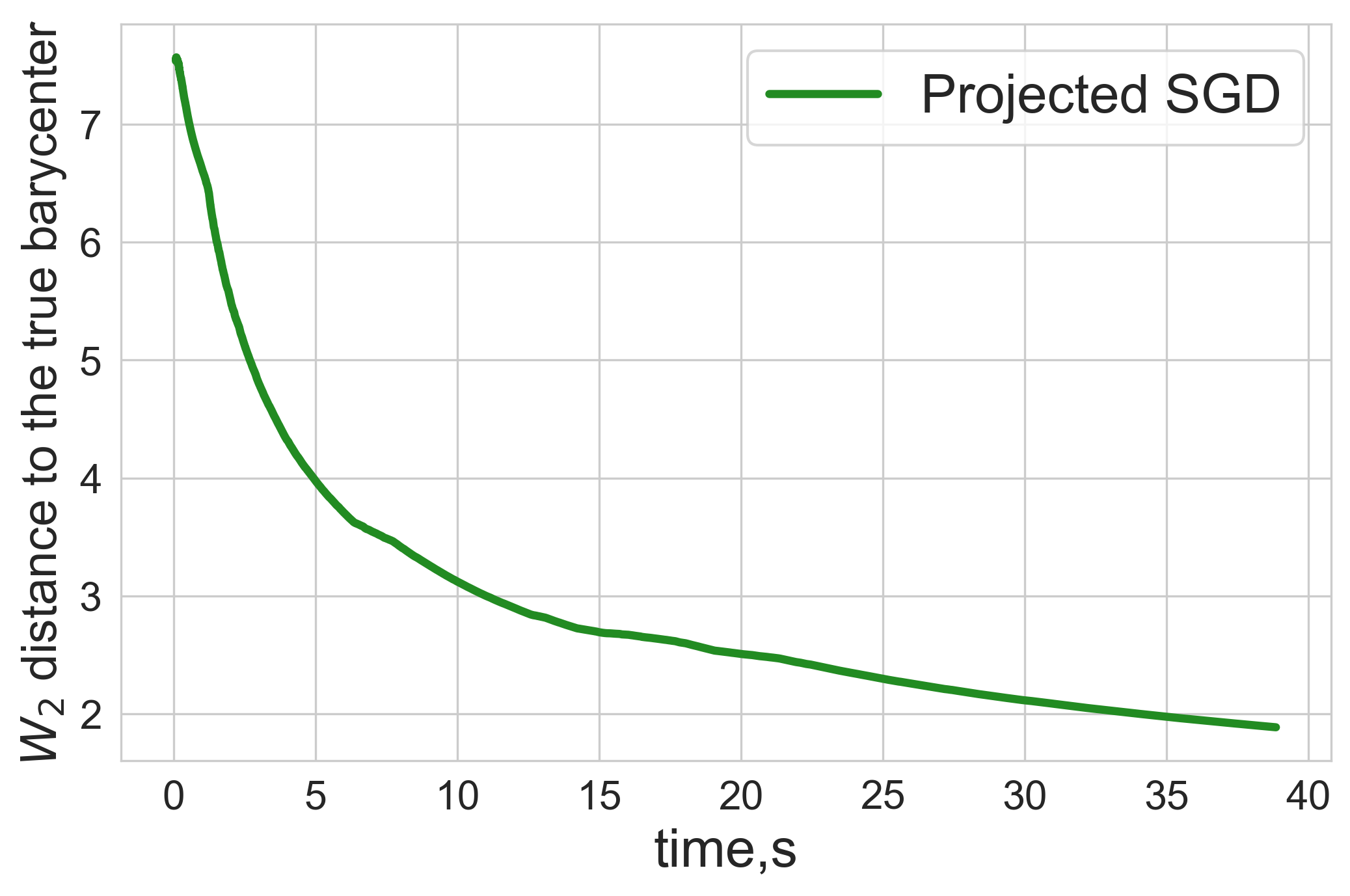}
\caption{Convergence of projected stochastic gradient descent to the true barycenter of $2\times10^4$ Gaussian measures in the $2$-Wasserstein distance. }
\label{fig:gausbarSGD}
\end{figure}

\subsection{The SAA Approach }

The empirical counterpart of problem \eqref{def:populationWBFrech} is the (empirical)  Wasserstein barycenter problem  
\begin{equation}\label{EWB_unregFrech}
\min_{p\in \Delta_n} \frac{1}{m}\sum_{i=1}^m W_\gamma(p,q_i),
\end{equation}
where  $q_1, q_2,...,q_m$ are some realizations of random variable  with distribution $\mathbb P$.

Let us define  $\hat p_\gamma^m \triangleq \arg  \min\limits_{p\in \Delta_n}{\frac{1}{m}}\sum_{i=1}^m W_\gamma(p,q_i)$  and its $\e'$-approximation $\hat p_{\e'}$ such that 
\begin{equation}\label{eq:fidelity_wass}
\frac{1}{m} \sum_{i=1}^m W_{\gamma}( \hat p_{\e'}, q_i) - \frac{1}{m} \sum_{i=1}^m W_{\gamma}(\hat p^m_{\gamma}, q_i)  \leq \e'.
\end{equation}
For instance, $\hat p_{\e'}$ can be calculated by the IBP algorithm \cite{benamou2015iterative} or the accelerated IBP algorithm \cite{guminov2019accelerated}.
The next theorem  specifies Theorem \ref{Th:contractSAA} for the Wassertein barycenter problem \eqref{EWB_unregFrech}.   
 
\begin{theorem}\label{Th:contract2}
Let $\hat p_{\e'}$ satisfies \eqref{eq:fidelity_wass}.
Then,  with probability  at least $1-\beta$  
\begin{align*}
 W_{\gamma}( \hat p_{\e'}) -  W_{\gamma}(p_{\gamma}^*)
   &\leq \sqrt{\frac{2M^2}{\gamma}\e'}  +\frac{4M^2}{\beta\gamma m},
\end{align*}
where $p^*_{\gamma} \triangleq \arg\min\limits_{p\in  \Delta_n}
W_\gamma(p)$.
Let $\e' = O \left(\frac{\e^2\gamma}{n\|C\|_\infty^2}  \right)$ and $m = O\left( \frac{M^2}{\beta \gamma \e} \right) =O\left( \frac{n\|C\|_\infty^2}{\beta \gamma \e} \right)$. Then, with probability  at least $1-\beta$  
\[W_{\gamma}( \hat p_{\e'}) -  W_{\gamma}(p_{\gamma}^*)\leq \e \quad \text{and} \quad \|\hat p_{\e'} - p^*_{\gamma}\|_2 \leq \sqrt{2\e/\gamma}.\]
The total complexity  of the accelerated IBP computing $\hat p_{\e'}$     is 
\begin{equation*}
\widetilde O\left(\frac{n^4\|C\|_\infty^4}{\beta \gamma^2\e^2} \right).
\end{equation*}
\end{theorem}

\begin{proof}
From Theorem \ref{Th:contractSAA} we get the first statement of the theorem
\[ W_{\gamma}( \hat p_{\e'}) -  W_{\gamma}(p_{\gamma}^*)
  \leq \sqrt{\frac{2M^2}{\gamma}\e'}  +\frac{4M^2}{\beta\gamma m}. \]
  From \cite{guminov2019accelerated}
  we have that complexity of the accelerated IBP is
  \[
  \widetilde O\left(\frac{mn^2\sqrt n\|C\|_\infty}{\sqrt{\gamma \e'}} \right).
  \]
  Substituting the expression for $m$ and the expression for $\e'$ from Theorem \ref{Th:contractSAA} 
   \[\e' = O \left(\frac{\e^2 \gamma}{M^2}  \right), \qquad m = O\left( \frac{M^2}{\beta \gamma \e} \right)\]
  to this equation we get the final statement of the theorem and finish the proof.
\end{proof}

 Next, we study the practical convergence of the Iterative Bregman Projections on truncated Gaussian measures.
 Figure \ref{fig:gausbarSGD} illustrates the convergence of the barycenter calculated by the IBP algorithm to the true barycenter of Gaussian measures in the $2$-Wasserstein distance within 10 seconds. For  the  convergence to the true barycenter w.r.t. the $2$-Wasserstein distance in the SAA approach, we refer to 
 \cite{boissard2015distribution}, however,  considering the  convergence  in the $\ell_2$-norm  (Theorem \ref{Th:contract2}) allows to obtain better convergence rate in comparison with the bounds for the $2$-Wasserstein distance.

\begin{figure}[ht!]
\centering
\includegraphics[width=0.5\textwidth]{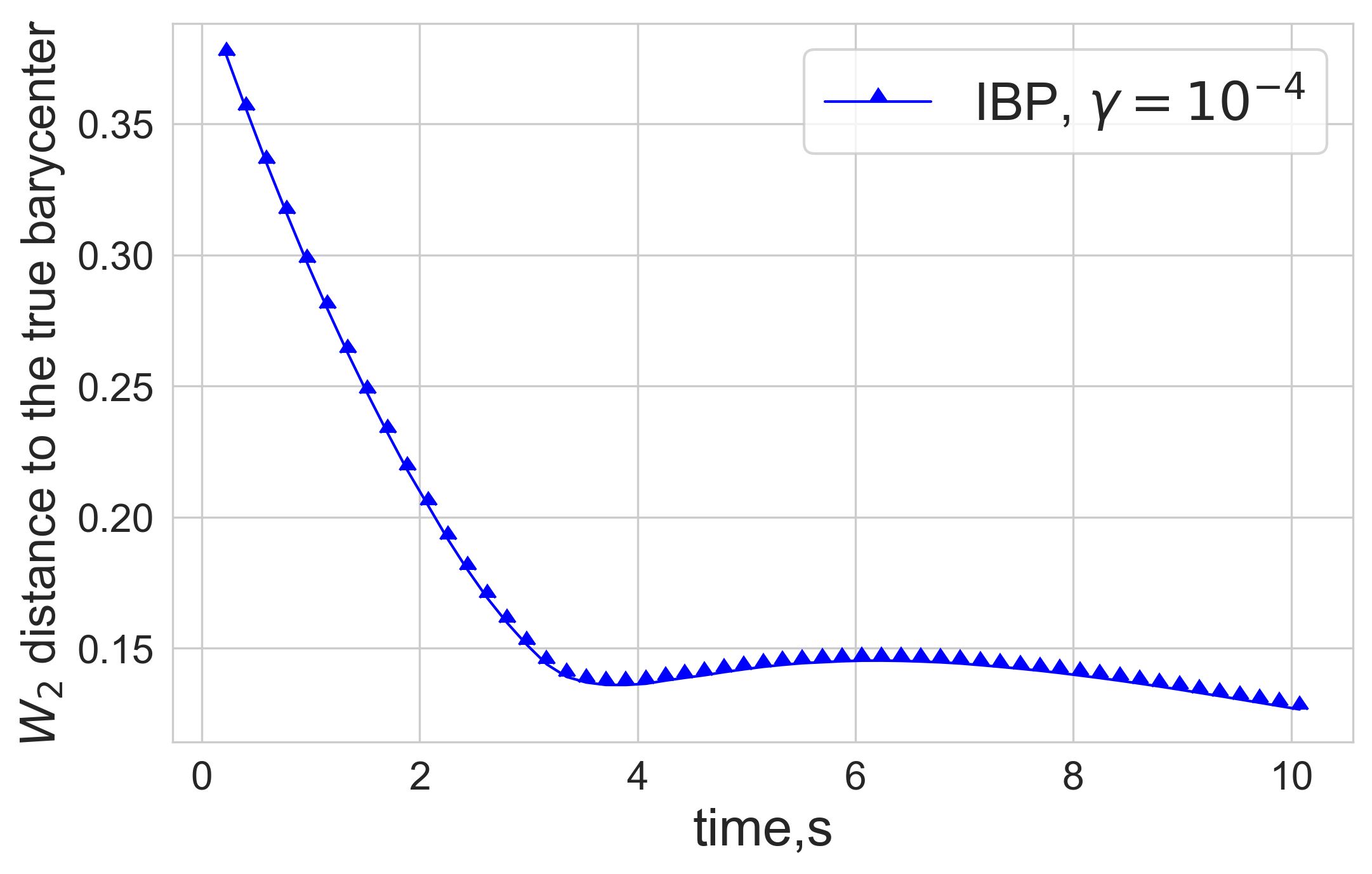}
\caption{Convergence of the Iterative Bregman Projections to the true barycenter of $2\times10^4$ Gaussian measures in the $2$-Wasserstein distance. }
\label{fig:gausbarIBP}
\end{figure}

 \subsection{Comparison of the SA and the SAA for the WB Problem}
 
Now we compare the complexity bounds for the  SA and the SAA implementations solving problem \eqref{def:populationWBFrech}.  
For the brevity, we skip the high probability details since we can fixed $\beta$ (say $\beta = 0.05$) in the all bounds. 
Moreover, based on \cite{shalev2009stochastic}, we assume that in fact  all  bounds of this paper have logarithmic dependence on  $\beta$ which is hidden in  $\widetilde{O}(\cdot)$ \cite{feldman2019high,klochkov2021stability}.

 \begin{table}[ht!]
     \caption{Total complexity of the SA and the SAA implementations for the problem $ \min\limits_{p\in \Delta_n}\E_q W_\gamma(p,q)$. }
     \small
     \hspace{-0.5cm}
{ \begin{tabular}{ll}
     \toprule
   \textbf{Algorithm}  &  \textbf{Complexity}  \\
        \midrule
          Projected SGD (SA) & $     \widetilde O\left( \frac{n^3\|C\|^2_\infty}{\gamma\e} \min\left\{ \exp\left( \frac{\|C\|_{\infty}}{\gamma} \right) \left( \frac{\|C\|_{\infty}}{\gamma} + \log\left(\frac{\|C\|_{\infty}}{\kappa \e^2} \right) \right),  \sqrt{\frac{n \|C\|^2_{\infty}}{ \kappa\gamma \e^2}} \right\} \right)$ \\
         \midrule
       Accelerated IBP (SAA) &  
       $\widetilde O\left(\frac{n^4\|C\|_{\infty}^4}{\gamma^2\varepsilon^2}\right)$\\
        \bottomrule
    \end{tabular}}
    \label{Tab:entropic_OT2}
\end{table}

  Table \ref{Tab:entropic_OT2} presents the total complexity of the numerical algorithms  implementing the SA and the SAA approaches. 
When $\gamma$ is not too \ag{large}, the complexity in the first row of the table  is achieved by the second term under the minimum, namely \[\widetilde O \left(\frac{n^3\sqrt n  \|C\|^3_\infty}{\gamma \sqrt{\gamma \kappa }\e^2}\right),\]
where   $ \kappa \triangleq \lm^+_{\min}\left(\nabla^2   W_{\gamma, q}^*(u^*)\right)$. This is \dm{typically bigger than the SAA complexity when $\kappa \ll \gamma/n$.}
 Hereby, the SAA approach may outperform the SA approach provided that the regularization parameter $\gamma$ is not too large.

From the practical point of view,  the SAA implementation  converges   much faster  than the SA implementation. 
Executing the SAA  algorithm in a distributed  manner only enhances this superiority since for the case when the objective is not Lipschitz smooth, the distributed implementation of the SA approach is not possible. This is the case of the Wasserstein barycenter problem, indeed,  the objective  is Lipschitz continuous but not Lipschitz smooth.

\section{Fr\'{e}chet Mean  with respect to  Optimal Transport}
\chaptermark{Fr\'{e}chet Mean  with respect to  Optimal Transport}

Now we are interested in finding a Fr\'{e}chet mean  with respect to  optimal transport
\begin{equation}\label{def:population_unregFrech}
\min_{p\in \Delta_n}W(p) \triangleq \E_q W (p,q).
\end{equation}

\subsection{The SA Approach with Regularization: Stochastic Gradient Descent}
The next theorem explains how the solution of strongly convex problem \eqref{def:populationWBFrech}    approximates
a solution of convex problem
\eqref{def:population_unregFrech} under the proper choice of the regularization parameter $\gamma$.

\begin{theorem}\label{Th:SAunreg}
Let $\tilde p^N \triangleq \frac{1}{N}\sum_{k=1}^{N}p^k $ be the average of  $N$ online outputs of Algorithm  \ref{Alg:OnlineGD} run with  $\delta = O\left(\e\right)$ and  $
N =\widetilde O \left( \frac{n\|C\|_\infty^2}{\gamma \e} \right)
$. Let $\gamma = \avg{{\e}/{(2 \mathcal{R}^2)}  } $ with $\mathcal{R}^2 = 2 \log n$. Then,  with probability  
at least $1-\beta$  the following holds
\begin{equation*}
 W(\tilde p^N) - W(p^*) \leq \e,
\end{equation*}
 where $p^*$ is a solution of \eqref{def:population_unregFrech}.
 
 The total complexity of 
 Algorithm  \ref{Alg:OnlineGD} with the accelerated Sinkhorn
is
\[\widetilde O \left(\frac{n^3\sqrt n  \|C\|^3_\infty}{\gamma \sqrt{\gamma \kappa }\e^2}\right)= \widetilde O \left(\frac{n^3\sqrt n  \|C\|^3_\infty}{\e^3 \sqrt{\e \kappa }}\right).\]
where $ \kappa \triangleq \lm^+_{\min}\left(\nabla^2   W_{\gamma, q}^*(u^*)\right)$. 
\end{theorem}

\begin{proof}
The proof of this theorem follows from  Theorem \ref{Th:contract} and the following \cite{gasnikov2015universal,kroshnin2019complexity,peyre2019computational}
\[
 W(p) -  W(p^*) \leq W_{\gamma}(p) - W_{\gamma}(p^*) + 2\gamma\log n \leq  W_{\gamma}(p) -  W_{\gamma}(p^*_{\gamma})+ 2\gamma\log n,
\]
where  $p \in \Delta_n $, $p^* = \arg\min\limits_{p\in \Delta_n}W(p) $. 
The choice  $\gamma = \frac{\e}{4\log n}$  ensures the following
\begin{equation*}
 W(p) -  W(p^*) \leq  W_{\gamma}(p) - W_{\gamma}(p^*_{\gamma}) + \e/2, \quad \forall p \in \Delta_n.   
\end{equation*}
This means that solving  problem \eqref{def:populationWBFrech} with $\e/2$ precision, we get a solution of problem \eqref{def:population_unregFrech} with $\e$ precision. 

 When $\gamma$ is not too \ag{large},  Algorithm  \ref{Alg:OnlineGD} uses the accelerated  Sinkhorn’s algorithm (instead of Sinkhorn’s algorithm). Thus, using $\gamma  = \frac{\e}{4\log n} $ and meaning that $\e$ is small, we get the complexity according to the statement of the theorem. 

\end{proof}

\subsection{The SA Approach: Stochastic Mirror Descent}
Now we propose an approach  to solve problem \eqref{def:population_unregFrech} without additional regularization.
The approach is based on mirror descent given by the iterative formula \eqref{eq:prox_mirr_step}.  We use simplex setup which 
provides a closed form solution for \eqref{eq:prox_mirr_step}.   Algorithm \ref{Alg:OnlineMD} presents the application of mirror descent to problem \eqref{def:population_unregFrech}, where
the gradient of $ W(p^k,q^k)$ can be calculated using dual representation of OT \cite{peyre2019computational} by any LP solver exactly
\begin{align}\label{eq:refusol}
     W(p,q) = \max_{ \substack{(u, \nu) \in \R^n\times \R^n,\\
     u_i+\nu_j \leq C_{ij}, \forall i,j \in [n]}}\left\{ \la u,p  \ra + \la \nu,q \ra  \right\}.
\end{align}
Then \[
\nabla_p    W(p,q) = u^*,
\]
where $u^*$ is a solution of \eqref{eq:refusol} such that $\la u^*,\one\ra =0$.

\begin{algorithm}[ht!]
\caption{
Stochastic Mirror Descent for the Wasserstein Barycenter Problem}
\label{Alg:OnlineMD}   
\begin{algorithmic}[1]
   \Require starting point $p^1 = (1/n,...,1/n)^T$,
   number of measures  $N$,  $q^1,...,q^N$, accuracy of gradient calculation $\delta$
   \State $\eta = \frac{\sqrt{2\log n}}{\|C\|_{\infty}\sqrt{N}}$
                \For{$k= 1,\dots, N$} 
                \State Calculate $\nabla_{p^k} W(p^k,q^k)$ solving dual LP   by any LP solver 
                \State
             \[p^{k+1} = \frac{p^{k}\odot \exp\left(-\eta\nabla_{p^k} W(p^k,q^k)\right)}{\sum_{j=1}^n [p^{k}]_j\exp\left(-\eta\left[\nabla_{p^k} W(p^k,q^k)\right]_j\right)} \]
                \EndFor
    \Ensure $\breve{p}^N = \frac{1}{N}\sum_{k=1}^{N} p^{k}$
\end{algorithmic}
 \end{algorithm}

The next theorem estimates the complexity of Algorithm \ref{Alg:OnlineMD} 

\begin{theorem}\label{Th:MD}
Let $\breve p^N$ be the output of Algorithm  \ref{Alg:OnlineMD} processing $N$ measures. Then,  with probability  
at least $1-\beta$  we have 
\begin{equation*}
  W(\breve p^N) - W({p^*})  = O\left(\frac{\|C\|_\infty \sqrt{\log ({n}/{\beta})}}{\sqrt{N}} \right),  
\end{equation*} 
Let  Algorithm \ref{Alg:OnlineMD} run with    $
N = \widetilde O \left( \frac{M_\infty^2R^2}{\e^2} \right) =\widetilde O \left( \frac{\|C\|_\infty^2}{\e^2} \right)
$, $ R^2 \triangleq {\rm KL}(p^1,p^*) \leq  \log n $.
Then,  with probability  
at least $1-\beta$  
 \[  W(\breve p^N) - W(p^*) \leq \e.\]
The {total} complexity of Algorithm \ref{Alg:OnlineMD} is
\[ \widetilde O\left( \frac{ n^3 \|C\|^2_\infty}{\e^2}\right).\]
\end{theorem}

\begin{proof}
From Theorem \ref{Th:MDgener} and using $M_\infty = O\left(\|C\|_\infty\right)$, we have
\begin{align}\label{eq:final_est234}
  W(\breve p^N) -  W(p^*)  & = O\left(\frac{\|C\|_\infty \sqrt{\log ({n}/{\beta})}}{\sqrt{N}} +2\delta  \right).
\end{align} 
Notice, that $\nabla_{p^k} W(p^k,q^k)$  can be calculated exactly by any LP solver. Thus, we take  $\delta = 0$ in \eqref{eq:final_est234} and get the first statement of the theorem.

The second statement of the theorem directly follows from this and the condition $ W(\breve p^N) - W(p^*)\leq \e$.

To get the complexity bounds we 
notice that the complexity for  calculating  $\nabla_p W(p^k,q^k)$ is $\tilde{O}(n^3)$  \cite{ahuja1993network,dadush2018friendly,dong2020study,gabow1991faster}, multiplying this by $N =  O \left( {\|C\|_\infty^2R^2}/{\e^2} \right) $ with $ R^2 \triangleq {\rm KL}(p^*,p^1) \leq \log n $, we get the last statement of the theorem.
\[\widetilde O(n^3N)  = {\widetilde O\left(n^3 \left(\frac{\|C\|_{\infty} R}{\e}\right)^2\right) =}  \widetilde O\left(n^3 \left(\frac{\|C\|_\infty}{\e}\right)^2\right).\]
\end{proof}

Next we compare  the
SA approaches with and without regularization of optimal transport in problem \eqref{def:population_unregFrech}. Entropic regularization  of optimal transport leads to strong convexity of regularized optimal transport in the $\ell_2$-norm, hence, the Euclidean setup should be used. Regularization parameter $\gamma = \frac{\e}{4 \log n}$ ensures $\e$-approximation for the unregularized solution.
In this case, we use stochastic gradient descent  with Euclidean projection onto simplex  since it converges faster for strongly convex objective. 
For non-regularized problem we can significantly use the simplex prox structure, indeed, we can apply stochastic mirror descent  with simplex setup (the Kullback-Leibler divergence as the Bregman divergence) with Lipschitz constant $M_\infty = O(\|C\|_\infty)$ that is $\sqrt{n}$ better than Lipschitz constant in the Euclidean norm $M = O(\sqrt{n}\|C\|_\infty)$.

We  studied  the convergence of stochastic mirror descent (Algorithm \ref{Alg:OnlineMD}) and stochastic gradient descent (Algorithm  \ref{Alg:OnlineGD}) in the $2$-Wasserstein distance within $10^4$ iterations (processing of $10^4$ probability measures). 
  Figure \ref{fig:gausbarcomparison1} confirms
  better convergence of stochastic mirror descent than projected stochastic gradient descent as stated in their theoretical complexity (Theorems \ref{Th:SAunreg} and \ref{Th:MD}). 
  
\begin{figure}[ht!]
\centering
\includegraphics[width=0.45\textwidth]{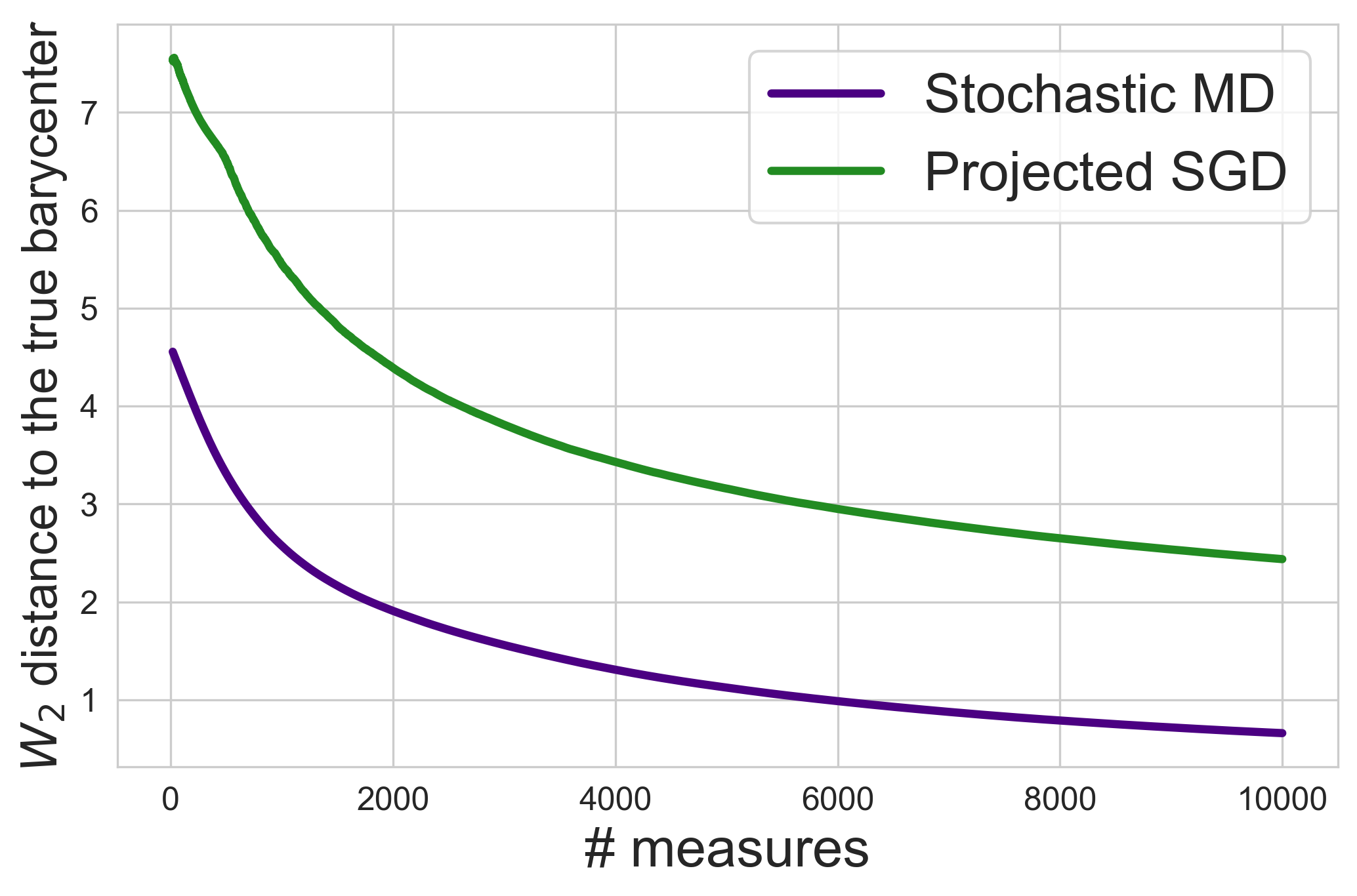}
\caption{Convergence of projected stochastic gradient descent, and stochastic mirror descent to the true barycenter of  $2\times10^4$ Gaussian measures in the $2$-Wasserstein distance. }
\label{fig:gausbarcomparison1}
\end{figure}

\subsection{The SAA Approach }

Similarly for the SA approach, we provide the proper choice of the regularization parameter $\gamma$ in  the SAA approach 
so that the 
solution of strongly convex problem \eqref{def:populationWBFrech}    approximates
a solution of convex problem
\eqref{def:population_unregFrech}.


\begin{theorem}\label{Th:SAAunreg}
Let $\hat p_{\e'}$ satisfy 
\begin{equation*}
\frac{1}{m} \sum_{i=1}^m W_{\gamma}( \hat p_{\e'}, q^i) - \frac{1}{m} \sum_{k=1}^m W_{\gamma}(\hat p^*_{\gamma}, q^i)  \leq \e',
\end{equation*}
where $ \hat p_\gamma^* = \arg\min\limits_{p\in \Delta_n}\dm{\frac{1}{m}}\sum\limits_{i=1}^m W_\gamma(p,q^i)$, $\e' = O \left(\frac{\e^2 \gamma}{n\|C\|_\infty^2}  \right)$, $m = O\left( \frac{n\|C\|_\infty^2}{\beta \gamma \e} \right)$, and  $\gamma = {\e}/{(2 \mathcal{R}^2)} $ with $\mathcal{R}^2 = 2 \log n$.
Then,  with probability  at least $1-\beta$  the following holds
\[W( \hat p_{\e'}) -  W(p^*)\leq \e.\]
The total complexity  of the accelerated IBP computing $\hat p_{\e'}$     is 
\begin{equation*}
\widetilde O\left(\frac{n^4\|C\|_\infty^4}{\beta \e^4} \right). 
\end{equation*}

\end{theorem}

\begin{proof}
The proof follows from Theorem \ref{Th:contract2} and the proof of Theorem \ref{Th:SAunreg} with $\gamma = {\e}/{(4 \log n)} $.
\end{proof}

\subsection{Penalization of the WB problem}

For  the population Wasserstein barycenter problem, we construct 1-strongly convex penalty function in the $\ell_1$-norm based on Bregman divergence.
We consider the following prox-function \cite{ben-tal2001lectures}
\[d(p) = \frac{1}{2(a-1)}\|p\|_a^2, \quad a = 1 + \frac{1}{2\log n}, \qquad p\in \Delta_n\]
 that is 1-strongly convex in the $\ell_1$-norm. Then Bregman divergence $ B_d(p,p^1)$  associated with $d(p)$  is
\[B_d(p,p^1) = d(p) - d(p^1) - \la \nabla d(p^1), p - p^1 \ra.\] 
$B_d(p,p^1)$ is 1-strongly convex w.r.t. $p$ in the $\ell_1$-norm and $\widetilde{O}(1)$-Lipschitz continuous in the $\ell_1$-norm on $\Delta_n$. One of the advantages of this penalization compared to the negative entropy penalization proposed in  \cite{ballu2020stochastic,bigot2019penalization}, is that we get the upper bound on the Lipschitz constant, the properties of strong convexity in the $\ell_1$-norm on $\Delta_n$ remain the same. Moreover, this penalization contributes to the better wall-clock time complexity than quadratic penalization \cite{bigot2019penalization} since the constants of Lipschitz continuity for $W(p,q)$ with respect to the $\ell_1$-norm is $\sqrt{n}$ better than with respect to the $\ell_2$-norm   but $R^2 = \|p^* - p^1\|_2^2\leq \|p^* - p^1\|_1^2 \leq  2 $ and $R^2_d = B_d(p^*,p^1) = O(\log n)$  are equal up to a logarithmic factor.

The regularized SAA problem is following
 \begin{equation}\label{EWB_Bregman_Reg}
 \min_{p\in \Delta_n}\left\{\frac{1}{m}\sum_{k=1}^m W(p,q^k) +  \lm B_d(p,p^1)\right\}.
\end{equation}
The next theorem is particular case of Theorem \eqref{th_reg_ERM} for the population WB problem \eqref{def:population_unregFrech} with $r(p,p^1) = B_d(p,p^1)$.
\begin{theorem}\label{Th:newreg}
Let $\hat p_{\e'}$ be  such that
 \begin{equation}\label{EWB_eps}
\frac{1}{m}\sum_{k=1}^m W(\hat p_{\e'},q^k) + \lm B_d(\hat p_{\e'},p^1)  -
\min_{p\in \Delta_n}\left\{\frac{1}{m}\sum_{k=1}^m W(p,q^k) +  \lm B_d(p,p^1)\right\}
\le \e'.
\end{equation}
To satisfy
\[ W( \hat p_{\e'}) -  W(p^*)\leq \e\]
with probability  at least $1-\beta$,
 we need to take  $\lm = \e/(2{R_d^2})$ and
 \[m = \widetilde O\left(\frac{ \|C\|_\infty^2}{\beta \e^2}\right), \]
 where
$ R_d^2 =  B_d(p^*,p^1) = {O(\log n)}$. The precision $\e'$ is defined as
\[\e' = \widetilde O\left(\frac{\e^3}{\|C\|_\infty^2 }\right).\]
The total complexity  of Mirror Prox computing $\hat p_{\e'}$  is 
\[ 
  \widetilde O\left(  \frac{ n^2\sqrt n\|C\|^5_\infty}{\e^5} \right).\]
\end{theorem}
\begin{proof}
The proof is based on saddle-point reformulation of the WB problem. 
Further, we provide the explanation how to do this (for more details see Chapter \ref{ch:WB}).
Firstly 
we  rewrite  the OT    as \cite{jambulapati2019direct}
\begin{equation}\label{eq:OTreform}
     W(p,q) = \min_{x \in \Delta_{n^2}} \max_{y\in [-1,1]^{2n}} \{d^\top x +2\|d\|_\infty(~ y^\top Ax -b^\top y)\},
\end{equation}
where  $b = 
(p^\top, q^\top)^\top$,  $d$ is vectorized cost matrix of $C$, $x$ be vectorized transport plan of  $X$,  and  $A=\{0,1\}^{2n\times n^2}$ is an incidence matrix.
Then we reformulate the WB problem as a saddle-point problem \cite{dvinskikh2020improved}
\begin{align}\label{eq:alm_distr}
         \min_{ \substack{ p \in \Delta^n, \\ \x \in \X \triangleq \underbrace{\Delta_{n^2}\times \ldots \times \Delta_{n^2}}_{m} }}    \max_{ \y \in  [-1,1]^{2mn}}
        \frac{1}{m} \left\{\boldsymbol d^\top \x +2\|d\|_\infty\left(\y^\top\boldsymbol A \x -\mathbf b^\top \y \right)\right\}, 
        \end{align}
    where
     $\x = (x_1^\top ,\ldots,x_m^\top )^\top  $,
  $\y = (y_1^\top,\ldots,y_m^\top)^\top $, 
    $\mathbf b = (p^\top, q_1^\top, ..., p^\top, q_m^\top)^\top$, 
$\boldsymbol d = (d^\top, \ldots, d^\top )^\top $, 
and    $\boldsymbol A = {\rm diag}\{A, ..., A\} \in \{0,1\}^{2mn\times mn^2}$ is  block-diagonal matrix.   
Similarly to \eqref{eq:alm_distr} we reformulate  \eqref{EWB_Bregman_Reg} as a saddle-point problem 
\begin{align*}
         \min_{ \substack{ p \in \Delta^n, \\ \x \in \X}}    \max_{ \y \in  [-1,1]^{2mn}}
         ~f_\lm(\x,p,\y) 
         &\triangleq 
        \frac{1}{m} \left\{\boldsymbol d^\top \x +2\|d\|_\infty\left(\y^\top\boldsymbol A \x -\mathbf b^\top \y \right)\right\} +\lm B_d(p,p^1).
        \end{align*}
The  gradient operator for $f(\x,p,\y)$ is  defined by
\begin{align}\label{eq:gradMPrec}
    G(\x, p, \y) = 
    \begin{pmatrix}
  \nabla_\x f \\
 \nabla_p f\\
 -\nabla_\y f 
    \end{pmatrix} = 
   \frac{1}{m} \begin{pmatrix}
   \boldsymbol d + 2\|d\|_\infty \boldsymbol A^\top \y \\
  -   2\|d\|_\infty \{[y_{i}]_{1...n}\}_{i=1}^m +\lm(\nabla d(p) - \nabla d(p^1)) \\
    2\|d\|_\infty(\boldsymbol A\x - \b)  
    \end{pmatrix},
\end{align}
where $[d(p)]_i = \frac{1}{a-1}\|p\|_a^{2-a}[p]_i^{a-1}$.

To get the complexity of MP we use the same reasons as in \cite{dvinskikh2020improved} with \eqref{eq:gradMPrec}.   The  total complexity  is
\[
\widetilde O\left(  \frac{ mn^2\sqrt{n} \|C\|_\infty}{\e'} \right)
\]
Then we use  Theorem  \ref{th_reg_ERM}
and get the exspressions for $m$, $\e'$ with $\lm = \e/(2{R_d}^2)$, where ${R_d}^2 =  B_d(p^*,p^1)$.
The number of measures is
 \[m = \frac{ 32 M_\infty^2 R_d^2}{\beta \e^2} = \widetilde O\left(\frac{\|C\|_\infty^2}{\beta \e^2}\right). \]
 The precision $\e'$ is defined as
\[\e' = \frac{\e^3}{64M_\infty^2 {R_d}^2} = \widetilde O\left(\frac{\e^3}{\|C\|_\infty^2}\right).\]

\end{proof}

  \subsection{Comparison of the SA and the SAA for the WB Problem.}
   Now we compare the complexity bounds for the  SA and the SAA implementations solving problem  \eqref{def:population_unregFrech}.   Table \ref{Tab:OT} presents the total complexity for the numerical algorithms.

 { 
   \begin{table}[H]
     \caption{Total complexity of the SA and the SAA implementations for the problem $ \min\limits_{p\in \Delta_n}\E_q W(p,q)$. }
     \begin{center}
    \begin{tabular}{lll}\toprule
  \textbf{Algorithm}  & \textbf{Theorem} & \textbf{Complexity}   \\
        \midrule
          \makecell[l]{Projected SGD (SA) \\ with $\gamma = \frac{\e}{4 \log n}$}
       & \ref{Th:SAunreg} & $ \widetilde O \left(\frac{n^3\sqrt n  \|C\|^3_\infty}{\e^3 \sqrt{\e \kappa }}\right)$ \\
                  \midrule
        Stochastic MD (SA) & \ref{Th:MD} & $\widetilde O\left( \frac{n^3\|C\|^2_{\infty}}{\e^2}\right) $  \\
         \midrule
         \makecell[l]{Accelerated IBP (SAA) \\ with $\gamma = \frac{\e}{4 \log n}$}
     &  \ref{Th:SAAunreg} &
  $\widetilde{O}\left(\frac{n^4\|C\|^4_{\infty}}{\e^4}\right)$
     \\
         \midrule
     \makecell[l]{ Mirror Prox with $B_d(p^*,p^1)$\\   penalization (SAA)} & \ref{Th:newreg} & \ag{$\widetilde O\left(\frac{ n^{2}\sqrt{n}\|C\|^{5} _{\infty}}{\e^{5} }\right)$} \\
        \bottomrule
    \end{tabular}
    \label{Tab:OT}
   \end{center}
\end{table}
  }

   For the SA algorithms, which are Stochastic MD and Projected SGD,  we can conclude the following: non-regularized approach (Stochastic MD)    uses simplex prox structure and gets better complexity bounds, indeed Lipschitz constant in the $\ell_1$-norm is $M_\infty = O(\|C\|_\infty)$, whereas Lipschitz constant in the Euclidean norm is $M = O(\sqrt{n}\|C\|_\infty)$.  The practical comparison of Stochastic MD (Algorithm \ref{Alg:OnlineMD})  and Projected SGD (Algorithm \ref{Alg:OnlineGD}) can be found in  Figure \ref{fig:gausbarcomparison1}.  
   
   For the SAA approaches (Accelerated IBP and Mirror Prox with specific penalization) we enclose the following: entropy-regularized approach (Accelerated IBP) has better dependence on $\e$ than penalized approach (Mirror Prox with specific penalization), however, worse dependence   on $n$. Using Dual Extrapolation method for the WP problem from paper \cite{dvinskikh2020improved} instead of Mirror Prox allows to omit $\sqrt{n}$ in the penalized approach.

  
  One of the main advantages of the SAA approach is the  possibility to perform it in a decentralized manner 
 in contrast to the SA approach, which cannot be executed in a decentralized manner or even in distributed  or parallel fashion for non-smooth objective  \cite{gorbunov2019optimal}. This is the case of the Wasserstein barycenter problem, indeed,  the objective  is Lipschitz continuous but not Lipschitz smooth.

\chapter{
Dual Methods for Strongly Convex Optimization}\label{ch:dual}

In this Chapter, we firstly present a stochastic dual  algorithm for an optimization problem with affine constraints whose objective is strongly convex. Then,
for the objective given by the sum of strongly convex functions, we show how to perform this algorithm in a decentralized manner over a network of agents.
This algorithm allows us to obtain   optimal bounds on the number of communication rounds and oracle calls of dual objective per node.
 Finally, we show that the results can be naturally applied to the Wasserstein barycenter problem since
 the dual formulation of entropy-regularized Wasserstein distances and their derivatives have closed-form representations and  can be computed for a cheaper price than the primal representations.

\section{Dual Problem Formulation}

We consider a convex optimization problem with affine constraint
\begin{equation}\label{eq:primal_pr_main}
    \min_{Ax=b, ~x \in\R^{n} }  F(x),
\end{equation}
where $F(x)$  is $\gamma_F$-strongly convex and possibly presented by the expectation $F(x)=\E F(x,\xi)$ w.r.t. $\xi \in \Xi$.

The  dual problem  for  problem   \eqref{eq:primal_pr_main}, written as a maximization problem, is given by the following problem with the Lagrangian dual  variable
$y \in \R^{n}  $
\begin{align}\label{eq:DualProblemmain}
\min_{ y \in \R^n} \Psi( y)  \triangleq
\max_{x \in \R^n }  \left\{ \langle y, Ax-b \rangle - F(x)\right\}.
\end{align}
By the Theorem \ref{th:primal-dualNes}, if $F(x)$ is $\gamma_F$-strongly convex, then function $\Psi(y)$ is $L_\Psi$-Lipschitz smooth with $L_\Psi = \lm_{\max}(A^\top A)/\gamma_F$,
where $\lm_{\max}(B)$ is  the maximum eigenvalue of symmetric matrix $B$.

When primal function $F(x)=\E F(x,\xi)$, the dual function is also presented by its expectation $\Psi(y) = \E\Psi(y,\xi)$ as well as its gradient. In this case we refer to stochastic dual oracle.  
For a deterministic function $F(x)$, we can also  refer to stochastic dual oracle when 
deterministic dual oracle, which returns the gradient of $\Psi$, is unavailable or very expensive.

 \subsection{Preliminaries on  Stochastic Oracle}

 We make the following assumptions on the stochastic dual oracle which returns
 the gradient of the dual objective $\Psi$ for all $y \in \R^{n}$
\begin{align}\label{eq:assup_on_oracle}
   &\E \nabla\Psi( y, \xi) = \nabla\Psi (y) \notag \\  
   &\E \exp \left( {\|\nabla\Psi( y, \xi) - \nabla\Psi (y)\|^2_2}/{\sigma_\Psi^2}\right) \leq \exp(1).
\end{align}
  We construct a stochastic approximation for $\nabla \Psi(y)$ by using batches of  size  $r$ 
        \begin{align}\label{eq:batched_estimates}
        \nabla^{r} \Psi(y, \{\xi^i\}^r_{i=1}) \triangleq \frac{1}{r}\sum_{i=1}^r \nabla \Psi(y,\xi^i).
        \end{align}
To  estimate the variance of minibatch stochastic gradient \eqref{eq:batched_estimates}, we refer to         
        \citep[Theorem 2.1]{juditsky2008large} and \citep[Lemma 2]{lan2012validation} on  large-deviations theory.
\begin{lemma} \citep[Theorem 2.1]{juditsky2008large}\label{lm:lemma1}
Let $\{D_k\}_{k=1}^N$ be a  sequence of random vectors (martingale-difference sequence)  such that for all  $k = 1,..., N,$
$\E[D_k|D_1,D_2,...,D_{k-1}] = 0 $. 
Let the sequence $\{D_k\}_{k=1}^N$ satisfies `light-tail' assumption
\[\E\left[\exp\left(\frac{\|D_k\|_2^2}{\sigma_k^2}\right)\big|\eta_1,\dots, \eta_{k-1}\right]\leq \exp(1) \quad \text{(a.s.)}, \quad k = 1,..., N.\]
Then for all $ \Omega \geq 0$
\begin{align*}
\Prob\left( \left\|\sum_{k=1}^N c_k D_k\right\|_2 \geq (\sqrt{2} + \sqrt{2}\Omega) \sqrt{\sum_{k=1}^N c^2_k\sigma_k^2} \right) \leq \exp\left(\frac{-\Omega^2}{3}\right),
\end{align*}
where $c_1, \dots, c_N$ are positive numbers.
\end{lemma}

\begin{lemma} \citep[Lemma 2]{lan2012validation}\label{lm:lemma2}
Let for all $k =1,\dots N$, $D_k = D_k(\{\eta_l\}_{l=1}^{k})$ be a deterministic function of i.i.d. realizations  $\{\eta_l\}_{l=1}^{k}$  such that 
\[\E\left[\exp\left(\frac{D_k^2}{\sigma_k^2}\right) \big|\eta_1,\dots, \eta_{k-1}\right]\leq \exp(1) \quad \text{(a.s.)}, \quad k = 1,..., N.\]
Then for all $ \Omega \geq 0$
\begin{align*}
\Prob\left( \sum_{k=1}^N c_k D^2_k \geq (1 + \Omega) \sum_{k=1}^N c_k\sigma_k^2 \right) \leq \exp\left({-\Omega}\right),
\end{align*}
where $c_1, \dots, c_N$ are positive numbers.
\end{lemma}
Now we use  these two lemmas to estimate the variance of the mini-batch stochastic gradient \eqref{eq:batched_estimates}. The next lemma gives exact constant for the reduced sub-Gaussian variance of the mini-batch gradient.
\begin{lemma}[Sub-Gaussian variance reduction]\label{lm:sigma_est}
Let stochastic gradient $\nabla \Psi(y, \xi)$ satisfies the following conditions
\begin{align*}
   &\E \nabla\Psi(y, \xi) = \nabla\Psi (y) \notag \\  &\E \exp \left( {\|\nabla\Psi(y, \xi) - \nabla\Psi (y)\|^2_2}/{\sigma_\Psi^2}\right) \leq \exp(1).
\end{align*}
Then, for the minibatch gradient 
$
        \nabla^{r} \Psi(y, \{\xi^i\}^r_{i=1}) = \frac{1}{r}\sum_{i=1}^r \nabla \Psi(y,\xi^i)
$ with batch size $r$,
         the following holds with $\hat\sigma_\Psi^2 = 50\sigma_\Psi^2/r$
\begin{align*}
&\E \nabla^{r} \Psi(y, \{\xi^i\}_{i=1}^{r}) = \nabla\Psi (y) \notag \\ 
&\E \exp\left({\|\nabla^{r} \Psi(y, \{\xi^i\}_{i=1}^{r}) - \nabla \Psi(y)\|_2^2}/{\hat\sigma_\Psi^2}\right) \leq  \exp(1), 
\end{align*}
\end{lemma} 
\begin{proof}
Lemma  \ref{lm:lemma1} allows us to write for any $\Omega\geq 0$ the following
\begin{align*}
\Prob\left( \left\|\sum_{i=1}^{r}\frac{1}{r}\nabla \Psi(y,\xi^i) - \nabla \Psi(y)\right\|_2 \geq (\sqrt 2+ \sqrt 2 \Omega)  \frac{\sigma_\Psi}{\sqrt{r}} \right) \leq \exp\left(\frac{-\Omega^2}{3}\right).
\end{align*}
Here we used $D_i = \nabla \Psi(y,\xi^i) - \nabla \Psi(y)$ as a martingale difference, $c_i = \frac{1}{r}$, and $\sigma_i^2 = \sigma_\Psi^2$ for all $i=1,...,r$.

For $\Omega \geq \frac{1}{3}$ let us rewrite the previous bound as follows
\begin{align}\label{eq:prob_stoch_grad}
\Prob\left( \left\|\sum_{i=1}^{r}\frac{1}{r}\nabla \Psi(y,\xi^i) - \nabla \Psi(y)\right\|_2 \geq  4\sqrt 2 \Omega \frac{\sigma_\Psi}{\sqrt{r}} \right) \leq \exp\left(\frac{-\Omega^2}{3}\right).
\end{align}
Next we estimate $\hat\sigma_\Psi^2$ 
 \begin{align}\label{eq:prob}
&\E \exp \left( {\|\nabla^{r} \Psi(y, \{\xi^i\}_{i=1}^{r}) - \nabla \Psi(\y)\|^2_2}/{\hat\sigma_\Psi^2}\right) \notag\\
&\triangleq \int_{0}^\infty \Prob\left(\exp\left({\|\nabla^{r} \Psi(y, \{\xi^i\}_{i=1}^{r})  - \nabla \Psi(y)\|^2_2}/{\hat\sigma_\Psi^2}\right) \geq x\right)dx \notag \\
&\leq  \int_{0}^{\exp\left(\frac{32\sigma^2_\Psi}{9r\hat \sigma^2_\Psi}\right)} \Prob\left(\exp\left({\|\nabla^{r} \Psi(y, \{\xi^i\}_{i=1}^{r})  - \nabla \Psi(y)\|^2_2}/{\hat\sigma_\Psi^2}\right) \geq x\right)dx \notag \\
&+ \int_{\frac{4\sqrt{2} \sigma_\Psi}{3\sqrt r}}^\infty \Prob\left(\|\nabla^{r} \Psi(y, \{\xi^i\}_{i=1}^{r})  - \nabla \Psi(y)\|_2 \geq z\right)\frac{2z}{\hat\sigma_\Psi^2}\exp\left(\frac{z^2}{\hat\sigma_\Psi^2}\right)dz\notag \\
&\leq 
\exp\left(\frac{32 \sigma^2_\Psi}{9r\hat \sigma^2_\Psi}\right)\notag \\
&+ \int_{\frac{4\sqrt{2} \sigma_\Psi}{3\sqrt r}}^\infty \Prob\left(\|\nabla^{r} \Psi(y, \{\xi^i\}_{i=1}^{r})  - \nabla \Psi(y)\|_2 \geq z\right)\frac{2z}{\hat\sigma_\Psi^2}\exp\left(\frac{z^2}{\hat\sigma_\Psi^2}\right)dz,
\end{align}
where we used $\Prob(\cdot)\leq 1$ and the change of variable $z^2 = \hat \sigma_\Psi^2\ln x$ for $x\in \left[\exp\left(\frac{32 \sigma^2_\Psi}{9r\hat \sigma^2_\Psi}\right), \infty\right)$. Making the following change of variable $z = 4\sqrt{2} \Omega  \frac{\sigma_\Psi}{\sqrt{r}}$ in \eqref{eq:prob}
 from \eqref{eq:prob_stoch_grad}   we have for any $\Omega \geq \frac{1}{3}$
\begin{align*}
&\E \exp\left({\|\nabla^{r} \Psi(y, \{\xi^i\}_{i=1}^{r}) - \nabla \Psi(y)\|_2^2}/{\hat\sigma_\Psi^2}\right)  \notag \\
&\leq \exp\left(\frac{32 \sigma^2_\Psi}{9r\hat \sigma^2_\Psi}\right)+ \int_{\frac{1}{3}}^\infty \exp \left(\frac{-\Omega^2}{3}\right) 8\sqrt 2 \Omega  \frac{\sigma^2_\Psi}{{r} \hat \sigma^2_\Psi}\exp \left(\frac{32\Omega^2\sigma_\Psi^2}{r \hat\sigma_\Psi^2 }\right)d\Omega \notag\\
&= \exp\left(\frac{32 \sigma^2_\Psi}{9r\hat \sigma^2_\Psi}\right)+ \frac{\sigma^2_\Psi}{ r \hat\sigma_\Psi^2}\int_{\frac{1}{3}}^\infty 8 \sqrt 2\Omega \exp\left(\frac{32\Omega^2\sigma_\Psi^2 - \Omega^2 r \hat\sigma_\Psi^2 }{3 r \hat\sigma_\Psi^2 }\right) d\Omega\\
&=
\exp\left(\frac{32 \sigma^2_\Psi}{9r\hat \sigma^2_\Psi}\right)+ \frac{4 \sqrt 2 \sigma^2_\Psi}{ r \hat\sigma_\Psi^2}\int_{\frac{1}{9}}^\infty  \exp\left(-\Omega^2\frac{  r \hat\sigma_\Psi^2 -32\sigma_\Psi^2  }{3 r \hat\sigma_\Psi^2 }\right) d(\Omega^2)\\
&=\exp\left(\frac{32 \sigma^2_\Psi}{9r\hat \sigma^2_\Psi}\right) + 
\frac{12 \sqrt 2 \sigma^2_\Psi }{  r \hat\sigma_\Psi^2 -32\sigma_\Psi^2  }
\exp\left(-\frac{  r \hat\sigma_\Psi^2 -32\sigma_\Psi^2  }{27 r \hat\sigma_\Psi^2 }\right)
,
\end{align*}
where we used  $\int_{1/9}^\infty e^{-ax}dx = \frac{1}{a}e^{-a/9}$ for any $a>0$. If we take $\hat\sigma_\Psi^2 =50\sigma^2_\Psi/r$ in the last inequality  we get the statement of the theorem
\begin{align*}
&\E\exp\left({\|\nabla^{r} \Psi(y, \{\xi^i\}_{i=1}^{r_k}) - \nabla \Psi(y)\|_2^2}/{\hat\sigma_\Psi^2}\right)  \notag \\
&\leq \exp(16/(9*25))+\frac{6\sqrt 2}{9}\exp(-18/27)\sim 1.6 \leq \exp(1). 
\end{align*}
\end{proof}

\subsection{Algorithm and  Convergence Rate}

Now we propose an algorithm (Algorithm  \ref{Alg:NDDualStochAlg})  to solve the pair of problems \eqref{eq:primal_pr_main} and \eqref{eq:DualProblemmain}. The algorithm is an accelerated version of the gradient descent method. The next theorem studies its convergence.

  \begin{algorithm}[t]
\caption{ Dual Stochastic  Accelerated Gradient  Algorithm}
\label{Alg:NDDualStochAlg}          
 \begin{algorithmic}[1]
   \Require Starting point $\lm^0 = \eta^0=\zeta^0= 0$, number of iterations $N$, $A_0=\alpha_0=0$, $L_\Psi = {\lm_{\max}(A^\top A)}/{\gamma_F}.$
                \For{$k=0,\dots, N-1$}
                \State
     \begin{equation}\label{eq:Alg_const}
                   \quad A_{k+1} = A_{k} + \alpha_{k+1}= 2 L_\Psi\alpha_{k+1}^2.
                   \end{equation}
                   \vspace{-0.7cm}
               \State
                \begin{equation}\label{eq:Alg_lambda}
                \lm^{k+1} = (\alpha_{k+1}\zeta^k + A_k \eta^k)/{A_{k+1}}.
                \end{equation}
                \vspace{-0.5cm}
               \State 
                Calculate $\nabla^{r_{k+1}} \Psi(\lm^{k+1},\{\xi_i\}_{i=1}^{r_{k+1}})$ according to \eqref{eq:batched_estimates} with batch size
                \begin{equation}\label{eq:batch_size_r}
                r_{k+1} = \max \left\{ 1,  50\sigma_\Psi^2 {\alpha}_{k+1} \log(N/\alpha)/{\e} \right\}. \end{equation}
                \vspace{-0.7cm}
               \State
             \vspace{-0.5cm}
             \begin{eqnarray}\label{eq:Alg_zeta}
                \zeta^{k+1}= \zeta^{k} - \alpha_{k+1} \nabla^{r _{k+1}} \Psi(\lm^{k+1},\{\xi^\ell\}_{\ell=1}^{r_{k+1}}).
                \end{eqnarray}
                \vspace{-0.7cm}
               \State 
                \vspace{-0.5cm}
                \begin{eqnarray}\label{eq:Alg_y}
               \eta^{k+1} =(\alpha_{k+1}\zeta^{k+1} + A_k \eta^k)/{A_{k+1}}.
              \end{eqnarray}
              \vspace{-0.7cm}
                \EndFor
    \Ensure  $ x^{N} \triangleq \frac{1}{A_{N}}\sum_{k=0}^{N} \alpha_k x(\lm^k,\{\xi^\ell\}_{\ell=1}^{r_k})$, where
    \[x(\lm^k,\{\xi^\ell\}_{\ell=1}^{r_k}) \triangleq \frac{1}{r_k}\sum_{\ell=1}^{r_k}  x(\lm^k,\xi^\ell) = \nabla^{r _{k}} \Psi(\lm^{k},\{\xi^\ell\}_{\ell=1}^{r_{k}})\] 
\end{algorithmic}
\end{algorithm}

\begin{theorem}\label{Th:stoch_err}
Let $F(x)$ be $\gamma_F$-strongly convex. Let $R_\lm$ be such that $\|\lm^*\|_2\leq R_\lm$, where $\lm^*$ is an exact solution of dual problem \eqref{eq:DualProblemmain}.   
Then, after  $N = O\left(\sqrt{\frac{L_\Psi  R_\lm^2}{\e}} \right)$ iterations, the output $x^N$ of Algorithm \ref{Alg:NDDualStochAlg} satisfies the following  with probability at least $ 1-\alpha$ 
\begin{align}
    F(x^N) - F(x^*) \leq \e, \quad \|Ax^N-b\|_2 \leq {\e}/{R_\lm},
\end{align}
where $L_\Psi = \lm_{\max}(A^\top A)/\gamma_F$. The number of dual oracle calls of $\nabla\Psi(\lm,\xi)$  is
\[
 \widetilde O\left( \max\left\{ \sqrt{\frac{L_\Psi {R_\lm}^2}{\e}}, \frac{\sigma_\Psi^2 {R_\lm}^2 }{\e^2} \right\} \right),
\]
where $\sigma_\Psi^2$ is sub-Gaussian variance of $ \nabla \Psi(\lm,\xi)$.
\end{theorem}

\begin{proof}[Sketch of the Proof]
Let us  define
 the set $B_{2R_\lm}(0) = \{\lm: \|\lm\|_2\leq 2R_\lm\}$.
From  \citep[Theorem 1]{dvurechensky2018decentralize}  it follows that  Algorithm \ref{Alg:NDDualStochAlg} generates the sequences $\{\lm^N, \zeta^N, y^N,\alpha^N, A^N \}$ satisfying
 \begin{align} \label{eq:ASGDConv}
    A_N\Psi(\eta^N) 
    &\leq \min_{\lm \in B_{2R_\lm}(0)}\left\{  \sum_{k=0}^N \alpha_{k} \left( \Psi(\lm^{k}) + \la \nabla^{r_k} \Psi(\lm^k, \{\xi^i\}_{i=1}^{r_k}), \lm - \lm^{k}\ra \right) \right\} + 2R^2_\lm\notag \\ 
 &+ \sum_{k=0}^{N-1}A_{k+1}\la \nabla^{r_{k+1}} \Psi(\lm^{k+1}, \{\xi^i\}_{i=1}^{r_{k+1}}) - \nabla \Psi(\lm^{k+1}), \eta^{k} - \lm^{k+1}\ra \notag\\
 &+ \sum_{k=0}^{N}\frac{A_{k}}{2L_\Psi}\|\nabla^{r_k} \Psi(\lm^k, \{\xi^i\}_{i=1}^{r_k})- \nabla \Psi(\lm^{k})\|_{2}^2.
    \end{align}
We denote the stochastic terms in \eqref{eq:ASGDConv}  as follows 
\begin{enumerate}
    \item $H_1 =\min\limits_{\lm\in B_{2R_\lm}(0)} \left\{ \sum_{k=0}^N \alpha_{k} \left( \Psi(\Blm^{k}) + \la\nabla^{r_k} \Psi(\lm^k, \{\xi^i\}_{i=1}^{r_k}), \lm - \lm^{k}\ra \right)  \right\}$,
    \item $H_2 = \sum_{k=0}^{N-1}A_{k+1}\la \nabla^{r_k} \Psi(\lm^{k+1}, \{\xi^i\}_{i=1}^{r_k}) - \nabla \Psi(\lm^{k+1}), \eta^{k} - \Blm^{k+1}\ra$,
    \item $H_3 = \sum_{k=0}^{N}\frac{A_{k}}{2L_\Psi}\|\nabla^{r_k} \Psi(\lm^k, \{\xi^i\}_{i=1}^{r_k})- \nabla \Psi(\lm^{k})\|_{2}^2.$
\end{enumerate}
By adding and subtracting $\sum_{k=0}^{N-1}\alpha_{k+1}\la \nabla\Psi(\lm^{k+1}), \lm^* - \lm^{k+1}\ra$ under the minimum in $H_1$ we get
\begin{align}\label{eq:H1}
    H_1
    &= \min_{ \lm\in B_{2R_\lm}(0)}  \left\{\sum_{k=0}^N \alpha_{k} \left( \Psi(\lm^{k}) + \la \nabla^{r_k} \Psi(\lm^k, \{\xi^i\}_{i=1}^{r_k}) - \nabla \Psi (\lm^{k}), \lm - \lm^{k}\ra + \la  \nabla \Psi (\lm^{k}), \lm - \lm^{k}\ra\right) \right\}\notag  \\
     &\leq \min_{ \lm\in B_{2R_\lm}(0)} \left\{  \sum_{k=0}^N \alpha_{k} \left( \Psi(\lm_{k}) + \la \nabla \Psi (\lm^{k}), \lm - \lm^{k}\ra \right)\right\} \notag\\
     &+  \max_{\lm\in B_{2R_\lm}(0)} \left\{ \sum_{k=0}^N \alpha_k\la \nabla^{r_k} \Psi(\lm^k, \{\xi^i\}_{i=1}^{r_k}) - \nabla \Psi(\lm^k), \lm \ra  \right\} \notag\\
     &+ \sum_{k=0}^N \alpha_k\la \nabla \Psi(\lm^k)-\nabla^{r_k} \Psi(\lm^k, \{\xi^i\}_{i=1}^{r_k}) , \lm^k \ra \notag \\
    &\leq \min_{ \lm\in B_{2R_\lm}(0)}  \left\{\sum_{k=0}^N \alpha_k (\Psi(\lm^k)+ \la\nabla\Psi(\lm^k), \lm - \lm^k\ra) \right\} \notag\\
    &+ 2R_\lm\|\sum_{k=0}^N\alpha_k(\nabla^r \Psi(\lm^k, \{\xi^i\}_{i=1}^r) - \nabla\Psi(\lm^k))\|_2 \notag\\
 &+ \sum_{k=0}^N\alpha_k\la  \nabla \Psi(\lm^k) -\nabla^{r_k} \Psi(\lm^k, \{\xi^i\}_{i=1}^{r_k}) , \lm^k \ra.
 \end{align}
 We denote the terms in \eqref{eq:H1} as follows 
\begin{enumerate}
    \item $H_4 = 2R_\lm\|\sum_{k=0}^N\alpha_k(\nabla^{r_k} \Psi(\lm^k, \{\xi^i\}_{i=1}^{r_k}) - \nabla\Psi(\lm^k))\|_2$,
    \item $H_5 = \sum_{k=0}^N\alpha_k\la  \nabla \Psi(\lm^k) -\nabla^{r_k} \Psi(\lm^k, \{\xi^i\}_{i=1}^{r_k}) , \lm^k \ra$.
\end{enumerate}
 By Cauchy–Schwarz inequality for $H_6 = H_2+H_5$ we have
\begin{align}\label{eq:large_h5}
    H_6 
    &\leq \sum_{k=0}^{N-1}  \|\nabla^{r_{k+1}} \Psi(\lm^{k+1}, \{\xi^i\}_{i=1}^{r_k})  - \nabla \Psi(\lm^{k+1})\|_2\|A_{k+1}\eta^k - A_{k+1}\lm^{k+1}- \alpha_{k+1}\lm^{k+1}\|_2  \\
    &\overset{\eqref{eq:Alg_lambda},\eqref{eq:Alg_const}}{=} \sum_{k=0}^{N-1}\alpha_{k+1}\|\nabla^{r_{k+1}} \Psi(\lm^{k+1}, \{\xi^i\}_{i=1}^{r_k}) - \nabla \Psi(\lm^{k+1})\|_2\|\eta^k - \Bzeta^{k}-\lm^{k+1}\|_2\notag\\
    &\leq 
    3\mathcal{R}\sum_{k=0}^{N-1}\alpha_{k+1}\|\nabla^{r_{k+1}} \Psi(\lm^{k+1}, \{\xi^i\}_{i=1}^{r_k}) - \nabla \Psi(\lm^{k+1})\|_2,
\end{align}
where $\mathcal{R} \geq \max\{\|\eta^k\|_2, \|\lm^{k+1}\|_2, \|\zeta^k\|_2, 2R_\lm\}$ for all $k=1,\dots, N$.\\

For $H_4$, $H_6$ we will use Lemma \ref{lm:lemma1}, and for $H_3$ we will refer to Lemma \ref{lm:lemma2}. 
We also will use Lemma \ref{lm:sigma_est} to estimate $\hat\sigma_\Psi^2 = 50\sigma_\Psi^2/r$.

Now we use Lemma \ref{lm:lemma1} for $H_6$.  We take $D_k = \nabla^{r_k} \Psi(\lm^k, \{\xi^i\}_{i=1}^{r_k}) - \nabla \Psi(\lm^k)$, $c_k =3\mathcal{R} \alpha_k$ and $\sigma_k^2 = \hat\sigma^2_\Psi = 50\sigma_\Psi^2/r_k$. Therefore, we get
\begin{align}\label{eq:H6}
&\Prob\left( H_6 \geq 3\mathcal{R}(\sqrt{2} + \sqrt{2}\Omega) \sqrt{\sum_{k=1}^N 50\alpha_k^2 \sigma_\Psi^2/r_k} \right)\notag \\
&=\Prob\left( H_6 \geq 30\mathcal{R}\sigma_\Psi(1 + \Omega) \sqrt{\sum_{k=1}^N \alpha_k^2 /r_k} \right)\leq \exp\left(\frac{-\Omega^2}{3}\right),
\end{align}
For $H_4$ we also use Lemma \ref{lm:lemma1}. We take $D_k = \nabla^{r_k} \Psi(\lm^k, \{\xi^i\}_{i=1}^{r_k}) - \nabla \Psi(\lm^k)$, $c_k = 2R_\lm\alpha_k$ and $\sigma_k^2 = \hat\sigma^2_\Psi = 50\sigma_\Psi^2/r_k$.
\begin{align}\label{eq:H4}
&\Prob\left( H_4 \geq 2R_\lm(\sqrt{2} + \sqrt{2}\Omega) \sqrt{\sum_{k=1}^N 50\alpha_k^2 \sigma_\Psi^2/r_k} \right)\notag \\
&=\Prob\left( H_4 \geq  20R_\lm\sigma_\Psi(1 + \Omega) \sqrt{\sum_{k=1}^N \alpha_k^2 /r_k} \right)\leq \exp\left(\frac{-\Omega^2}{3}\right),
\end{align}
Now we use Lemma \ref{lm:lemma2} for $H_3$. We take $D_k = \|\nabla^{r_k} \Psi(\lm^k, \{\xi^i\}_{i=1}^{r_k})- \nabla \Psi(\lm^{k})\|_{2}$,  $c_k = \frac{A_{k}}{2L_\Psi}$,  and $\sigma_k^2 = \hat \sigma^2_\Psi = 50 \sigma^2_\Psi/r_k$ for all $k=1,..., N$ and get the following
\begin{align*}
&\Prob\left( H_3\geq 25(1 + \Omega)\sigma^2_\Psi \sum_{k=1}^N \frac{A_k}{L_\Psi r_k} \right)= \Prob\left( H_3 \geq 50\sigma^2_\Psi(1 + \Omega) \sum_{k=1}^N \frac{\alpha^2_k}{r_k} \right)
\leq \exp\left({-\Omega}\right).
\end{align*}
We can equivalently rewrite it as follows
\begin{align}\label{eq:H3}
\Prob\left( H_3 \geq 50\sigma^2_\Psi(1 + \Omega^2/3) \sum_{k=1}^N \alpha_k^2 \right)
\leq \exp\left({-\frac{\Omega^2}{3}}\right).
\end{align}
Next we again consider \eqref{eq:ASGDConv}
\begin{align}\label{eq:A_N_H} 
    A_N\Psi(\eta^N) 
    &\leq  H_1 + 2R_\lm^2+H_2 +H_3\notag \\
    &\leq \min_{ \lm\in B_{2R_\lm}(0)}  \left\{\sum_{k=0}^N \alpha_k (\Psi(\lm^k)+ \la\nabla\Psi(\lm^k), \lm - \lm^k\ra)\right\} + 2R_\lm^2\notag \\
   & + H_2 +H_3+H_4 +  H_5. 
\end{align}
Next we estimate  the r.h.s of \eqref{eq:A_N_H}. We consider
\begin{align}\label{eq:pr_dual_Lend_of_the_proof_1}
 \frac{1}{A_N}  \min_{ \lm\in B_{2R_\lm}(0)}  \left\{\sum_{k=0}^N \alpha_k (\Psi(\lm^k)+ \la\nabla\Psi(\lm^k), \lm - \lm^k\ra)\right\}. 
\end{align}
By the definition of the dual function and by the Demyanov--Danskin theorem we have
\[\Psi(\lm) =\la \lm,~  Ax(\lm)-b\ra - F(x(\lm)) \quad \text{ and } \quad \nabla \Psi(\lm) = Ax(\lm)-b,\]
where
$x(\lm) = \arg\max\limits_{x\in\R^{n}}\left\{\la\lm, ~ Ax-b\ra - F(x)\right\}$. Using this in \eqref{eq:pr_dual_Lend_of_the_proof_1} we get
\begin{align}\label{eq:mini_est}
& \frac{1}{A_N}\min_{ \lm\in B_{2R_\lm}(0)}  \left\{\sum_{k=0}^N \alpha_k \left(\la \lm^k,~ A x(\lm^k)-b\ra -F(x(\lm^k))+ \la Ax(\lm^k)-b, ~\lm - \lm^k\ra\right)\right\}\notag \\
&=- \frac{1}{A_N}\sum_{k=0}^N \alpha_k F(x(\lm^k))  +  \min_{ \lm\in B_{2R_\lm}(0)}  \left\{\frac{1}{A_N} \sum_{k=0}^N \alpha_k\la A x(\lm^k)-b,~ \lm \ra\right\}\notag\\
&\leq -  F(\hat x^N) - \max_{ \lm\in B_{2R_\lm}(0)} \left\{\la A \hat x^N-b,~ \lm \ra\right\} = -  F(\hat x^N) - 2R_\lm\|A \hat x^N-b\|_2.
\end{align}
where we used  $\hat x^N \triangleq \frac{1}{A_N}\sum_{k=0}^N \alpha_k x(\lm^k)$.
Then we estimate the rest terms of the r.h.s. of \eqref{eq:A_N_H}. From the union bound applied for \eqref{eq:H3}, and \eqref{eq:H4}, \eqref{eq:H6} and making the change $\alpha = \exp\left(-\frac{\Omega^2}{3}\right)$ we have with probability $\geq 1 - 3\alpha$ 
\begin{align*}
   & H_3+ H_4 +H_6 \leq \notag\\
    &50\sigma^2_\Psi(1 + \ln(1/\alpha)) \sum_{k=1}^N \alpha_k^2/r_k +20R_\lm\sigma_\Psi(1 + \sqrt{3\ln(1/\alpha)}) \sqrt{\sum_{k=1}^N \alpha_k^2 /r_k} \notag\\
    &+30\mathcal{R}\sigma_\Psi(1 + \sqrt{3\ln(1/\alpha)}) \sqrt{\sum_{k=1}^N \alpha_k^2 /r_k}.
\end{align*}
By the definition of $\mathcal{R} \geq 2R_\lm$ we have
\begin{align*}
    & H_3+ H_4 +H_6 \leq \notag\\
   &50\sigma^2_\Psi(1 + \ln(1/\alpha)) \sum_{k=1}^N \alpha^2_k/r_k +40\mathcal{R}\sigma_\Psi(1 + \sqrt{3\ln(1/\alpha)}) \sqrt{\sum_{k=1}^N \alpha_k^2 /r_k}. 
\end{align*}
By the definition of $r$ \eqref{eq:batch_size_r} we get
\begin{align}\label{eq:H346}
    & H_3+ H_4 +H_6  \notag\\
   &\leq 50\sigma^2_\Psi(1 + \ln(1/\alpha)) \sum_{k=1}^N \frac{\e \alpha_k}{\sigma_\Psi^2 \ln({N}/{\delta})} +40\mathcal{R}\sigma_\Psi(1 + \sqrt{3\ln(1/\alpha)}) \sqrt{\sum_{k=1}^N \frac{\e \alpha_k}{\sigma_\Psi^2 \ln({N}/{\delta})}} \notag\\
   &= 50(1 + \ln(1/\alpha))  \frac{\e A_N}{ \ln({N}/{\delta})} +40\mathcal{R}(1 + \sqrt{3\ln(1/\alpha)}) \sqrt{ \frac{\e A_N}{ \ln({N}/{\delta})}} ,
\end{align}
where we used $A_N = \sum_{k=1}^N\alpha_k$ from \eqref{eq:Alg_const} in the last equality.
We sum up \eqref{eq:mini_est} and \eqref{eq:H346} we rewrite \eqref{eq:A_N_H} and divide it by $A_N$. We get with probability $\geq 1 - 3 \alpha$ the following
\begin{align}\label{eq:A_N_H2}
&\Psi(\eta^N) +  F(\hat x^N) +  2R_\lm \|A \hat x^N-b\|_2
       \leq  \frac{2R_\lm^2}{A_N} \notag \\
+&50(1 + \ln(1/\alpha))  \frac{\e }{ \ln({N}/{\delta})} +40\mathcal{R}(1 + \sqrt{3\ln(1/\alpha)}) \sqrt{ \frac{\e }{A_N \ln({N}/{\delta})}}
\end{align}
The next steps are to prove  $\mathcal R =O\left(R_\lm\right)$ and transfer from the   $\hat x^N$ to  the output of the Algorithm \ref{Alg:NDDualStochAlg}, that is $x^N$, by using   large deviation bounds. This can be found in 
the paper \cite{gorbunov2019optimal},

\end{proof}

\section{Decentralized Optimization}
\paragraph{Background on Distributed Optimization.} 
A distributed system is a system of  computing nodes (agents, machines, processing units), whose interactions  are  constrained by the system structure.
In distributed computing, a problem is divided into many tasks, assigned to different agents.
The  agents cooperatively solve the global task by solving their local problems  and transferring information (usually, a vector) to other nodes.

Distributed optimization has recently   gained  increased interest due to
large-scale problems  encountered in machine learning. Usually these problems aggregate enormous data and they need to be  solved in a reasonable time  with no prohibitive expenses. It can also occur that the data itself is stored or collected in a distributed manner (e.g., sensors in a sensor network obtained the state of the environment from different geographical parts, or  micro-satellites collecting local information). In both these settings, distributed systems can be used. They process faster and  more data than one computer since the work is divided  between many computing nodes. The application of  distributed systems includes  formation control of unmanned vehicle \cite{ren2006consensus}, power system control \cite{ram2009distributed}, information processing and decision making in sensor networks,  distributed averaging, statistical inference and learning \cite{nedic2017fast}.

There are two scenarios of distributed optimization: centralized and decentralized. In centralized optimization, there is a central node (master) which coordinates the work of other nodes (slaves). Parallel architecture is a special case of the centralized architecture as  it always contains  master node.  
Unfortunately, centralized architecture has a synchronization drawback and a high requirement for the master node \cite{scaman2017optimal}. To address these disadvantages to some extent, \dm{a} decentralized distributed architecture \dm{should be used} \cite{bertsekas1989parallel,kibardin1979decomposition}. 
In decentralized scenario, there is no particular  node,  all agents are equivalent  and  their communications 
are constrained only by a network arhcitecture: 
each agent can communicate  only with its immediate neighbors.  
This decentralized setting is more robust since  decentralized algorithm  does not crash when one of computing node fails. 
Moreover, decentralized computing can be preformed on   time-varying (wireless) communication networks.




A large number of distributed algorithms have been developed to minimize an objective  given in the form of the average of functions $f_i$'s accessible by different nodes (agents, computers) in a network 
Thus, we consider the following convex optimization problem
\begin{equation}\label{eq:sum_randeq} 
 \min_{x\in \R^n} f(x) \triangleq \frac{1}{m} \sum_{i=1}^m f_i(x),
\end{equation}
where  $f_i(x)$'s  are $\gamma$-strongly convex and possibly presented by the expectation $f_i(x)=\E f_i(x,\xi)$ w.r.t. $\xi \in \Xi$.

\subsection{Decentralized Dual Problem Formulation}
To solve \eqref{eq:sum_randeq}  on  a network of agents, a transition to its dual problem is used. For this,
we introduce artificial constraint $x_1=x_2=\dots=x_m$ to  \eqref{eq:sum_randeq}  and rewrite it as follows 
\begin{equation}\label{eq:distr2}
     \min_{\substack{ x_1=...=x_m, \\ x_1,\dots, x_m \in \R^n,  }}      F(\x) \triangleq \frac{1}{m} \sum_{i=1}^m f_i(x_i), 
\end{equation}
 where   $\x = (x_1^\top, x_2^\top, ..., x_n^\top)^\top$ is the stack column vector.
Further, we will  replace the constraint $x_1=\dots=x_m$  with affine constraints representing the network structure. 

\paragraph{Network system.} Let 
a network of  $m$ nodes (agents, computing units) be presented by a fixed connected undirected graph $G= (V,E)$, where $V$ is a set of $m$ nodes, and $E = \{(i,j): i,j \in V\}$ is a set of edges. 
The network structure imposes communication constraints: agent $i$ can communicate (exchange information) only with its immediate neighbors (i.e.,  with agent $~ j\in V$ such that $(i,j)\in E$. 

Let us also define  
 a symmetric  and positive semi-definite  matrix $W \in \R^{m\times m}$, which
will represent a network structure. We define this matrix 
by the Laplacian matrix of the graph $G$. The elements of $W$ are  presented as
\begin{align*}
[{W}]_{ij} = \begin{cases}
-1,  & \text{if } (i,j) \in E,\\
\text{deg}(i), &\text{if } i= j, \\
0,  & \text{otherwise,}
\end{cases}
\end{align*}
 where ${\rm deg}(i)$ is the degree of vertex $i$ (i.e., the number of neighboring nodes). 
 
 Let us further define matrix 
 \begin{equation}\label{eq:matrixWdef}
     \WW \triangleq W \otimes I_n,
 \end{equation}
  where $\otimes$ is the Kronecker product and $I_n$ is the identity matrix. Matrix $\WW$ inherits the properties of $ W$, including the symmetry and positive semi-definiteness. Furthermore,
the vector $\boldsymbol{1}$ is the unique (up to a scaling factor) eigenvector of $\WW$ associated with the eigenvalue
$\lambda=0$.  Thus, the equality constraint
$x_1 = \cdots = x_m$ 
 is equivalent to affine constraint $\WW\x = 0$. Moreover,  the following identity holds \cite{scaman2017optimal}
 \[x_1 = \cdots = x_m \quad \Longleftrightarrow \quad \WW\x =0 \Longleftrightarrow \sqrt{\WW}\x = 0.\]
 Thus, the problem \eqref{eq:distr2} can be rewritten as optimization problem with affine constraints 
\begin{equation}\label{eq:distr2main}
     \min_{\substack{  \sqrt \WW \x, \\ x_1,\dots, x_m \in \R^n }}      F(\x) \triangleq \frac{1}{m}\sum_{i=1}^m f_i(x_i). 
\end{equation}
The  dual problem  for  problem   \eqref{eq:distr2} (written as a maximization problem) is given by the following problem with the Lagrangian dual  variable
$\y \in \R^{mn}$
\begin{align}\label{eq:DualPr}
\min_{\y \in \R^{mn}} \Psi(\sqrt{\WW}\y)
&\triangleq \max_{\x \in\R^{mn} }  \left\lbrace  \langle \y, \sqrt{\WW}\x\rangle -F(\x)\right\rbrace
= \max_{\x \in\R^{mn} }  \left\lbrace  \langle \y, \sqrt{\WW}\x\rangle - \frac{1}{m}\sum_{i=1}^m f_i(x_i)\right\rbrace 
\notag\\
&=  \frac{1}{m}\max_{\x \in\R^{mn} }  \sum_{i=1}^m \left\lbrace  m\langle y_i, [\sqrt{\WW}\x]_i\rangle -  f_i(x_i)\right\rbrace 
= \frac{1}{m}\sum_{i=1}^m\psi_i\left(m[\sqrt{\WW}\y]_i \right),
\end{align}
where each $
    \psi_i(\lm_i) = \max\limits_{x_i \in\R^{n} } \left\{ \la  \lm_i,x_i \ra - f_i(x_i) \right\}$
is the Fenchel--Legendre transform of  $f_i(x_i)$ and  the
vector $[\sqrt{\WW}\x]_i$ represents the $i$-th $n$-dimensional block of $\sqrt{\WW}\x$. 

By Theorem \ref{th:primal-dualNes}, if  $ F(\x)$ is $\gamma_F$-strongly convex, then  $\Psi(\sqrt{\WW}\y)$ is $L_\Psi$-Lipschitz smooth with $L_\Psi = \lm_{\max}(W)/\gamma_F$, where $\gamma_F = \gamma/m$.

By Demyanov--Danskin theorem \cite{dem1990introduction,danskin2012theory}  we have
\begin{align}\label{eq:demyanov_dan}
    \nabla \Psi(\sqrt{\WW}\y) = \sqrt{\WW}\x(\sqrt{\WW}\y),  
\end{align}
where 
$
    \x(\sqrt{\WW}\y) = \arg\max\limits_{\x \in \R^{mn}}  \left\{ \la  \x, \sqrt  \WW\y\ra - F(\x) \right\}
$.\\

  We construct a stochastic approximation for $\nabla \Psi(\y)$ by using batches of  size  $r$ 
        \begin{align}\label{eq:batched_estimate2243}
        \nabla^{r} \Psi(\y, \{\boldsymbol \xi^\ell\}^r_{\ell=1}) \triangleq \frac{1}{r}\sum_{\ell=1}^r \nabla \Psi(\y, \boldsymbol \xi^\ell).
        \end{align}
With the change of variable  $  \bar \y := \sqrt{\WW}\y $, this can be rewritten as
        \begin{align}
        \nabla^{r} \Psi(\sqrt \WW\y, \{\boldsymbol \xi^\ell\}^r_{\ell=1}) = \sqrt \WW  \nabla^{r} \Psi( \bar \y, \{\boldsymbol\xi^\ell\}^r_{\ell=1}) = \frac{1}{r}\sum_{\ell=1}^r \sqrt \WW \nabla \Psi(\bar \y,\boldsymbol\xi^\ell).  
        \end{align}
If  each $\nabla \psi_i(\bar{y}_i, \xi_i)$ has sub-Gaussian variance
\begin{equation*}
\E \exp\left({ \|\nabla \psi_i(\bar{y}_i, \xi_i)- \nabla \psi_i(\bar{y}_i)\|_2^2}/{\sigma_{\psi}^{2}}\right) \le \exp(1).
\end{equation*}
 Then $\nabla \Psi(\sqrt \WW\y, \boldsymbol\xi)$   has sub-Gaussian variance with  $\sigma_{\ga{\Psi}}^2 = O\left( \ga{\lambda_{\max}(W)m\sigma_{\psi}^2}\right)$ (Lemma \ref{lm:sigma_subgaus}).

\begin{lemma}\label{lm:sigma_subgaus}
Let each $\nabla \psi_i(\bar{y}_i,\xi_i)$ ($i=1,...,m$) has $\sigma_\psi^2$ sub-Gaussian variance
\begin{align*}
&\E \nabla \psi_i(\bar{y}_i, \xi_i)= \nabla \psi_i(\bar{y}_i),\\
&\E \exp\left( \|\nabla \psi_i(\bar{y}_i, \xi_i)- \nabla \psi_i(\bar{y}_i)\|_2^2/{\sigma_{\psi}^{2}}\right) \le \exp(1).
\end{align*}
Then $\nabla \Psi(\sqrt{\WW}\y, \boldsymbol\xi)$ has  $\sigma_\Psi^2 = O\left( \lm_{\max}(W)m\sigma_\psi^2\right)$ sub-Gaussian variance, where 
\begin{align*}
 \Psi(\sqrt{\WW}\y) =  \frac{1}{m}\sum_{i=1}^m  \psi_i\left(m[\sqrt{\WW}\y]_i \right).
\end{align*}
\end{lemma}
\noindent \textit{Sketch of the Proof.} 
We provide the proof of this lemma for  variance $\sigma_\Psi^2$ (non-sub-Gaussian). Let
\begin{align*}
   \E\|\nabla \Psi(\sqrt{\WW}\y, \boldsymbol\xi) - \E\nabla \Psi(\sqrt{\WW}\y,\boldsymbol\xi)\|^2 \leq   \sigma_\Psi^2.  
\end{align*}
Then we estimate  $\nabla \Psi(\sqrt{\WW}\y,\boldsymbol\xi)$ 
\begin{align*}
\nabla \Psi(\sqrt{\WW}\y,\boldsymbol\xi) = \sqrt \WW \nabla \Psi(\bar \y,\boldsymbol\xi)
= \sqrt \WW  \cdot \frac{1}{m}\cdot m
\begin{pmatrix}
&\nabla \psi_1\left(\bar{y}_1,\xi_1 \right)\\
&\hspace{0.5cm}\vdots\\
&\nabla \psi_m\left(\bar{y}_m,\xi_m \right),
\end{pmatrix}
\end{align*}
where $\bar \y = \sqrt{\WW}\y$. Then 
\begin{align*}
&\|\nabla \Psi(\sqrt{\WW}\y,\boldsymbol\xi) - \E\nabla \Psi(\sqrt{\WW}\y,\boldsymbol\xi) \|^2_2 
= \|\sqrt \WW \nabla \Psi(\bar \y,\boldsymbol\xi) - \sqrt \WW \E\nabla \Psi(\bar \y,\boldsymbol\xi)\|_2^2 \\
&=  
\left\la \begin{pmatrix}
&\nabla \psi_1\left(\bar{y}_1,\xi_1 \right) - \E\nabla \psi_1\left(\bar{y}_1,\xi_1 \right)\\
&\hspace{0.5cm}\vdots\\
&\nabla \psi_m\left(\bar{y}_m,\xi_m \right)- \E\nabla \psi_m\left(\bar{y}_m,\xi_m \right)
\end{pmatrix}, 
~\WW \begin{pmatrix}
&\nabla \psi_1\left(\bar{y}_1,\xi_1 \right)- \E\nabla \psi_1\left(\bar{y}_1,\xi_1 \right)\\
&\hspace{0.5cm}\vdots\\
&\nabla \psi_m\left(\bar{y}_m,\xi_m \right) - \E\nabla \psi_m\left(\bar{y}_m,\xi_m \right)
\end{pmatrix} \right\ra \\
&\leq  \lm_{\max}(W)  
\left\| 
\begin{pmatrix}
&\nabla \psi_1\left(\bar{y}_1,\xi_1 \right) - \E\nabla \psi_1\left(\bar{y}_1,\xi_1 \right)\\
&\hspace{0.5cm}\vdots\\
&\nabla \psi_m\left(\bar{y}_m,\xi_m \right)
-\E\nabla \psi_m\left(\bar{y}_m,\xi_m \right)
\end{pmatrix}
\right\|_2^2.
\end{align*}
Taking the expectation  we obtain
\begin{align*}
\sigma^2_\Psi 
&\leq
\E\|\nabla \Psi(\sqrt{\WW}\y,\boldsymbol\xi) - \E\nabla \Psi(\sqrt{\WW}\y,\boldsymbol\xi)\|^2_2 \\
&\leq  \lm_{\max}(W)  
\E\left\| 
\begin{pmatrix}
&\nabla \psi_1\left(\bar{y}_1,\xi_1 \right) - \E\nabla \psi_1\left(\bar{y}_1,\xi_1 \right)\\
&\hspace{0.5cm}\vdots\\
&\nabla \psi_m\left(\bar{y}_m, \xi_m \right) - \E\nabla \psi_m\left(\bar{y}_m,\xi_m \right)
\end{pmatrix}
\right\|_2^2 \\
&\leq  \lm_{\max}(W)  m \sigma_\psi^2.
\end{align*}
More precise  proof with sub-Gaussian variance can be performed similarly to the proof of Lemma \ref{lm:sigma_est}. 
 \hfill$ \square$

The optimization problem \eqref{eq:DualPr} is convex  unconstrained optimization problem
and can be solved by gradient-type algorithms.
If the gradient of $\Psi$ is $L_\Psi$-Lipschitz continuous, then the gradient descent method does not provide optimal estimates in contradistinction to its accelerated version
 \cite{nesterov2004introduction}. 
 However, for the clarity we explain how  problem \eqref{eq:DualPr} can be solved in a decentralized manner using the gradient descent method in the following example.
 
 \begin{example}
The iterative procedure of the gradient descent algorithm for problem \eqref{eq:DualPr} is  presented as follows ($k=0,1,2,..., N$)
\begin{equation}\label{eq:grad_step_fisrt}
\y^{k+1}= \y^{k}- \frac{1}{L_\Psi}\nabla\Psi(\sqrt{\WW}\y^k) \overset{\eqref{eq:demyanov_dan}}{=} \y^{k} -\frac{1}{L_\Psi}\sqrt{\WW}\x(\sqrt{\WW}\y^k). \end{equation}
Without change of variable, it is unclear how to perform  this  procedure in a distributed manner. Let  $  \bar \y := \sqrt{W}\y $, then the gradient step \eqref{eq:grad_step_fisrt} multiplied by $\sqrt \WW$ can be rewritten as 
\[\bar \y^{k+1} = \bar\y^{k} -\frac{1}{L_\Psi}\WW\x(\bar\y^{k}).\]
This procedure can be performed in a decentralized manner on a network of agents. Namely each agent $i = 1,..., m$ calculates 
\begin{equation*}
    \bar y_i^{k+1} = \bar y_i^{k} -\frac{1}{L_\Psi}[\WW\x(\bar\y^{k})]_i = \bar y_i^{k} -\frac{1}{L_\Psi} \sum_{j=1}^n \WW_{ij}x_j(\bar  y_j^k).
\end{equation*}

Multiplication $\WW\x$ naturally defines communications in the network because  the elements of matrix  $\WW_{ij}$
\begin{align*}
{\WW}_{ij} = \begin{cases}
-I_{n\times n},  & \text{if } (i,j) \in E,\\
\text{deg}(i)I_{n\times n}, &\text{if } i= j, \\
0_{n\times n},  & \text{otherwise,}
\end{cases}
\end{align*}
are non-zero only for neighboring nodes  $i,j$, and  \[[\x(\bar \y^k)]_j = \arg\max\limits_{\x \in \R^{mn}}  \left\{ \la  x_j, \bar y^k_j\ra - f(x_j) \right\} = x_j(\bar y^k_j).\]  
 \end{example}

Similarly, to the gradient descent method, we can apply its accelerated version (Algorithm \ref{Alg:NDDualStochAlg}) in a decentralized manner. The decentralized version of Algorithm \ref{Alg:NDDualStochAlg} with change of variables
\[
\bar \Beta = \sqrt \WW \Beta, \quad \bar \Blm =\sqrt \WW \Blm, \quad \bar \Bzeta = \sqrt \WW \Bzeta
\]
is presented in Algorithm \ref{Alg:DualStochAlg}

\subsection{Algorithm and  Convergence Rate}\label{subsec:dualapproachstoch}

  \begin{algorithm}[t]
\caption{Decentralized Dual Stochastic  Accelerated Gradient  Algorithm}
\label{Alg:DualStochAlg}          
 \begin{algorithmic}[1]
   \Require Starting point $\bar \Blm^0 = \bar \Beta^0=\bar \Bzeta^0= 0$, number of iterations $N$, $A_0=\alpha_0=0$,
   \State For each agent $i\in V~ (i=1,...,m)$
    \For{$k=0,\dots, N-1$}
        \State $A_{k+1} = A_{k} + \alpha_{k+1}= 2L_\Psi\alpha_{k+1}^2. $
        \State $\bar\lm_i^{k+1} = (\alpha_{k+1}\bar\zeta_i^k + A_k \bar \eta_i^k)/{A_{k+1}}.$
        \State 
                Calculate $\nabla^{r _{k+1}} \psi_i(\bar\lm_i^{k+1},\{\xi_i^\ell\}_{\ell=1}^{r_{k+1}})$ from \eqref{eq:DualPr} according to \eqref{eq:batched_estimate2243} with mini-batch size 
\[r_{k+1} =\max \left\{ 1,  50\sigma_\Psi^2 {\alpha}_{k+1} \ln(N/\alpha)/{\e} \right\}, \]
where $\sigma_\Psi^2 = O\left(\lm_{\max}(W)m \sigma_\psi^2\right)$
               \State
$\bar \zeta_i^{k+1}= \bar\zeta_i^{k} - \alpha_{k+1} \sum_{j=1}^m \WW_{ij}\nabla^{r _{k+1}} \psi_j(\bar\lm_j^{k+1},\{\xi^\ell_j\}_{\ell=1}^{r_{k+1}}).$
 \State 
            $
              \bar \eta_i^{k+1} =(\alpha_{k+1}\bar\zeta_i^{k+1} + A_k \bar\eta_i^k)/{A_{k+1}}.
              $
                \EndFor
        \Ensure  $\boldsymbol x^{N} = (\tilde x_1^\top, \dots, \tilde x_m^\top)^\top $, where     $ \tilde x_i \triangleq \frac{1}{A_{N}}\sum_{k=0}^{N} \alpha_k x_i(\bar \lm_i^k,\{\xi^\ell_i\}_{\ell=1}^{r_k})$ for all $i=1,\dots,m$ with
    \[ x_i(\bar \lm_i^k,\{\xi^\ell_i\}_{\ell=1}^{r_k}) \triangleq \frac{1}{r_k}\sum_{\ell=1}^{r_k}  x_i(\bar \lm_i^k,\xi_i^\ell) = \nabla^{r _{k}} \psi_i(\bar \lm_i^{k},\{\xi_i^\ell\}_{\ell=1}^{r_{k}}).\]
\end{algorithmic}
\end{algorithm}

The next theorem is a decentralized variant of Theorem \ref{Th:stoch_err} for particular case of matrix $A = \sqrt \WW$ and $b=0$ together with the fact $\mathcal R =O\left(R_\lm\right)$ from  \cite{gorbunov2019optimal}.
 Let $\chi(W) = \frac{\lm_{\max}(W)}{\lm_{\min}^+(W)}$ be the   condition number
of matrix $W$.

  \begin{theorem}\label{Th:stoch_errdecentr}
Let $f_i(x)$'s  be $\gamma$-strongly convex functions.
 Let $R_\lm$ be such that $\|\lm^*\|_2 \leq R_\lm$, where $\lm^*$ is an exact solution of dual problem \eqref{eq:DualPr}.  
Let for all $i=1,...,m$, $\|\nabla f_i(x^*)\|_2\leq M$, where $x^*$ is the solution of \eqref{eq:distr2main}.
Then, after  $N = O\left(\sqrt{\frac{ M^2 }{\gamma \e}\chi(W)} \right)$ iterations, the output $\x^N$ of Algorithm \ref{Alg:DualStochAlg} satisfies the following  with probability at least $ 1-{3}\alpha$ 
\begin{align*}
    F(\x^N) - F(\x^*) \leq \e, \quad \|\sqrt \WW \x^N\|_2 \leq {\e}/{R_\lm}.
\end{align*}
The number of dual oracle calls of $\nabla\psi_i(\lm_i,\xi_i)$  is
\[
\widetilde O\left( \max\left\{ \sqrt{\frac{ M^2 }{\gamma \e}\chi(W)}, \frac{ M^2 \sigma_\psi^2}{\e^2} \chi(W) \right\} \right),
\]
where $\sigma_\psi^2$ is sub-Gaussian variance of $\nabla \psi_i(\lm_i,\xi_i)$.
\end{theorem}     
\begin{proof}
Using the fact $\mathcal R =O\left(R_\lm\right)$ proved in \cite{gorbunov2019optimal}, we improve the number of iterations from Theorem \ref{Th:stoch_err} as follows
\begin{align}\label{eqLNdefece}
  N =O\left(\sqrt{\frac{L_\Psi R_\lm^2}{\e}}\right),  
\end{align}
where $L_\Psi = \lm_{\max}(W)/\gamma_F$ is the constant of Lipschitz smoothness for $\Psi(\sqrt{\WW\y})$,  and $\gamma_F = \gamma/m$.
Then we use  \cite{lan2017communication} to estimate the radius of the dual solution $\lm^*$ (corresponding to the minimal Euclidean distance if there are more than one solution)
\begin{align}\label{eqKlfdfdfdfd}
   \|\lm^*\|^2_2 \leq R^2_\lm 
   &= \frac{\|\nabla F(\x^*)\|_2^2}{\lm^+_{\min}(W)} \leq \frac{ 
   \left\| \frac{1}{m}
\begin{pmatrix}
&\nabla f_1(x^*)\\
&\hspace{0.5cm}\vdots\\
&\nabla f_m(x^*)\\
\end{pmatrix}
\right\|_2^2
   }{\lm^+_{\min}(W)}  =\frac{\sum_{i=1}^m \|\nabla f_i(x^*)\|_2^2}{m^2\lm^+_{\min}(W)} \notag \\
   &\leq \frac{M^2}{m \lm^+_{\min}(W)}, 
\end{align}
where  $\lm^+_{\min}(W)$ is  the  minimal
non-zero eigenvalue of matrix $W$.
Then using  $L_\Psi = \lm_{\max}(W)/\gamma_F$, $\gamma_F = \gamma/m$ and \eqref{eqKlfdfdfdfd} in \eqref{eqLNdefece} 
we get
\[
N = O\left(\sqrt{\frac{ M^2 }{\gamma \e}\chi(W)} \right).\]
The number of dual oracle calls of $\nabla\psi_i(\lm_i,\xi_i)$  is (Theorem \ref{Th:stoch_err})
\begin{align*}
 &\widetilde O\left( \max\left\{ N, \frac{\sigma_\Psi^2 R_\lm^2 }{\e^2} \right\} \right) \\
 &\overset{\eqref{eqKlfdfdfdfd}}{=}  \widetilde O\left( \max\left\{ \sqrt{\frac{ M^2_F }{\gamma \e}\chi(W)}, \frac{\lm_{\max}(W)m\sigma_\psi^2  }{\e^2} \cdot \frac{ M^2}{m\lm^+_{\min}(W)}\right\} \right)\\
 &= \widetilde O\left( \max\left\{ \sqrt{\frac{ M^2 }{\gamma \e}\chi(W)}, \frac{ M^2 \sigma_\psi^2}{\e^2} \chi(W) \right\} \right),
\end{align*}
where we used  $\sigma_\Psi^2 = O(\lm_{\max}(W) m\sigma_\psi^2)$ (Lemma \ref{lm:sigma_subgaus}) and $\chi(W) = \frac{\lm_{\max}(W)}{\lm^+_{\min}(W)}$ is   the condition number
of matrix $W$.

\end{proof}

\section{Wasserstein Barycenter Problem}\label{sec:problem}
    
    In this section,
    we apply the results stated above in a broad sense to the  Wasserstein barycenter problem defined with respect to entropy-regularized optimal transport
\begin{equation}\label{eq:W_bary_reg}
   \min_{p \in \Delta_n}  \frac{1}{m} \sum_{i=1}^m {W}_\gamma(p,q_i),
\end{equation}
where ${W}_\gamma(p,q_i)$ is $\gamma$-strongly convex w.r.t $p$ in the $\ell_2$-norm.

\subsection[Decentralized Dual Formulation]{Decentralized Dual Formulation}

To state the Wasserstein barycenter problem \eqref{eq:W_bary_reg} in a decentralized manner, we rewrite it as follows
\begin{equation}\label{eq:W_bary_reg22}
   \min_{\substack{p_1=...=p_m, \\p_1,...,p_m \in \Delta_n} } \frac{1}{m} \sum_{i=1}^m {W}_\gamma(p_i,q_i) = \min_{\substack{\sqrt{\WW} \p=0, \\p_1,...,p_m \in \Delta_n} } \frac{1}{m} \sum_{i=1}^m {W}_\gamma(p_i,q_i),
\end{equation}
where $\p = (p_1^\top,...,p_m^\top)^\top$ is the column vector and $\WW$ is defined in  \eqref{eq:matrixWdef}.
The dual problem to \eqref{eq:W_bary_reg22} is
\begin{equation}\label{eq:W_bary_regdual}
  \min_{\y \in \R^{nm}} W^*_{\gamma,\q}(\sqrt{\WW}\y)\triangleq  \frac{1}{m}\sum_{i=1}^{m} {W}^*_{\gamma, q_i}(m[\sqrt{\WW}\y]_i),
\end{equation}
where  $\y = (y_1^\top, ..., y_m^\top)^\top \in \R^{nm}$ is the Lagrangian dual multiplier, $\q =  (q_1^\top, ..., q_m^\top)^\top \in \R^{nm}$ and 
\begin{align}\label{eq:WassFenchLeg}
W^*_{\gamma,\q}(\sqrt{\WW}\y)
&\triangleq \max_{p_1,...p_m \in \Delta_n}  \left\lbrace \left\langle \sqrt{\WW}\y, \p \right\rangle -\frac{1}{m}\sum_{i=1}^m{W}_{\gamma}(p_i, q_i) \right\rbrace\\
& = \frac{1}{m}\sum_{i=1}^m \max_{p_i \in \Delta_n}  \left\lbrace \left\langle m[\sqrt{\WW}\y]_i, p_i \right\rangle -{W}_{\gamma}(p_i, q_i) \right\rbrace =\frac{1}{m}\sum_{i=1}^m{W}^*_{\gamma,q_i}(m[\sqrt{\WW}{\y}]_i),
\end{align}

\paragraph{Recovery of the Primal Solution. }
By Demyanov--Danskin theorem \cite{dem1990introduction,danskin2012theory} and from the definition of dual funtion for Wasserstein distances \eqref{eq:FenchLegdef}, we have
\begin{align}\label{eq:demyanov_danprecovery}
    \nabla W^*_{\gamma,q}(\lm) = p(\lm), 
\end{align}
where \eqref{eq:cuturi_primal}
\begin{equation}\label{eq:primal_sol_recov}
\forall l =1,...,n \qquad [p (\lm)]_l = \sum_{j=1}^n [q]_j \frac{\exp\left(([\lm]_l-C_{lj})/\gamma\right)  }{\sum_{i=1}^n\exp\left(([\lm]_i-C_{ji})/\gamma\right)}.
\end{equation}

In papers \cite{uribe2018distributed,dvinskikh2019primal}  a dual distributed algorithm for the Wasserstein barycenter problem was proposed. This algorithm is a deterministic version of Algorithm \ref{Alg:DualStochAlg}. The next theorem states its  convergence.
\begin{theorem}\citep[Corollary 6]{dvinskikh2019primal}
After $
    N = \widetilde O\left(\sqrt{\frac{ n\|C\|^2_\infty }{ \gamma \varepsilon} \chi(W) }\right)$
iterations, the output of  $\tilde \p = (\tilde p_1^T,\cdots,\tilde p_m^T)^T$  of distributed accelerated gradient method with the primal solution recovery \eqref{eq:primal_sol_recov} satisfies
\begin{align*}
\frac{1}{m}\sum_{i=1}^m W_{\gamma}(\tilde p_i, q_i)- 
\frac{1}{m}\sum_{i=1}^m  W_{\gamma}(p^*, q_i) \leq \e, \qquad \|\sqrt{\WW}\tilde \p\|_2 \leq  \e/R_{\y}.
\end{align*}
 The total per node complexity  is    
 \begin{align*}
\widetilde{O}\left(n^2\sqrt{\frac{ n\|C\|^2_\infty  }{ \gamma\varepsilon} \chi(W)} \right).
    \end{align*}
\end{theorem}

\subsection[Decentralized Dual Stochastic  Algorithm]{Decentralized Dual Stochastic  Algorithm }
The complexity of dual oracle call for the gradient  of the dual function for entropy-regularized   optimal transport \eqref{eq:demyanov_danprecovery} is $O(n^2)$. Using randomize technique, we can reduce it 
 to $O(n)$. 
To do so, we
randomize the true gradient \eqref{eq:primal_sol_recov} by taking component $j$ with probability $[q]_j$
\[
[\nabla W_{\gamma,q}^*(\lm,\xi)]_l =  \frac{\exp\left(([\lm]_l-C_{l \xi})/\gamma\right)  }{\sum_{\ell=1}^n\exp\left(([\lm]_\ell-C_{\ell\xi })/\gamma\right)}, \qquad \forall l =1,...,n. 
\]

  \begin{algorithm}[ht!]
\caption{Decentralized Dual Stochastic  Accelerated Gradient  Algorithm for WB's}
\label{Alg:DualWass}          
 \begin{algorithmic}[1]
   \Require Starting point $\bar \Blm^0 = \bar \Beta^0=\bar \Bzeta^0=\x^0= 0$, number of iterations $N$, $A_0=\alpha_0=0$,
   \State For each agent $i\in V~ (i=1,...,m)$
                \For{$k=0,\dots, N-1$}
                \State $ A_{k+1} = A_{k} + \alpha_{k+1}= 2L_\Psi\alpha_{k+1}^2.$
               \State$
                \bar\lm_i^{k+1} = (\alpha_{k+1}\bar\zeta_i^k + A_k \bar \eta_i^k)/{A_{k+1}}.$
               \State 
                For each $i=1,...,m$, calculate $\nabla^{r_{k+1}} W^*_{\gamma,q_i}(\bar\lm_i^{k+1}, \{\xi^\ell_i\}^{r_{k+1}}_{\ell=1})$
                \begin{align}
[  \nabla^{r_{k+1}} W^*_{\gamma,q_i}(\bar\lm_i^{k+1}, \{\xi^\ell_i\}^{r_{k+1}}_{\ell=1})]_l  
&= \frac{1}{r_{k+1}}\sum_{\ell=1}^{r_{k+1}} [\nabla W_{\gamma,q_i}^*(\bar\lm_i^{k+1},\xi_i^\ell)]_l \notag\\
&=  \frac{1}{r_{k+1}}\sum_{\ell=1}^{r_{k+1}} \frac{\exp\left(([\bar{\lm}_i^{k+1}]_l-C_{l \xi_i^\ell})/\gamma\right)  }{\sum_{t=1}^n\exp\left(([\bar{\lm}_i^{k+1}]_t-C_{t\xi_i^\ell })/\gamma\right)}, 
\end{align}
for all $ l =1,...,n$
                with batch size
               \[
                r_{k+1} =\max \left\{ 1,  50 \lm_{\max}(W)m{\alpha}_{k+1} \ln(N/\alpha)/{\e} \right\}. \]
               \State $
               \bar \zeta_i^{k+1}= \bar\zeta_i^{k} - \alpha_{k+1} \sum_{j=1}^m \WW_{ij}\nabla^{r_{k+1}} W^*_{\gamma,q_j}(\bar\lm_j^{k+1}, \{\xi_i^\ell\}^{r_{k+1}}_{\ell=1}).$
               \State $
              \bar \eta_i^{k+1} =(\alpha_{k+1}\bar\zeta_i^{k+1} + A_k \bar\eta_i^k)/{A_{k+1}}.$
                \EndFor
        \Ensure     $ \tilde \p = (\tilde p_1^\top,..., \tilde p_m^\top)^\top$, where $ \tilde p_i = \frac{1}{A_{N}}\sum_{k=0}^{N} \alpha_k p_i(\bar \lm^k_i,\{\xi^\ell_i\}_{\ell=1}^{r_k})$ for all $i=1,\dots,m$ with 
            \[ p_i(\bar \lm_i^k,\{\xi^\ell_i\}_{\ell=1}^{r_k}) \triangleq \frac{1}{r_k}\sum_{\ell=1}^{r_k}  p_i(\bar \lm_i^k,\xi_i^\ell) = \nabla^{r_{k}} W^*_{\gamma,q_i}(\bar\lm_i^{k}, \{\xi^\ell_i\}^{r_{k}}_{\ell=1}).\]
\end{algorithmic}
\end{algorithm}

\paragraph{Recovery of the Primal Solution. }

  We construct a stochastic approximation for $\nabla W^*_{\gamma,\q}(\sqrt \WW \y)$ by using batches of  size  $r$ 
and the change of variable  $  \bar \y := \sqrt{\WW}\y $
        \begin{align}
        \nabla^{r} W^*_{\gamma,\q}(\sqrt \WW\y, \{ \boldsymbol \xi^j\}^r_{j=1}) = \sqrt \WW \nabla^r W^*_{\gamma,\q}(\bar \y, \{ \boldsymbol \xi^j\}^r_{j=1}) 
        &= \frac{1}{r}\sum_{j=1}^r \sqrt \WW \nabla W^*_{\gamma,\q}(\bar \y, \boldsymbol \xi^j) \notag \\
        &=\frac{1}{r}\sum_{j=1}^r \sqrt \WW \p (\bar \y, \boldsymbol \xi^j), 
        \end{align}
         where  $[\p(\bar \y, \boldsymbol \xi)]_i = p_i(\bar y_i, \xi_i)$ is      
\begin{equation}\label{eq:primal_sol_recov33}
\forall l =1,...,n \qquad [p_i (\bar y_i, \xi_i)]_l =  \frac{\exp\left(([\bar y_i]_l-C_{l\xi_i})/\gamma\right)  }{\sum_{\ell=1}^n\exp\left(([\bar y_i]_\ell-C_{\ell\xi_i})/\gamma\right)}.
\end{equation}

The next theorem presents an application of  Theorem \ref{Th:stoch_errdecentr} (with changing the constant for the weighted problem) to the Wasserstein barycenter problem.
\begin{theorem}
 Let $R_\lm$ be such that $\|\lm^*\|\leq R_\lm$, where $\lm^*$ be an exact solution of dual problem \eqref{eq:W_bary_regdual}.
Let  the batch size be taken   $r_k = \max \left\{ 1,  50 \lm_{\max}(W)m   {\alpha}_{k+1} \ln(N/\alpha)/{\e} \right\}$.
Then after $
    N = O\left(\sqrt{\frac{ n\|C\|^2_\infty }{ \gamma \varepsilon} \chi(W) }\right)$ iterations
 for the output $\tilde \p = (\tilde p_1^\top,..., \tilde p_m^\top)^\top$ of Algorithm~\ref{Alg:DualWass} 
 the following holds with probability at least $ 1-3\alpha$
\begin{align*}
\frac{1}{m}\sum_{i=1}^m W_{\gamma}(\tilde p_i, q_i)- 
\frac{1}{m}\sum_{i=1}^m  W_{\gamma}(p^*, q_i) \leq \e, \qquad \|\sqrt{\WW}\tilde \p\|_2 \leq  \e/R_{\lm}.
\end{align*}
Moreover, the per node complexity of Algorithm~\ref{Alg:DualWass}  is
\[
 \widetilde O\left( n \cdot \max\left\{ \sqrt{\frac{n\|C\|_\infty^2}{\gamma \e}\chi(W)} , \frac{ n\|C\|_\infty^2}{\e^2} \chi(W) \right\} \right).\]
\end{theorem}
\begin{proof}
The proof of the theorem follows from the Theorem \ref{Th:stoch_errdecentr}. Thus, we have the following number of iterations
\begin{equation*}
N = O\left(\sqrt{\frac{ M^2 }{\gamma \e}\chi(W)} \right)  \overset{Th. \ref{Prop:wass_prop}}{=} O\left(\sqrt{\frac{ n\|C\|_\infty^2}{\gamma \e}\chi(W)} \right).
\end{equation*}
The number of oracle calls  of $\nabla W_{\gamma,q_i}^*(\bar y_i,\xi_i)$  is (Theorem \ref{Th:stoch_errdecentr})
\begin{equation}\label{ffldsmvnnddd}
\widetilde O\left( \max\left\{ \sqrt{\frac{ M^2 }{\gamma \e}\chi(W)}, \frac{ M^2 \sigma_\psi^2}{\e^2} \chi(W) \right\} \right),
\end{equation}
 where $\sigma^2_\psi$
 is sub-Gaussian variance of $\nabla W_{\gamma,q_i}^*(\bar y_i,\xi_i)$.
Now we estimate  variance $\sigma_\psi^2$ of $\nabla W^*_{\gamma,q_i}(\bar y_i,\xi_i)$ \eqref{eq:primal_sol_recov33}
\begin{align*}
  \sigma^2_\psi   &= \max_{\bar y_i}\left\{\E \|p_i(\bar y_i,\xi_i)\|_2^2  - \left( \E \|p_i(\bar y_i,\xi_i)\|_2\right)^2\right\} \notag\\
  &\leq \max_{\bar y_i}\left\{\E \|p_i(\bar y_i,\xi_i)\|_2^2 \right\} \leq \max_{\bar y_i}\left\{\E \|p_i(\bar y_i,\xi_i)\|_1^2\right\} = 1.
\end{align*}
Thus, we have $\sigma_\psi^2 \leq 1$. Using this and $M \leq \sqrt n \|C\|_\infty$ (Theorem \ref{Prop:wass_prop}) in \eqref{ffldsmvnnddd} we get
\[
 \widetilde O\left( \max\left\{ \sqrt{\frac{n\|C\|_\infty^2}{\gamma \e}\chi(W)} , \frac{ n\|C\|_\infty^2 }{\e^2} \chi(W) \right\} \right).\]
Multiplying this by the cost for calculating  $\nabla W_{\gamma, q_i}^*(\bar y_i,\xi_i)$,  which is $O(n)$, we get  the  per node  complexity  
\[
 \widetilde O\left( n \cdot \max\left\{ \sqrt{\frac{n\|C\|_\infty^2}{\gamma \e}\chi(W)} , \frac{ n\|C\|_\infty^2 }{\e^2} \chi(W) \right\} \right).\]
\end{proof}


\chapter{Saddle Point Approach for the Wasserstein Barycenter Problem}\label{ch:WB}
\chaptermark{Saddle Point Approach for the WB Problem}

In this Chapter, we provide a primal   algorithm to compute unregularized Wasserstein barycenters  with no limitations in contrast to the regularized-based methods, which are numerically unstable under a small value of the regularization parameter. The algorithm is based on the saddle point problem reformulation and the application of mirror prox algorithm with a specific norm. We also show how the algorithm can be executed in a decentralized manner.  The complexity of the proposed algorithms meets the best known results in decentralized and non-decentralized  setting.

\paragraph{Previous Works.}
Optimal transport problem (OT) \eqref{def:optimaltransport} is not an easy task. Indeed, to solve this problem between two discrete histograms of size $n$, one needs to make  $\widetilde O(n^3)$ arithmetic calculations \cite{tarjan1997dynamic,peyre2019computational}, e.g., by using simplex method or interior\dm{-}point method. To overcome the computational issue, entropic regularization of the OT was proposed by \citet{cuturi2013sinkhorn}. It enables an application of the Sinkhorn's algorithm, which is based on alternating minimization procedures and has  $\widetilde O(n^2\|C\|^2_\infty/\e^2)$ convergence rate \cite{altschuler2017near-linear,dvurechensky2018computational} to approximate a solution of OT with $\e$-precision. Here $C \in \R^{n\times n}_+$ is a ground cost matrix of  transporting a unit of mass between probability measures, and the regularization parameter before negative entropy  is of order $\e$. The Sinkhorn's algorithm can be accelerated to   $\widetilde O\left({n^2 \sqrt n \|C\|_\infty}/{\e}\right)$  convergence rate  \cite{guminov2019accelerated}. In practice, the accelerated Sinkhorn\dm{'s algorithm} converges  faster than the Sinkhorn\dm{'s algorithm}, and in theory, it has better  dependence on $\e$ but not on $n$.  Also a faster practice convergence is achieved also by modifications of the Sinkhorn's algorithm, e.g., the  Greenkhorn algorithm \cite{altschuler2017near-linear} of the same 
 convergence rate as the Sinkhorn's algorithm. 

However, all  entropy-regularized based approaches are numerically unstable when the regularizer parameter $\gamma$ before negative entropy is  small (this also means that precision  $\e$ is high as $\gamma $ must be selected proportional to $\e$ \cite{peyre2019computational,kroshnin2019complexity}). 
The recent work  of \citet{jambulapati2019direct} provides  an optimal method for \dm{solving the} OT \dm{problem}, based on dual \dm{extrapolation}  \cite{nesterov2007dual} and area-convexity \cite{sherman2017area}, with  convergence rate  $\widetilde O(n^2\|C\|_\infty/\e)$. This method \dm{works} without additional penalization and, moreover,  it eliminates the term $\sqrt n$ in the bound for the accelerated Sinkhorn's algorithm.  The  rate  $\widetilde O(n^2\|C\|_\infty/\e)$ was also obtained in a number of  works of \citet{blanchet2018towards,allen2017much,cohen2017matrix}. 
Table~\ref{Tab:comp0}, incorporates the most popular algorithms solving  OT problem.
\begin{table}[ht!]
\caption{Algorithms for OT problem and their rates of convergence }
\begin{center}
\begin{tabular}{lll}
\textbf{Paper} &\textbf{Approach}  &\textbf{Complexity} \\
\hline \\
 \cite{dvurechensky2018computational}  & Sinkhorn        &  $\widetilde O\left( \frac{ n^2  \|C\|^2_\infty}{\e^2}\right) $  \\
\cite{guminov2019accelerated} & Accelerated Sinkhorn            & $ \widetilde O\left(\frac{n^2 \sqrt{n} \|C\|_\infty}{\e}\right)$   \\
\makecell[tl]{ \\ \cite{jambulapati2019direct}} & \makecell[tl]{Optimal algorithm based on \\ dual extrapolation\\  with area-convexity} &  \makecell[tl]{ \\ $\widetilde O\left(\frac{n^2\|C\|_\infty}{\e}\right)$}
\end{tabular}
\label{Tab:comp0}
\end{center}
\end{table}

Wasserstein barycenter (WB) problem \eqref{def:Wassersteinbarycenter} of $m$ measures consists in minimizing  the sum of $m$  
 squared $2$-Wasserstein distances (generated by OT metric) to all objects in the set.
Regularizing each OT distance in the sum by negative entropy leads to presenting the WB problem as Kullback--Leibler projection that can be performed by the iterative Bregman projections (IBP) algorithm  \cite{benamou2015iterative}. The IBP is an extension of the Sinkhorn’s algorithm for $m$ measures,  and hence, its complexity is  $m$ times more than the Sinkhorn complexity, namely $ \widetilde O\left({ mn^2  \|C\|^2_\infty}/{\e^2}\right)$ \cite{kroshnin2019complexity}. 
An analog of the accelerated Sinkhorn's algorithm for the WB problem of $m$ measures is the  accelerated IBP algorithm with complexity $\widetilde O\left({mn^2 \sqrt{n} \|C\|_\infty}/{\e}\right)$  \cite{guminov2019accelerated}, that is also $m$ times more than the accelerated Sinkhorn complexity. 
Another fast version of the IBP algorithm was recently proposed by \citet{lin2020fixed}, named FastIBP with complexity $ \widetilde O\left({mn^2\sqrt[3]{n}  \|C\|^{4/3}_\infty}/{\e^{4/3}}\right)$.





\paragraph{Contribution.}
We propose a new algorithm, based on mirror prox with specific prox-function, for the WB problem  which does not suffer from a small value of the regularization parameter and, at the same time, has complexity not worse than the celebrated (accelerated) IBP. Moreover, this algorithm can be performed in a decentralized manner.


Table~\ref{Tab:comp} illustrates the contribution by comparing our new algorithm, called `Mirror prox with specific norm', with the most popular algorithms for the WB problem. Algorithm `Dual extrapolation with area-convexity' was proposed in joint paper \cite{dvinskikh2020improved} together with `Mirror prox with specific norm' as an improved version of `Mirror prox with specific norm' under the weaker convergence requirements of area-convexity. 
`Dual extrapolation with area-convexity' has the best theoretical rate of convergence for the Wasserstein barycenter problem, which is probably optimal.
However, it does not have so obvious decentralized interpretation which `Mirror prox with specific norm' has.

  {
\begin{table}[ht]
\caption{Algorithms  for the WB problem and their rates of convergence  } 
\begin{center}
\begin{tabular}{lll}
\textbf{Approach}  & \textbf{Paper} & \textbf{Complexity} \\
\hline \\
 IBP  &\cite{kroshnin2019complexity}     &  
$\widetilde O\left( \frac{ mn^2  \|C\|^2_\infty}{\e^2}\right) $    \\
Accelerated IBP & \cite{guminov2019accelerated}               & 
   $ \widetilde O\left(\frac{mn^2 \sqrt{n} \|C\|_\infty}{\e}\right)$   
   \\
  FastIBP   & \cite{lin2020fixed}        & 
$ \widetilde O\left(\frac{mn^2\sqrt[3]{n}  \|C\|^{ 4/3}_\infty}{\e\sqrt[3]{\e}}\right)$    
\\
    \makecell[tl]{Mirror prox\\ with  specific norm }     & \cite{dvinskikh2020improved}   &  $ \widetilde O\left(\frac{ mn^2 \sqrt{ n} \|C\|_\infty}{\e}\right) $  \\
   \makecell[tl]{Dual  extrapolation\\ with 
area-convexity }     & \cite{dvinskikh2020improved}  &  $ \widetilde O\left( \frac{ mn^2   \|C\|_\infty}{\e} \right)$ \\
\end{tabular}
\label{Tab:comp}
\end{center}
\end{table}
}

Figure \ref{fig:compMirrorIBP} illustrates numerically instability of the IBP with regularizing parameter $\gamma$ algorithm when \dm{a} high-precision $\e$ of calculating Wasserstein barycenters is desired since  $\gamma $ must be selected proportional to $\e$ \cite{peyre2019computational,kroshnin2019complexity}. `Dual extrapolation with area-convexity' and `Mirror prox with specific norm' \cite{dvinskikh2020improved} produce good results.

    \begin{figure}[ht!]
        \centering
    \begin{subfigure}[b]{2.5cm}
        \centering
        \includegraphics[width=0.8\linewidth]{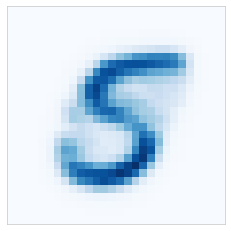}
        \label{fig:fig1}
    \end{subfigure}
    \begin{subfigure}[b]{2.5cm}
        \centering
        \includegraphics[width=0.8\linewidth]{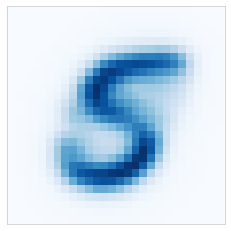}
        \label{fig:fig2}
    \end{subfigure}
    \begin{subfigure}[b]{2.5cm}
        \centering
        \includegraphics[width=0.8\linewidth]{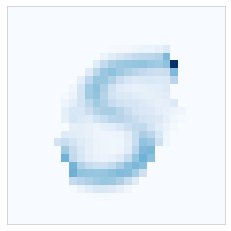}
        \label{fig:fig3}
    \end{subfigure}
    \begin{subfigure}[b]{2.5cm}
        \centering
        \includegraphics[width=0.8\linewidth]{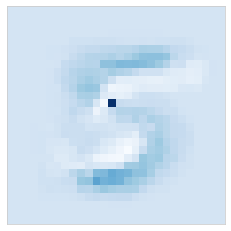}
        \label{fig:fig4}
    \end{subfigure}\\
    \vspace{0.3cm}
        \begin{subfigure}[b]{2.5cm}
        \centering
        \includegraphics[width=0.8\linewidth]{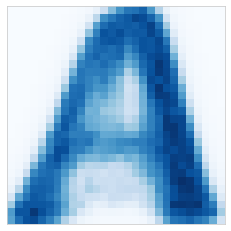}
        \captionsetup{justification=centering}
        \caption*{Mirror Prox for WB}
        \label{fig:fig5}
    \end{subfigure}
    \begin{subfigure}[b]{2.5cm}
        \centering
        \includegraphics[width=0.8\linewidth]{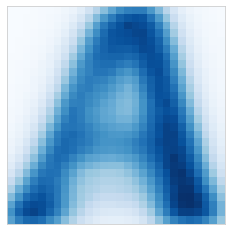}
        \captionsetup{justification=centering}
        \caption*{Dual Extra- polation }
        \label{fig:fig6}
    \end{subfigure}
    \begin{subfigure}[b]{2.5cm}
        \centering
        \includegraphics[width=0.8\linewidth]{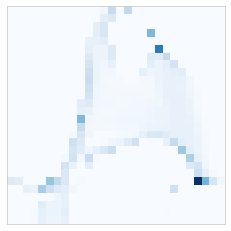}
        \captionsetup{justification=centering}
        \caption*{IBP, \\ $\gamma=10^{-3}$}
        \label{fig:fig7}
    \end{subfigure}
    \begin{subfigure}[b]{2.5cm}
        \centering
        \includegraphics[width=0.8\linewidth]{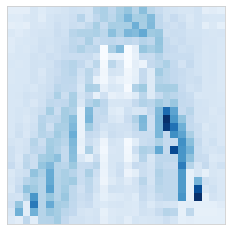}
        \captionsetup{justification=centering}
        \caption*{IBP, \\ $\gamma=10^{-5}$}
        \label{fig:fig8}
    \end{subfigure}
    \caption{ Wasserstein barycenters of hand-written digits `5' from the MNIST dataset (first row) and  Wasserstein barycenters of letters `A' from the notMNIST dataset (second row).}\label{fig:compMirrorIBP}
    \end{figure}

Figure \ref{fig:gausbar} demonstrates  better approximations of the true Gaussian barycenter by `Dual extrapolation with area-convexity' and `Mirror prox with specific norm'  compared to the $\gamma$-regularized IBP  barycenter. 
  The regularization parameter for the IBP algorithm (from the POT python library) is taken as smallest as possible  under which the IBP  still works since the  smaller $\gamma$, the closer  regularized IBP barycenter is to the true barycenter.

\begin{figure}[ht!]
\centering
\includegraphics[width=0.4\textwidth]{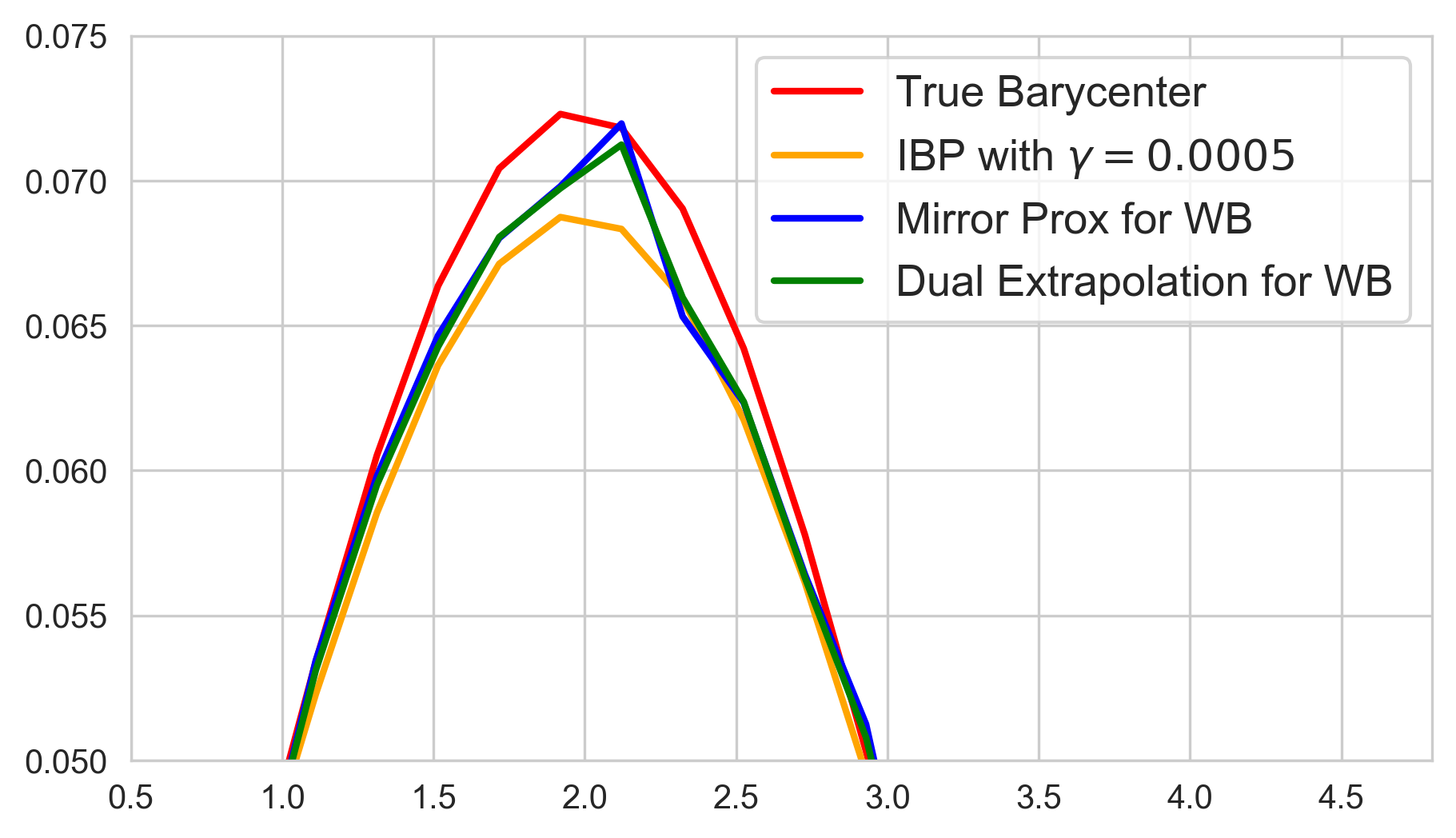}
\caption{Convergence of the barycenters to the true barycenter of  Gaussian measures. }
\label{fig:gausbar}
\end{figure}

The algorithm `Mirror prox with specific norm' can be also preformed in a decentralized manner and has the same per node complexity as Decentralized FGD \cite{dvinskikh2019primal} up to the dependence on communication matrix. 
For the star network, we can compare the complexity of decentralized mirror-prox  with the complexity  of the IBP running in $ \widetilde O\left({ n^2  
}/{\e^2}\right)$ time per node \cite{kroshnin2019complexity}. Decentralized mirror-prox has better dependence on $\e$, namely $1/\e$, as well as the  accelerated IBP with 
$\widetilde O\left({n^2 \sqrt{n} 
}/{\e}\right)$ complexity per node of  \cite{guminov2019accelerated}. The details of the comparison can be found in Table \ref{Tab:distr_comp}

  {
\begin{table}[ht]
\caption{Distributed algorithms  for the WB problem and their per node complexity  } 
\hspace{-0.5cm}
\small
\begin{tabular}{llll}
\textbf{Approach}  & \textbf{Paper} & \textbf{Architecture} & \textbf{ Complexity per node} \\
\hline \\
 IBP  &\cite{kroshnin2019complexity}    & star &  
$\widetilde O\left( \frac{ n^2  \|C\|^2_\infty}{\e^2}\right) $    \\
Accelerated IBP &  \cite{guminov2019accelerated}               & star &
   $ \widetilde O\left(\frac{n^2 \sqrt{n} \|C\|_\infty}{\e}\right)$   
   \\
  FastIBP   & \cite{lin2020fixed}      &  star  & 
$ \widetilde O\left(\frac{n^2\sqrt[3]{n}  \|C\|^{ 4/3}_\infty}{\e\sqrt[3]{\e}}\right)$    
\\
Decentralized FGD & \cite{dvinskikh2019primal} & any & $\widetilde{O}\left(\frac{n^2\sqrt{ n   \chi(W)}\|C\|_\infty}{ \varepsilon} \right)$ \\
    \makecell[tl]{Decentralized \\
    mirror prox\\ with  specific norm}     & \cite{rogozin2021decentralized}   & any & 
   $\widetilde{O}\left(\frac{n^2\sqrt{ n   \chi(W)}\|C\|^{3/2}_\infty}{ \varepsilon} \right)$
   \\
\end{tabular}
\label{Tab:distr_comp}
\end{table}
}


\section{Mirror Prox for Wasserstein Barycenters}
Our new approach is based on mirror prox algorithm with specific prox-function for  the Wasserstein barycenter problem formulated as a saddle-point problem.  To present the Wasserstein barycenter problem as a saddle-point problem, we refer to the work  \cite{jambulapati2019direct}, where the authors obtain  saddle-point representation for optimal transport problem. To so, they vectorize the cost matrix and transport plan.
\subsection{Saddle Point Formulation}

We consider optimal transport problem \eqref{def:optimaltransport} between two discrete measures $\pp = \sum_{i=1}^n p_i\delta_{z_i}$ and $\qq = \sum_{i=1}^n p_i\delta_{y_i}$ of support size $n$. The histograms 
$p$ and $q$ are from the probability simplex $\Delta_n$.
Let $d$ be vectorized cost matrix $C$, and let $x$ be vectorized transport plan $\pi \in U(p,q) \triangleq\{ \pi\in \R^{n \times n}_+: \pi \one_{n} =p, \pi^T \one_{n} = q\}$. 
Due to the marginals $p,q$  of transport plan $\pi $ are from probability simplex $\Delta_n$, it holds that $\sum_{i,j=1}^n \pi_{ij} = 1$. We also introduce
$b = \begin{pmatrix}
p\\
q
\end{pmatrix} $ and  incidence matrix $A=\{0,1\}^{2n\times n^2}$. Then the optimal transport problem \eqref{def:optimaltransport} can be rewrutten as 
\[
\min_{Ax=b, ~ x\in \Delta_{n^2}} d^\top x.
\]
And then based on the definition of the $\ell_1$-norm, this problem can be presented as a saddle-point problem \cite{jambulapati2019direct}
\begin{equation*}
    \min_{x \in \Delta_{n^2}} \max_{y\in [-1,1]^{2n}} \{d^\top x +2\|d\|_\infty(~ y^\top Ax -b^\top y)\}.
\end{equation*}
 Using this representation for optimal transport problem we present the Wasserstein barycenter problem of 
 histograms  $q_1, q_2,..., q_m \in \Delta_n$ as follows
\begin{equation}\label{eq:def_saddle_probWBbeg}
    \min_{ p \in \Delta_n} \frac{1}{m}\sum_{i=1}^m \min_{ x_i\in \Delta_{n^2}} \max_{~ y_i\in [-1,1]^{2n}}   \{d^\top x_i +2\|d\|_\infty\left(y_i^\top Ax_i -b_i^\top y_i\right)\},
\end{equation}
where  $b_i = \begin{pmatrix}
p \\q_i \end{pmatrix}$. 
Next, we define spaces $\X \triangleq \prod^m \Delta_{n^2} \times \Delta_{n}$ and $\Y \triangleq [-1,1]^{2mn}$, where  $ \prod^m \Delta_{n^2} \times \Delta_{n}$ is a short form of  $  \underbrace{\Delta_{n^2}\times \ldots \times \Delta_{n^2}}_{m} \times \Delta_{n} $, and present \eqref{eq:def_saddle_probWBbeg}   for  column vectors $\x = (x_1^\top,\ldots,x_m^\top, p^\top)^\top \in \X $
and  $\y = (y_1^\top,\ldots,y_m^\top)^\top \in \Y$ as follows
\begin{align}\label{eq:def_saddle_probFinal}
    \min_{ \x \in \X}  \max_{ \y \in \Y} &~F(\x,\y)\triangleq  \frac{1}{m} \left\{\boldsymbol d^\top \x +2\|d\|_\infty\left(\y^\top\boldsymbol A \x -\c^\top \y \right)\right\}, 
\end{align}
where 
$\boldsymbol d = (d^\top, \ldots, d^\top, \boldsymbol  0_n^\top)^\top $, $\c = (\boldsymbol 0_n^\top, q_1^\top, \ldots , \boldsymbol 0_n^\top, q_m^\top)^\top $ and   
$\boldsymbol A = 
\begin{pmatrix}
\hat A & \mathcal{E}       
\end{pmatrix}
 \in \{-1,0,1\}^{2mn\times (mn^2+n)} $ with block-diagonal matrix  
  $\hat{A} = {\rm diag}\{A, ..., A\} $ of \(m\) blocks,
 and matrix
 \[
    \mathcal{E}^\top = \begin{pmatrix}
\begin{pmatrix}
              -I_n & 0_{n \times n} 
        \end{pmatrix}   
\begin{pmatrix}
              -I_n & 0_{n \times n} 
        \end{pmatrix}   \cdots  \begin{pmatrix}
              -I_n & 0_{n \times n} 
        \end{pmatrix}     \end{pmatrix}.
 \]
Since objective $F(\x, \y)$ in \eqref{eq:def_saddle_probFinal}   is convex in $\x$ and concave in $\y$, problem \eqref{eq:def_saddle_probFinal}  is a saddle-point representation of the Wasserstein barycenter problem. 
We will evaluate the quality of an algorithm, that outputs a pair of solutions $(\widetilde \x,\widetilde \y) \in (\X,\Y)$,  through the so-called duality gap
\begin{equation}\label{eq:precision_alg_mirr}
    \max_{\y \in \Y} F\left( \widetilde \x,\y\right) - \min_{ \x \in \X} F\left(\x,\widetilde \y \right)  \leq \e.
\end{equation}

\subsection{Algorithm and Convergence Rate}
\paragraph{Setups.}
\begin{itemize}
    \item We endow space $\Y \triangleq [-1,1]^{2nm}$ with the  standard Euclidean setup: the Euclidean norm $ \|\cdot\|_2$,  prox-function $d_\Y(\y) = \frac{1}{2}\|\y\|_2^2$ and the corresponding Bregman divergence
    
$      B_\Y(\y, \breve \y) = \frac{1}{2}\|\y -\breve \y\|_2^2$.
We define $R^2_\Y = \sup\limits_{\y \in \Y }d_\Y(\y) - \min\limits_{\y \in \Y }d_\Y(\y)$.  

    \item
We endow space   $ \X \triangleq  \prod^m\Delta_{n^2}\times \Delta_{n}$ with norm $ \|\x\|_\X = \sqrt{\sum_{i=1}^m \|x_i\|^2_1 +m\|p\|_1^2}$ for $\x = (x_1,\dots,x_m,p)^T$, where $\|\cdot\|_1 $ is the $\ell_1$-norm. 
 We  endow $\X$ with  prox-function $d_\X(\x) = \sum_{i=1}^m  \la x_i,\log x_i \ra +m\la p,\log  p \ra$  and  corresponding  Bregman divergence
\begin{align*}
     B_\X(\x, \breve \x) = &\sum_{i=1}^m  \la x_i, \log (x_i /  \breve x_i) \ra -\sum_{i=1}^m\boldsymbol 1^\top( x_i -  \breve x_i)  +m\la p, \log (p/\breve p)  \ra - m\boldsymbol 1^\top( p -   \breve p). 
\end{align*}
We  define $R^2_\X = \sup\limits_{\x \in  \X }d_\X(\x) - \min\limits_{\x \in  \X }d_\X(\x)$.
\end{itemize}



The next definition clarifies the notion of smoothness for the objective in convex-concave problems.
\begin{definition}\label{def:smooth}
$F(\x,\y)$ is $ (L_{\x\x},L_{\x\y}, L_{\y\x}, L_{\y\y})$-smooth if for any $\x, \x' \in \X$ and $\y,\y' \in \Y$, 
\begin{align*}
     \|\nabla_\x f(\x,\y) - \nabla_\x f(\x',\y)\|_{\X^*}
     &\leq L_{\x\x}\|\x-\x' \|_{\X},\\
       \|\nabla_\x f(\x,\y) - \nabla_\x f(\x,\y')\|_{\X^*}
       &\leq L_{\x\y} \|\y-\y' \|_{\Y},\\
         \|\nabla_\y f(\x,\y) - \nabla_\y f(\x,\y')\|_{\Y^*}
         &\leq L_{\y\y}\|\y-\y' \|_{\Y} ,\\
           \|\nabla_\y f(\x,\y) - \nabla_\y f(\x',\y)\|_{\Y^*}
           &\leq L_{\y\x} \|\x-\x' \|_{\X}.
\end{align*}
\end{definition}
We consider mirror prox (MP) \cite{nemirovski2004prox} on space $\Z \triangleq \X\times \Y$  with prox-function $d_\Z(\z) = a_1d_\X(\x)+a_2d_\Y(\y)$ and corresponding Bregman divergence $B_\Z(\z,\breve \z) = a_1B_\X(\x,\breve \x) + a_2B_\Y(\y,\breve \y)$, where   $a_1 = \frac{1}{R_{\X}^2}$, $a_2 =\frac{1}{R_{\Y}^2} $
\begin{align*}
& \begin{pmatrix}
  \u^{k+1}  \\
\v^{k+1} 
    \end{pmatrix} = \arg\min_{\z \in \Z} \{ \eta G(\x^k,\y^k)^\top \z + B_\Z(\z, \z^k) \},\\
&\hspace{4mm}\z^{k+1} = \arg\min_{\z \in \Z} \{ \eta G(\u^{k+1},\v^{k+1})^\top \z + B_\Z(\z,\z^k) \},
\end{align*}
where
$\eta$ is learning rate,  $\z^1 = \arg\min\limits_{\z \in \Z} d_\Z(\z)$ and $G(\x,\y)$ is the  gradient operator defined as follows
\begin{align*}
    G(\x,\y) = 
    \begin{pmatrix}
  \nabla_\x F(\x,\y) \\
 -\nabla_\y F(\x,\y)
    \end{pmatrix} = 
   \frac{1}{m} \begin{pmatrix}
   \boldsymbol d + 2\|d\|_\infty \boldsymbol A^\top \y \\
   2\|d\|_\infty(\c -\boldsymbol A\x)
    \end{pmatrix}.
\end{align*}
If $F(\x,\y)$ is $ (L_{\x\x},L_{\x\y}, L_{\y\x}, L_{\y\y})$-smooth, then to satisfy \eqref{eq:precision_alg_mirr} with $\widetilde \x = \frac{1}{N}\sum_{k=1}^N \u^k$, $\widetilde \y = \frac{1}{N}\sum_{k=1}^N \v^k$, one needs to perform
\begin{equation}\label{eq:MP_number_it}
    N = \frac{4}{\e} \max\{ L_{\x\x}R_{\X}^2, L_{\x\y}R_{\X}R_{\Y}, L_{\y\x}R_{\Y}R_{\X}, L_{\y\y}R_{\Y}^2\})
\end{equation}
iterations of the MP \cite{nemirovski2004prox,bubeck2014theory}  with
\begin{equation}\label{eq:MP_lear_rate}
    \eta = {1}/{(2\max\{ L_{\x\x}R_{\X}^2, L_{\x\y}R_{\X}R_{\Y}, L_{\y\x}R_{\Y}R_{\X}, L_{\y\y}R_{\Y}^2\})}.
\end{equation}

\begin{algorithm}[ht!]
    \caption{Mirror Prox for the Wasserstein Barycenter Problem}
    \label{MP_WB}
    \small
    \begin{algorithmic}[1]
\Require measures $q_1,...,q_m$, linearized cost matrix $d$, incidence matrix $A$, step $\eta$, $p^0=\frac{1}{n}\boldsymbol 1_{n}$, $x_1^0=...= x_m^0 = \frac{1}{n^2}\boldsymbol 1_{n^2}$, $y_1^0 = ... =y_m^0 =\boldsymbol 0_{2n}$
        \State   $\alpha = 2\|d\|_\infty \eta n  $,
         $\beta =6\|d\|_\infty \eta \log n/m$,
         $\gamma = 3 \eta \log n$.
        \For{ $k=0,1,2,\cdots,N-1$ }
         \For{ $i=1,2,\cdots, m$}
         \State  
         
         $v_i^{k+1} = y^k_i + \alpha\left(  
         A x_i^k - 
         \begin{pmatrix}
         p^k\\
         q_i
         \end{pmatrix}
        \right),$
        
         Project $v_i^{k+1}$ onto $[-1,1]^{2n}$
\State  
\[  u^{k+1}_i =\frac{ x^{k}_i \odot  \exp\left\{ 
       - \gamma \left( d+ 2\|d\|_\infty A^\top y^{k}_i \right)
        \right\}}{\sum\limits_{l=1}^{n^2} [x^{k}_i]_l\exp\left\{  
       - \gamma \left( [d]_l+ 2\|d\|_\infty [A^\top y^{k}_i]_l \right)
        \right\}}    \]
        \EndFor
     \State  
     \[s^{k+1} = \frac{p^{k} \odot \exp\left\{ 
        \beta
        \sum_{i=1}^m [y^k_i]_{1...n}
        \right\}}{\sum_{l=1}^n [p^{k}]_l\exp\left\{  
        \beta
        \sum_{i=1}^m [y^k_i]_{l}
        \right\}} \] 
      \For{ $i=1,2,\cdots, m$}
        \State
       
        $y_i^{k+1} = y^k_i + \alpha\left( A u_i^{k+1}-
        \begin{pmatrix}
          s^{k+1}\\
          q_i
        \end{pmatrix}
        \right)$
        
        Project $y_i^{k+1}$ onto $[-1,1]^{2n}$

     \State  \[  
         x^{k+1}_i =\frac{ x^{k}_i\odot \exp\left\{ 
       - \gamma \left( d+ 2\|d\|_\infty A^\top v^{k+1}_i \right)
        \right\}}{\sum\limits_{l=1}^{n^2} [x^{k}_i]_l\exp\left\{  
       - \gamma \left( [d]_l+ 2\|d\|_\infty [A^\top v^{k+1}_i]_l \right)
        \right\}}  \]
        \EndFor
        \State 
         \[  p^{k+1} = \frac{p^{k} \odot \exp\left\{ 
        \beta
        \sum_{i=1}^m [v^{k+1}_i]_{1...n}
        \right\}}{\sum_{l=1}^n [p^{k}]_l\exp\left\{  
        \beta
        \sum_{i=1}^m [v^{k+1}_i]_{l}
        \right\}}  \] 
          \EndFor
        \Ensure 
 $\widetilde \u = 
        \sum\limits_{k=1}^N \begin{pmatrix}
         u_1^k\\
         \vdots \\
         u_m^k\\
        s^k
        \end{pmatrix}
       $, $\widetilde \v = 
        \sum\limits_{k=1}^N
        \begin{pmatrix}
       v_1^k\\
       \vdots \\
      v_m^k
       \end{pmatrix}
       $ 
    \end{algorithmic}
\end{algorithm}

\begin{lemma}\label{lm:Lipsch}
Objective $F(\x,\y)$ in  \eqref{eq:def_saddle_probFinal}  is $ (L_{\x\x},L_{\x\y}, L_{\y\x}, L_{\y\y})$-smooth with $L_{\x\x}=L_{\y\y}=0$ and $L_{\x\y}=L_{\y\x} = {2 \dm{\sqrt 2}\|d\|_\infty}/{m} $.
\end{lemma}
\noindent \dm{\textit{Proof. }}
Let us consider bilinear function \[f(\x,\y) \triangleq  \y^\top\boldsymbol A \x\] that is equivalent to   $F(\x,\y)$ from  \eqref{eq:def_saddle_probFinal}  up to multiplicative constant $2\|d\|_\infty/m$ and linear terms.
As $f(\x,\y) $ is bilinear, $L_{\x\x}=L_{\y\y}=0$ in Definition \ref{def:smooth}. Next we estimate $L_{\x\y}$ and $L_{\y\x}$.
By the definition of $L_{\x\y}$ and the spaces $\X,\Y$ defined in the Setup we have
\[\|\nabla_\x f(\x,\y) - \nabla_\x f(\x,\y')\|_{\X^*}
       \leq L_{\x\y} \|\y-\y' \|_{2}.\]
    Since   $\nabla_\x f(\x,\y) = \boldsymbol A^\top \y$ we get
    \begin{equation}\label{eq:Lxydef}
           \|\boldsymbol A^\top (\y - \y')\|_{\X^*}
       \leq L_{\x\y} \|\y-\y' \|_{2}. 
    \end{equation} 
      By the definition of dual norm we have
   \begin{equation}\label{eq:dualconjnorm}
   \|\boldsymbol A^\top (\y - \y')\|_{\X^*} = \max_{\|\x\|_\X \leq 1}\la \x, \boldsymbol A^\top (\y - \y') \ra.
   \end{equation}
       As $\la \x, \boldsymbol A^\top (\y - \y') \ra $ is a linear function,  \eqref{eq:Lxydef} can be rewritten using \eqref{eq:dualconjnorm} as
\[   L_{\x\y} = \max_{\|\y - \y'\|_2\leq 1} \max_{\|\x\|_\X\leq 1} \la  \x,\boldsymbol A^\top (\y - \y') \ra.          \]
Making the change of variable $\tilde \y = \y - \y'$ and using the equality $\la \x, \A^\top \tilde \y\ra = \la  \A\x ,\tilde \y \ra$ we get
\begin{equation}\label{eq:Lxydefrewrt}
L_{\x\y}  = \max_{\|\tilde \y\|_2\leq 1} \max_{\|\x\|_\X\leq 1} \la  \A\x,\boldsymbol  \tilde \y \ra.    
\end{equation}
By the same arguments we can get the same expression for $L_{\y\x}$ up to rearrangement of maximums.
Then since the   $\ell_2$-norm  is the conjugate norm for the $\ell_2$-norm , we rewrite \eqref{eq:Lxydefrewrt} as follows 
\begin{equation}\label{eq:Lxy00}
   L_{\x\y}  =  \max_{\|\x\|_\X\leq 1} \|\boldsymbol A \x\|_2. 
\end{equation}
By the definition of matrix $\boldsymbol A$ we get 
\begin{equation}\label{eq:Axnorm}
\|\boldsymbol A \x \|_2^2 = \sum_{i=1}^{m}\left\|   Ax_i -  \begin{pmatrix}
p\\
0
\end{pmatrix} \right\|^2_2 
\leq \sum_{i=1}^{m}\|   Ax_i\|_2^2 + m\|p\|_2^2.
\end{equation}
The last bound holds due to $\la A x_i, (p^\top,0_n^\top)^\top \ra \geq 0$ since the entries of $A,x,p$ are non-zero.
By the definition of vector $\x$ we have
\begin{align}\label{eq:normest1}
\max_{\|\x\|_\X\leq 1} \|\boldsymbol A \x \|_2^2   &=  
\max_{\|\x\|^2_\X\leq 1} \|\boldsymbol A \x \|_2^2 =\max_{\sum_{i=1}^m\|x_i\|_1^2 + m\|p\|_1^2\leq 1}\|\boldsymbol A \x \|_2^2 \notag\\
&\stackrel{\eqref{eq:Axnorm}}{=}\max_{\alpha \in \Delta_{m+1}} \left( \sum_{i=1}^{m} \max_{\|x_i\|_1 \leq \sqrt{\alpha_{i}}} \|   Ax_i\|_2^2 +
\max_{\|p\|_1\leq \sqrt{\frac{\alpha_{m+1}}{m}}} m\|p\|_2^2\right) \notag\\
&=\max_{\alpha \in \Delta_{m+1}}  \left( \sum_{i=1}^{m} \alpha_{i} \max_{ \|x_i\|_1 \leq 1}   \|   Ax_i\|_2^2 +
\max_{\|p\|_1\leq 1} \alpha_{m+1}\|p\|_2^2\right).
\end{align}
By the definition of incidence matrix $A$ we get that
$
Ax_i = (h_1^\top, h_2^\top)^\top
$,where $h_1$ and $h_2$ such that $\boldsymbol 1^\top h_1 = \boldsymbol 1^\top h_2= \sum_{j=1}^{n^2} [x_i]_j$ = 1 since $x_i \in \Delta_{n^2}  ~\forall i =1,...,m$.
Thus,
\begin{equation}\label{eq:Ax22New}
  \|A x_i\|_2^2 = \|h_1\|_2^2 + \|h_2\|_2^2 \leq \|h_1\|_1^2 + \|h_2\|_1^2 = 2.
  \end{equation}
For the second term in the r.h.s. of \eqref{eq:normest1} we have 
\begin{equation}\label{eq:p_estim01}
    \max_{\|p\|_1\leq 1} \alpha_{m+1}\|p\|_2^2 \leq \max_{\|p\|_1\leq 1} \alpha_{m+1}\|p\|_1^2 = \alpha_{m+1}.
\end{equation}
Using  \eqref{eq:Ax22New} and \eqref{eq:p_estim01} in \eqref{eq:normest1} we get
\begin{align*}
\max_{\|\x\|_\X\leq 1} \|\boldsymbol A \x \|_2^2  
&\leq 
\max_{\alpha \in \Delta_{m+1}}  \left( 2\sum_{i=1}^{m} \alpha_{i}  
 + \alpha_{m+1}\right) \leq \max_{\alpha \in \Delta_{m+1}}  2 \sum_{i=1}^{m+1} \alpha_{i} = 2.
\end{align*}
Using this for \eqref{eq:Lxy00} we have
that $L_{\x\y}=L_{\y\x} = \sqrt{2}$. To get the constant of smoothness  for function $F(\x,\y)$ we multiply these constants by $2\|d\|_\infty/m$ and finish the proof.

 \hfill$ \square$

The next theorem gives the complexity bound of the MP algorithm for the Wasserstein barycenter problem with  prox-function $d_{\Z}(\z)$. For this particular problem, formulated \dm{as a} saddle-point problem  \eqref{eq:def_saddle_probFinal}, \dm{the} MP algorithm has closed-form solutions presented in Algorithm  \ref{MP_WB}.
\begin{theorem}\label{Th:first_alg}
Assume that $F(\x,\y)$ in  \eqref{eq:def_saddle_probFinal}  is $ (0,{2\dm{\sqrt{2}}\|d\|_\infty}/{m}, {2\dm{\sqrt{2}}\|d\|_\infty}/{m}, 0)$-smooth and $R_{\X}=\sqrt{3m\log n}$, $R_\Y = \sqrt{mn}$.
Then after $N = { 8\|d\|_\infty}\sqrt{\dm{6}n\log n}/{\e}$ iterations,  Algorithm \ref{MP_WB} with 
$\eta   =\frac{1}{4 \|d\|_\infty\sqrt{\dm{6}n \log n}}$
outputs a pair $(\widetilde \u, \widetilde \v)  \in (\X,\Y)$ such that
\begin{align*}
    \max_{\y \in \Y} F\left( \widetilde \u ,\y\right) - \min_{ \x \in \X} F\left(\x,\widetilde \v\right)  \leq \e.
\end{align*}
The total complexity of Algorithm \ref{MP_WB} is \[O\left({mn^2}\sqrt{n\log n}\|d\|_\infty {\e^{-1}}\right).\]
\end{theorem}
\noindent \textit{Proof. }
By Lemma \ref{lm:Lipsch},  $F(\x,\y)$  is $ (0,{2\dm{\sqrt{2}}\|d\|_\infty}/{m}, {2\dm{\sqrt{2}}\|d\|_\infty}/{m}, 0)$-smooth. Then the  bound on duality gap follows from the direct substitution of the expressions for $R_\X$, $R_\Y$ and $L_{\x\x}$, $L_{\x\y}$, $L_{\y\x}$, $L_{\y\y}$ in \eqref{eq:MP_number_it} and \eqref{eq:MP_lear_rate}.

The complexity of one iteration of Algorithm \ref{MP_WB} is $O\left(mn^2\right)$ as the number of non-zero elements in matrix A is $2n^2$, and $m$ is the number of vector-components in $\y$ and $\x$. Multiplying this by the number of iterations $N$, we get the last statement of the theorem. 

 \hfill$ \square$

As $d$ is the vectorized cost matrix of $C$, we may reformulate \dm{the} complexity results of Theorem \ref{Th:first_alg} with respect to $C$ as $O\left({mn^2}\sqrt{n\log n}\|C\|_\infty \e^{-1} \right)$.

Moreover, the complexity results may be improved by $\sqrt n$ term \cite{dvinskikh2020improved}.
\begin{theorem}\citep{dvinskikh2020improved}
    Dual Extrapolation algorithm with area-convexity after 
    \[
        N =  {8\Vert d \Vert_\infty(60 \log n  + 9 \Vert d \Vert_\infty)}/{\e} 
    \] iterations outputs a pair $(\widetilde \u, \widetilde \v)  \in (\X,\Y)$ such that
\begin{align*}
    \max_{\y \in \Y} F\left( \widetilde \u ,\y\right) - \min_{ \x \in \X} F\left(\x,\widetilde \v\right)  \leq \e.
\end{align*} \dm{I}t can be done in \dm{wall-clock time} $\widetilde O(mn^2 \Vert d \Vert_\infty \e^{-1}).$
\end{theorem}

\section{Decentralized Mirror Prox for  Wasserstein Barycenters}

\subsection{Decentralized Saddle-Point Formulation}
To present the Mirror Prox algorithm for the Wasserstein Barycenter  problem in a decentralized manner, we rewrite problem
\eqref{eq:def_saddle_probWBbeg} by introducing artificial constraints $p_1 =p_2=...=p_m$
as follows\begin{align}\label{eq:norm_obj}
 \frac{1}{m} \sum_{i=1}^m    \min_{ \substack{p \in \Delta_n, \\ p_1 =...=p_m}}\min_{ x_i\in \Delta_{n^2}} &\max_{~ y_i\in [-1,1]^{2n}}   \{d^\top x_i + 2\|d\|_\infty\left(y_i^\top Ax_i -b_i^\top y_i\right)\}.
\end{align}
Next  we rewrite  this problem for the stacked column vectors $\p = (p_1^\top \in \Delta_n,\cdots,p_m^\top \in \Delta_n)^\top  \in \Pp \triangleq \prod^m \Delta_{n}$, $\x = (x_1^\top \in \Delta_{n^2},\ldots,x_m^\top \in  \Delta_{n^2})^\top \in \X   \triangleq \prod^m \Delta_{n^2} $ (where $\prod^m \Delta_{n^2}$ is the Cartesian product of $m$ simplices),
and  $\y = (y_1^\top,\ldots,y_m^\top)^\top \in \Y \triangleq [-1,1]^{2mn}$. Then we rewrite  the objective in \eqref{eq:norm_obj}  without normalizing factor $1/m$. We intend  to minimize this objective with accuracy $m\e$.
\begin{align}\label{eq:alm_distrNew}
         &\min_{ \substack{ \x \in \X, \\  \p \in \Pp, \\ p_1=...=p_m }}    \max_{ \y \in \Y}
         ~f(\x,\p,\y) \triangleq 
         \boldsymbol d^\top \x +2\|d\|_\infty\left(\y^\top\boldsymbol A \x -\b^\top \y \right), 
        \end{align}
        where  $\b = (p_1^\top, q_1^\top, ..., p_m^\top, q_m^\top)^\top$, 
$\boldsymbol d = (d^\top, \ldots, d^\top )^\top $, 
   $\boldsymbol A = {\rm diag}\{A, ..., A\} \in \{0,1\}^{2mn\times mn^2}$ is  block-diagonal matrix.
To enable distributed computation of this problem, the constraint $ p_1 = \dots = p_m $ is replaced by 
$\WW \p = 0$ (matrix $W$ is defined in  \eqref{eq:matrixWdef}).
Finally, we introduce Lagrangian dual variable $\z = (z_1^\top, ..., z^\top_m) \in \Z \triangleq \R^{nm}$, scaled by $\gamma$, to constraint $\WW \p=0$ for the problem \eqref{eq:alm_distrNew} and rewrite it as follows
\begin{align}\label{eq:def_saddle_prob2}
    \min_{\substack{ \x \in \X, \\\p \in \Pp } } \max_{\substack{ \y \in \Y, \\ \z \in \R^{nm} }} 
    &~F(\x,\p,\y, \z)\triangleq  \boldsymbol d^\top \x + 2\|d\|_\infty\left(\y^\top\boldsymbol A \x -\b^\top \y \right) +\gamma \la \z,\WW\p \ra.
\end{align}

\subsection{Algorithm and Convergence Rate}
\paragraph{Setup.}
 \begin{itemize}
      \item
     We endow space $\V \triangleq \Y \times \Z  \triangleq [-1,1]^{2nm} \times \R^{nm}$  with the standard Euclidean setup: the Euclidean norm $ \|\cdot\|_2$,  prox-function $d_\v(\v) = \frac{1}{2}\|\v\|_2^2$, and the corresponding Bregman divergence 
 $B_{\v} = \frac{1}{2}\|\v -\breve \v\|_2^2$. We define $R^2_\v = \max\limits_{\v \in \V \cap B_R(0) }d_\v(\v) - \min\limits_{\v \in \V \cap B_R(0)}d_\v(\v)$. Here $B_R(0)$ is a ball of radius $R$ centered in $0$.
     \item  We endow space  $\U \triangleq \X\times \Pp \triangleq \prod^m \Delta_{n^2} \times \prod^m \Delta_{n}  $ with the folllowing  norm $ \|\u\|_\u = \sqrt{\sum_{i=1}^m \|x_i\|^2_1 +\sum_{i=1}^m\|p_i\|_1^2}$, where $\|\cdot\|_1 $ is the $\ell_1$-norm, prox-function
$d_\u(\u) = \sum_{i=1}^m  \la x_i,\log x_i \ra +\sum_{i=1}^m \la p_i,\log p_i \ra$,  and the  corresponding  Bregman divergence
\[ B_\u(\u, \breve \u) = \sum_{i=1}^m  \langle x_i, \log (x_i /  \breve x_i) \rangle -\sum_{i=1}^m\boldsymbol 1_{n^2}^\top( x_i -  \breve x_i) + \sum_{i=1}^m  \langle p_i, \log (p_i /  \breve p_i) \rangle -\sum_{i=1}^m\boldsymbol 1_{n}^\top( p_i -  \breve p_i).\]
We define $R^2_\u = \max\limits_{\u \in  \U }d_\u(\u) - \min\limits_{\u \in  \U }d_\u(\u)$, 

 \end{itemize}
We consider mirror prox algorithm on   space $\U \times \V$ with the prox-function $ a d_\u(\u)+ b d_\v(\v)$ and the corresponding Bregman divergence  $a B_{\u}(\u, \breve \u) + b B_\v(\v,\breve \v)$, where   $a = \frac{1}{R_{\u}^2}$, $b =\frac{1}{R_{\v}^2} $.

 The  gradient operator for $F(\x,\p,\y,\z)$ is  defined by
\begin{align*}
    G(\x,\p,\y,\z) = 
    \begin{pmatrix}
  \nabla_\x F(\x,\p,\y,\z) \\
 \nabla_\p F(\x,\p,\y,\z)\\
 -\nabla_\y F(\x,\p,\y,\z) \\
 -\nabla_\z F(\x,\p,\y,\z)\\
    \end{pmatrix} = 
    \begin{pmatrix}
   \boldsymbol d + 2\|d\|_\infty \boldsymbol A^\top \y \\
   \gamma\WW^\top\z -   2\|d\|_\infty \{[y_{i}]_{1...n}\}_{i=1}^m \\
  -  2\|d\|_\infty(\A\x - \b)   \\
   -\gamma \WW\p
    \end{pmatrix}.
\end{align*}
Here $[y_{i}]_{1...n}$ is the first $n$ component of vector $y_i\in[-1,1]^{2n}$, and $\{[y_{i}]_{1...n}\}_{i=1}^m$ is a short form of $([y_{1}]_{1...n},[y_{2}]_{1...n},..., [y_{m}]_{1...n})$.

\begin{lemma}\label{lm:Lipsch1}
Objective $F(\u,\v)$ in \eqref{eq:def_saddle_prob2} is $ (L_{\u\u},L_{\u\v}, L_{\v\u}, L_{\v\v})$-smooth with $L_{\u\u}=L_{\v\v}=0$ and $L_{\u\v}=L_{\v\u} = \sqrt{{8\|d\|^2_\infty}+ \gamma\lm_{\max} (W)^2} $.
\end{lemma}

\begin{proof}[Proof of Lemma \eqref{lm:Lipsch1}]
As $F(\u,\v)$ is bilinear, $L_{\u\u}=L_{\v\v}=0$. Next, we estimate $L_{\u\v}$ and $L_{\v\u}$.
By the definition of $L_{\u\v}$  and the spaces $\U,\V$ we have
\begin{equation}\label{eq:Ldef_conj}
\|\nabla_\u F(\u,\v) - \nabla_\u F(\u,\v')\|_{\U^*}
       \leq L_{\u\v} \|\v-\v' \|_{2}.
       \end{equation}
 From the definition of dual norm, it follows
   \begin{align*}
   &\|\nabla_\u F(\u,\v) - \nabla_\u F(\u,\v')\|_{\U^*} = \max_{\|\u\|_\U \leq 1}\la\u, \nabla_\u F(\u,\v) - \nabla_\u F(\u,\v')\ra.
   \end{align*}
   From this and \eqref{eq:Ldef_conj} we get
    \begin{equation}\label{eq:Lreform}
       \max_{\|\u\|_\U \leq 1}\la \u, \nabla_\u F(\u,\v) - \nabla_\u F(\u,\v')\ra
       \leq L_{\u\v} \|\v-\v' \|_{2}.  
    \end{equation}

By the definition of $F(\cdot)$ and $\U = \X\times \Pp$ we have
\begin{align*} \nabla_\u F = \begin{pmatrix}
  \nabla_\x F \\
 \nabla_\p F\\
    \end{pmatrix} 
    &= 
    \begin{pmatrix}
   \boldsymbol d + 2\|d\|_\infty \boldsymbol A^\top \y \\
  \gamma \WW^\top\z -    2\|d\|_\infty  \{[y_{i}]_{1...n}\}_{i=1}^m 
    \end{pmatrix}.  \end{align*}
   From this and $\V \triangleq \Y\times \Z$,  
    \begin{align*} 
    \nabla_\u F(\u,\v) - \nabla_\u F(\u,\v')
    &=  \begin{pmatrix}
2\|d\|_\infty\boldsymbol A^ \top ( \y - \y') \\
  \gamma \WW^\top(\z-\z') -    2\|d\|_\infty (\{[y_{i} - y'_{i}]_{1...n}\}_{i=1}^m)
    \end{pmatrix} \\
    &= 
    \begin{pmatrix}
  2\|d\|_\infty \boldsymbol A &  -   2\|d\|_\infty 
    \mathcal E  \\
 0_{mn \times mn^2} & \gamma\WW
    \end{pmatrix}^\top 
    \begin{pmatrix}
        \y - \y'  \\
   \z - \z'
    \end{pmatrix}, 
    \end{align*}
    where $\mathcal E \in \{1,0\}^{2mn\times mn}$ is block-diagonal matrix
    \[
    \mathcal E = 
\begin{pmatrix}
\begin{pmatrix}
       I_n \\
         0_{n\times n} 
\end{pmatrix} &   \cdots &  0_{2n\times n}\\
\vdots  & \ddots  & \vdots  \\
0_{2n\times n} &  \cdots &  \begin{pmatrix}
       I_n \\
         0_{n\times n} 
\end{pmatrix}
\end{pmatrix}.\]
      From this it follows that $\nabla_\u F(\cdot) $ is linear function in $\v - \v'$, then \eqref{eq:Lreform}  can be rewritten as
\begin{align}\label{eq:Lcong2nor}
L_{\u\v}  &= \max_{\|\v - \v'\|_2\leq 1} \max_{\|\u\|_\U\leq 1} \left\la  \u,   \begin{pmatrix}
 2\|d\|_\infty \boldsymbol A &  -  2\|d\|_\infty 
    \mathcal E  \\
 0_{mn \times mn^2} & \gamma \WW
    \end{pmatrix}^\top 
    \begin{pmatrix}
        \v - \v' 
    \end{pmatrix}\right \ra.          
    \end{align}
By the same arguments we can get the same expression for $L_{\v\u}$ up to rearrangement of maximums.
Next, we use the fact that the $\ell_2$-norm is the conjugate norm for the $\ell_2$-norm. From this and \eqref{eq:Lcong2nor} it follows
\begin{equation}\label{eq:Lxy}
  L_{\u\v} = \hspace{-0.1cm}\max_{\|\u\|_\U\leq 1} \left\| \begin{pmatrix}
  2\|d\|_\infty\boldsymbol A &  -  2\|d\|_\infty
    \mathcal E  \\
 0_{mn \times mn^2} & \gamma\WW
    \end{pmatrix} \u \right\|_2. 
\end{equation} 

After that, we write
\begin{align}\label{eq:complete_eq_forL}
    \max_{\|\u\|_\U\leq 1}~
    &
    \left\| 
    \begin{pmatrix}
    2\|d\|_\infty \boldsymbol A &  -   2\|d\|_\infty 
    \mathcal E  \\
    0_{mn \times mn^2} &\gamma \WW
    \end{pmatrix} 
    \begin{pmatrix} \x \\ \p \end{pmatrix}
    \right \|_2^2 \\
    &=  \max_{\|\u\|^2_\U\leq 1} 
    \left\|\begin{pmatrix}
    2\|d\|_\infty(\boldsymbol A \x - \mathcal E \p )\\
  \gamma  \WW\p
    \end{pmatrix}  \right \|_2^2 \notag\\
    &=
    \max_{\|\x\|^2_\X + \|\p\|^2_\Pp \leq 1} 
    \left(  4\|d\|^2_\infty \|\boldsymbol A \x - \mathcal E \p \|_2^2  +\gamma^2 \|\WW\p \|_2^2 \right)\notag\\
    &\leq 
  4\|d\|^2_\infty  \max_{\|\x\|^2_\X + \|\p\|^2_\Pp\leq 1}\|\boldsymbol A \x - \mathcal E \p \|_2^2 \notag +  \gamma^2\max_{\|\p\|^2_\Pp\leq 1}\|\WW\p \|_2^2.
\end{align}

We consider the first term of the r.h.s. of \eqref{eq:complete_eq_forL} under the minimum
\begin{align}\label{eq:Axnorm}
\|\boldsymbol A \x - \mathcal E \p \|_2^2 
&= \sum_{i=1}^{m}\left\|   Ax_i -  \begin{pmatrix}
p_i\\
0_n
\end{pmatrix} \right\|^2_2 \leq \sum_{i=1}^{m}\|   Ax_i\|_2^2 + \sum_{i=1}^{m}\|p_i\|_2^2.
\end{align}
The last bound holds due to $\la A x_i, (
p_i^\top,
0_n^\top) \ra \geq 0$ as the entries of $A,\x,\p$ are non-negative.
Next we take the minimum in \eqref{eq:Axnorm}
\begin{align}\label{eq:normest}
\max_{\|\x\|^2_\X + \|\p\|^2_\Pp\leq 1}
\|\boldsymbol A \x - \mathcal E \p \|_2^2 
&= \max_{\sum_{i=1}^m \left(\|x_i\|_1^2 + \|p_i\|_1^2\right) \leq 1}\|\boldsymbol A \x - \mathcal E \p \|_2^2  \notag\\
&\stackrel{\eqref{eq:Axnorm}}{\leq}\max_{\alpha \in \Delta_{2m}} \left( \sum_{i=1}^{m} \max_{\|x_i\|_1 \leq \sqrt{\alpha_{i}}} \|   Ax_i\|_2^2 + \sum_{i=1}^{m} \max_{\|p_i\|_1\leq \sqrt{\alpha_{i+m}}} \|p_i\|_2^2\right) \notag\\
&=\max_{\alpha \in \Delta_{2m}} \left( \sum_{i=1}^{m} \alpha_{i} \max_{ \|x_i\|_1 \leq 1}   \|   Ax_i\|_2^2 +\sum_{i=1}^{m}\alpha_{i+m}
\max_{\|p_i\|_1\leq 1} \|p_i\|_2^2\right).
\end{align}
By the definition of incidence matrix $A$ we get
$
Ax_i = (h_1^\top, h_2^\top)
$, where $h_1$ and $h_2$ such that $\boldsymbol 1^\top h_1 = \boldsymbol 1^\top h_2= \sum_{j=1}^{n^2} [x_i]_j$ = 1 as $x_i \in \Delta_{n^2}  ~\forall i =1,...,m$.
Thus,
\begin{equation}\label{eq:Ax22}
  \|A x_i\|_2^2 = \|h_1\|_2^2 + \|h_2\|_2^2 \leq \|h_1\|_1^2 + \|h_2\|_1^2 = 2.
  \end{equation}
As $p_i \in \Delta_n, \forall i =1, ...,m$ we have
\begin{equation}\label{eq:p_estim}
    \max_{\|p_i\|_1\leq 1} \|p_i\|_2^2 \leq \max_{\|p_i\|_1\leq 1} \|p_i\|_1^2 = 1.
\end{equation}
Using  \eqref{eq:Ax22} and \eqref{eq:p_estim} in \eqref{eq:normest} we get
\begin{align}\label{eq:Adefnormfind}
\max_{\|\x\|^2_\X + \|\p\|^2_\Pp\leq 1}
&\|\boldsymbol A \x - \mathcal E \p \|_2^2\leq 
\max_{\alpha \in \Delta_{2m}} \left( 2\sum_{i=1}^{m} \alpha_{i}  
 + \sum_{i=m+1}^{2m} \alpha_{i}  \right) \leq \max_{\alpha \in \Delta_{2m}}  2 \sum_{i=1}^{2m} \alpha_{i} = 2.
\end{align}

 Now we consider the second term of the r.h.s. of \eqref{eq:complete_eq_forL}.
 \begin{align}\label{eq:W calcnorm}
   \max_{\|\p\|^2_\Pp \leq 1}  \left\|\WW\p \right\|_2^2 
   =  \max_{\sum_{i=1}^m\|p_i\|^2_1 \leq 1}  \left\|\WW\p \right\|_2^2.
 \end{align}
 The set $\sum_{i=1}^m\|p_i\|^2_1 \leq 1$ is contained in the set $\sum_{j=1}^n\sum_{i=1}^m[p_i]^2_j \leq 1$  as cros-product terms of $\|p_i\|^2_1$ are non-negative. Thus, we can change the constraint in the minimum in \eqref{eq:W calcnorm} as follows
  \begin{align}\label{eq:W calcnormfin}
  \max_{\sum_{i=1}^m\|p_i\|^2_1 \leq 1}  \left\|\WW\p \right\|_2^2 
  &\leq   \max_{\sum_{j=1}^n\sum_{i=1}^m[p_i]^2_j \leq 1}  \left\|\WW\p \right\|_2^2 =\max_{ \|\p\|^2_2 \leq 1}  \left\|\WW\p \right\|_2^2 \\
  &= \max_{ \|\p\|_2 \leq 1}  \left\|\WW\p \right\|_2^2 \triangleq \lm_{\max}(\WW)^2 =\lm_{\max}(W)^2.
  \end{align}
 The last inequality holds due to $ \WW \triangleq W \otimes  I_n$ and the properties of the Kronecker product for eigenvalues.
Using\eqref{eq:Adefnormfind} and \eqref{eq:W calcnormfin} in \eqref{eq:complete_eq_forL} for the estimation of $  L_{\u\v}$ from \eqref{eq:Lxy}, we get \[L_{\u\v}=L_{\v\u} =  \sqrt{8\|d\|^2_\infty+ \gamma^2 \lm_{\max} (W)^2}.\]

\end{proof}

This lemma   allows us to obtain the following convergence result
\begin{theorem}\label{Th:first_alg11}
Let $\|\z\|^2_2 \leq R^2$, then $R_{\u}=\sqrt{3m\log n}$ and $R_\v = \sqrt{mn+R^2/2}$ with
\[
R^2 
 = \frac{\|\nabla_\p f(\x,\p^*,\y)\|^2_2}{\gamma\lm_{\min}^+\left(W\right)} 
\leq \frac{4mn\|d\|^2_\infty}{\gamma \lm^+_{\min}(W)},
\]
where $\lm^+_{\min}(W)$ is the minimal positive eigenvalue of $W$.
Then after $ N = {4L_{\u\v} R_\u R_\v}/{(m\e)}$ iterations,  Algorithm \ref{MP_WB_distr}   with 
$\eta   =\frac{1}{2L_{\u\v} R_\u R_\u}$
outputs a pair $(\widetilde \u, \widetilde \v)$ such that
\begin{align*}
    \max_{\substack{ \y \in \Y,\\ \|\z\|_2 \leq R }} F\left( \widetilde \u, \y,\z\right) - \min_{ \substack{ \x \in \X,\\ \p \in \Pp }} F\left(\x,\p,\widetilde \v\right)  \leq \e.
\end{align*}
The total complexity of Algorithm \ref{MP_WB_distr} per node is 
\[O\left(\frac{n^2}{\e}\sqrt{n\log n} \sqrt{\chi}
    \|d\|_\infty^{3/2}  \right).\]
\end{theorem}

\begin{proof}[Proof of Theorem \ref{Th:first_alg11}]
The constants of smoothness for $F(\u,\v)$ follows from  Lemma \ref{lm:Lipsch1}. The  bound on duality gap follows from the theory of Mirror-Prox with proper $R_\U$, $R_\V$ and $L_{\u\v}$, $L_{\u\v}$, $L_{\v\u}$, $L_{\v\v}$.

To estimate $R$, we calculate the $\ell_2$-norm of the objective in \dd{\eqref{eq:alm_distrNew}
\begin{align*}
\|\nabla_\p f(\x, \p^*, \y)\|_2^2 &= \left\| 2\|d\|_\infty \{[y_{i}]_{1...n}\}_{i=1}^m  \right\|_2^2 = \sum_{i=1}^m 4\|d\|^2_\infty\|[y_{i}]_{1...n}\|_2^2 \leq 4 mn \|d\|^2_\infty .
\end{align*}}
Thus, we get
\[
R^2 
 = \frac{\|\nabla_\p f(\x,\p^*,\y)\|^2_2}{\gamma\lm_{\min}^+\left(W\right)} 
\leq \frac{4mn\|d\|^2_\infty}{\gamma\lm^+_{\min}(W)}.
\]
To simplify the expression for $R_\u$, $R_\v$ and $L_{\u\v}$ we use that for any $a,b$, $a+b\leq 2\max\{a,b\}$ and $\sqrt{a^2+b^2}\leq \sqrt 2 \max\{a,b\}$:
\begin{align*}
    &R_\v = mn+R^2/2 \leq \sqrt{2mn}\max\left\{1,\frac{2\|d\|_\infty}{\sqrt{\gamma \lm_{\min}^+(W)}} \right\}\\
    &L_{\u\v} \leq  \sqrt{2} \max\{\sqrt{8}\|d\|_\infty,\gamma \lm_{\max}(W)\}\\
\end{align*}
The complexity of one iteration of Alg. \ref{MP_WB_distr}  per node is $O\left(n^2\right)$ as the number of non-zero elements in matrix A is $2n^2$. Multiplying this by the number of iterations $N$ we get 
\begin{align*}
     O(n^2N) &=O\left(n^2 L_{\u\v}R_{u}R_v/(m\e) \right) \\
     &= O\left(\frac{n^2}{m\e} \sqrt{m \log n}\sqrt{mn}\max\left\{1,\frac{2\|d\|_\infty}{\sqrt{\gamma \lm_{\min}^+(W)}} \right\} \max\{\sqrt{8}\|d\|_\infty,\gamma \lm_{\max}(W)\}\right)  \\
     &=O\left(\frac{n^2}{\e}\sqrt{n\log n} \max\left\{1,\frac{2\|d\|_\infty}{\sqrt{\gamma \lm_{\min}^+(W)}} \right\} \max\{\sqrt{8}\|d\|_\infty,\gamma \lm_{\max}(W)\}  \right).
\end{align*}
We can minimize this expression over $\gamma$ to get the minimal total complexity. We take $\gamma =\frac{\sqrt{8}\|d\|_\infty}{\lm_{\max(W)}} $
and we get the final statement 
\begin{equation}
         O\left(\frac{n^2}{\e}\sqrt{n\log n} \sqrt{\chi}
    \|d\|_\infty^{3/2}  \right),
\end{equation}
where we used  the notation of the condition number for matrix $W$: $\chi = \frac{\lm_{\max}(W)}{\lm_{\min}^+(W)}$.
\end{proof}


 
\begin{algorithm}[ht!]
    \caption{Decentralized Mirror-Prox for Wasserstein Barycenters}
    \label{MP_WB_distr}
    \footnotesize
    \begin{algorithmic}[1]
\Require measures $q_1,...,q_m$, linearized cost matrix $d$, incidence matrix $A$, step $\eta$, starting points $p^1=\frac{1}{n}\boldsymbol 1_{n}$, $x_1^1=...= x_m^1 = \frac{1}{n^2}\boldsymbol 1_{n^2}$, $y_1^1 = ... =y_m^1 =\boldsymbol 0_{2n}$
        \State   $\alpha = 2\|d\|_\infty \eta (mn + R^2/2)/m  $,
         $\beta =6  \|d\|_\infty \eta \log n$,
         $\kappa = 3 \eta  \log n$, $\theta = \eta (mn + R^2/2)/m$
        \For{ $k=1,2,\cdots,N-1$ }
         \For{ $i=1,2,\cdots, m$}
         
         \State  
\[  u^{k+1}_i =\frac{ x^{k}_i \odot  \exp\left\{ 
       - \kappa \left( d+ 2\|d\|_\infty A^\top y^{k}_i \right)
        \right\}}{\sum\limits_{l=1}^{n^2} [x^{k}_i]_l\exp\left\{  
       - \kappa \left( [d]_l+ 2\|d\|_\infty [A^\top y^{k}_i]_l \right)
        \right\}}  \]
      \State  
    \[s_i^{k+1} = \frac{p_i^{k} \odot \exp\left\{ 
        \beta
         [y^k_i]_{1...n} - 3 \eta \log n \sum_{j=1}^m \gamma W_{ij}z^k_j
        \right\}}{\sum_{l=1}^n [p_i^{k}]_l\exp\left\{  
        \beta
        [y^k_i]_{l} - 3 \eta \log n \left[\sum_{j=1}^m \gamma W_{ij}z^k_j\right]_l
        \right\}} \]      
         \State  
         \[v_i^{k+1} = y^k_i + \alpha\left(  
         A x_i^k - 
         \begin{pmatrix}
         p_i^k\\
         q_i
         \end{pmatrix}
        \right), \quad \text{project } v_i^{k+1} \text{ onto } [-1,1]^{2n}.\]
        
\State \[\lm_i^{k+1} = z^k_i + \theta \sum_{j=1}^m \gamma W_{ij} p^k_j\]
             \State  \[
         x^{k+1}_i =\frac{ x^{k}_i\odot \exp\left\{ 
       - \kappa \left( d+ 2\|d\|_\infty A^\top v^{k+1}_i \right)
        \right\}}{\sum\limits_{l=1}^{n^2} [x^{k}_i]_l\exp\left\{  
       - \kappa \left( [d]_l+ 2\|d\|_\infty [A^\top v^{k+1}_i]_l \right)
        \right\}} \]
        \State 
        \[  p_i^{k+1} = \frac{p_i^{k} \odot \exp\left\{ 
        \beta
       [v^{k+1}_i]_{1...n} - 3 \eta \log n \sum_{j=1}^m \gamma W_{ij} \lm^{k+1}_j
        \right\}}{\sum_{l=1}^n [p_i^{k}]_l\exp\left\{  
        \beta
        [v^{k+1}_i]_{l} - 3\eta \log n \left[\sum_{j=1}^m \gamma W_{ij}\lm^{k+1}_j\right]_l
        \right\}}\]
        \State
        \[y_i^{k+1} = y^k_i + \alpha\left( A u_i^{k+1}-
        \begin{pmatrix}
          s_i^{k+1}\\
          q_i
        \end{pmatrix}
        \right), \quad \text{project } y_i^{k+1} \text{ onto } [-1,1]^{2n}. \]
        
     \State \[z_i^{k+1} = z^k_i + \theta \sum_{j=1}^m \gamma W_{ij} s^{k+1}_j\]

          \EndFor
           \EndFor
        \Ensure 
                   $\widetilde \u = \frac{1}{N}\sum\limits_{k=1}^{N} 
                   \begin{pmatrix}
          u_1^k\\
         \vdots\\
        u_m^k\\
     s_1^k\\
         \vdots,
         s_m^k 
         \end{pmatrix}$, 
           $\widetilde \v = \frac{1}{N}\sum\limits_{k=1}^{N}
          \begin{pmatrix}
           v_1^k\\
         \vdots\\
        v_m^k\\
        \lm_1^k\\
         \vdots\\
      \lm_m^k
      \end{pmatrix}$
    \end{algorithmic}
\end{algorithm}

\subsection{Experiments}
Next, we illustrate the work of  Algorithm \ref{MP_WB_distr}. We randomly generated  10 Gaussian measures   with equally spaced support of 100 points in $[-10, -10]$,  mean from  
 $[-5, 5]$ and variance from $[0.8, 1.8]$.
We studied the convergence of calculated barycenters to the theoretical true barycenter \cite{delon2020wasserstein} on the Erd\H{o}s-R\'enyi random graph with probability of edge creation $p=0.5$.
Figure \ref{fig:comparis_conv1} shows the 
the convergence of Algorithm 
\ref{MP_WB_distr} with respect to the function optimality gap and consensus gap. 
  The slope ration $-1$ on logarithmic scale  fits theoretical dependence   of  the desired accuracy $\e$ on  number of iterations ($N \sim \e^{-1}$,  Theorem \ref{Th:first_alg11}).

\begin{figure}[ht!]
    \begin{subfigure}[b]{.5\textwidth}
  \centering
  \includegraphics[width=1\linewidth]{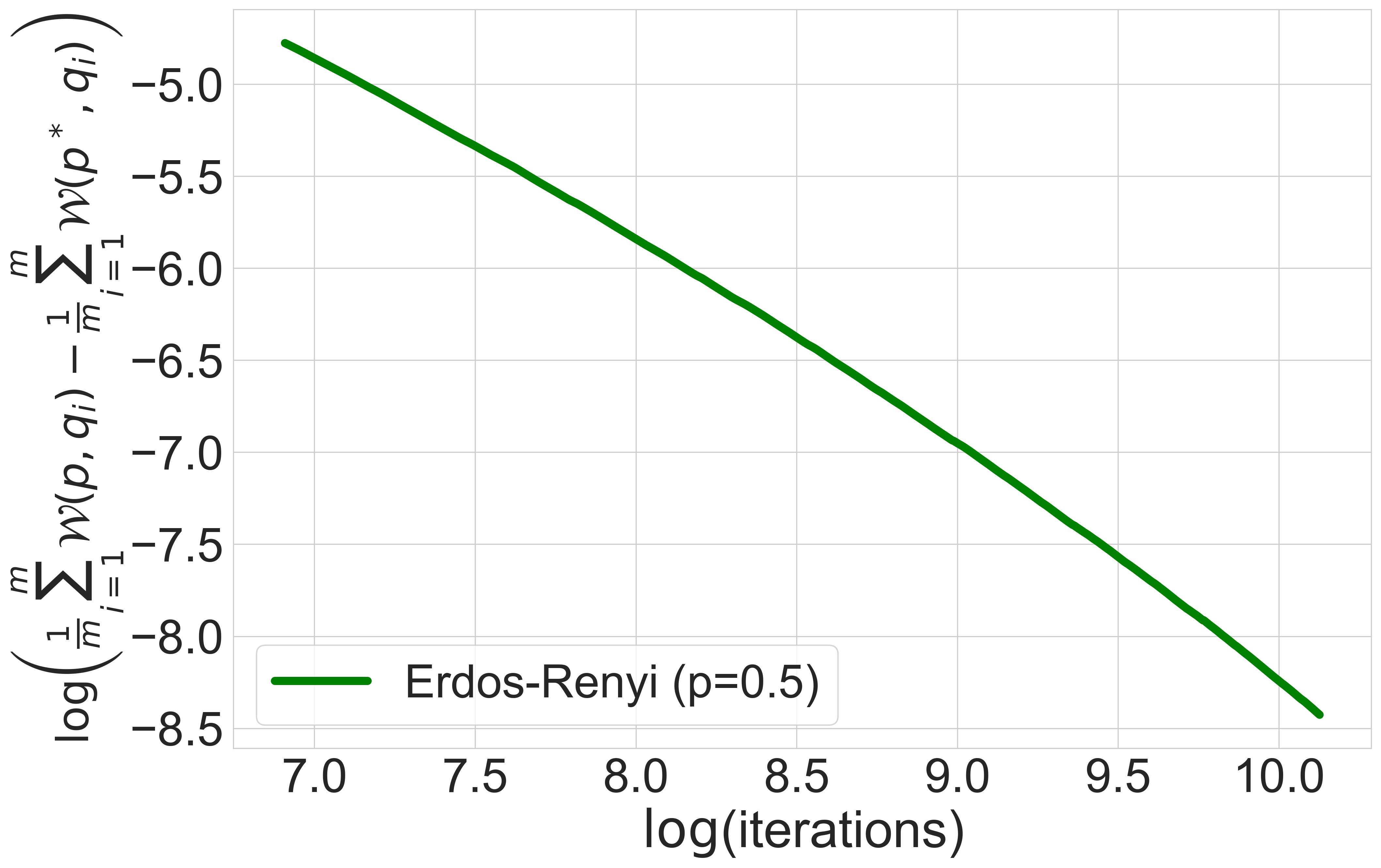}
\end{subfigure}
\begin{subfigure}[b]{.5\textwidth}
  \centering
  \includegraphics[width=1\linewidth]{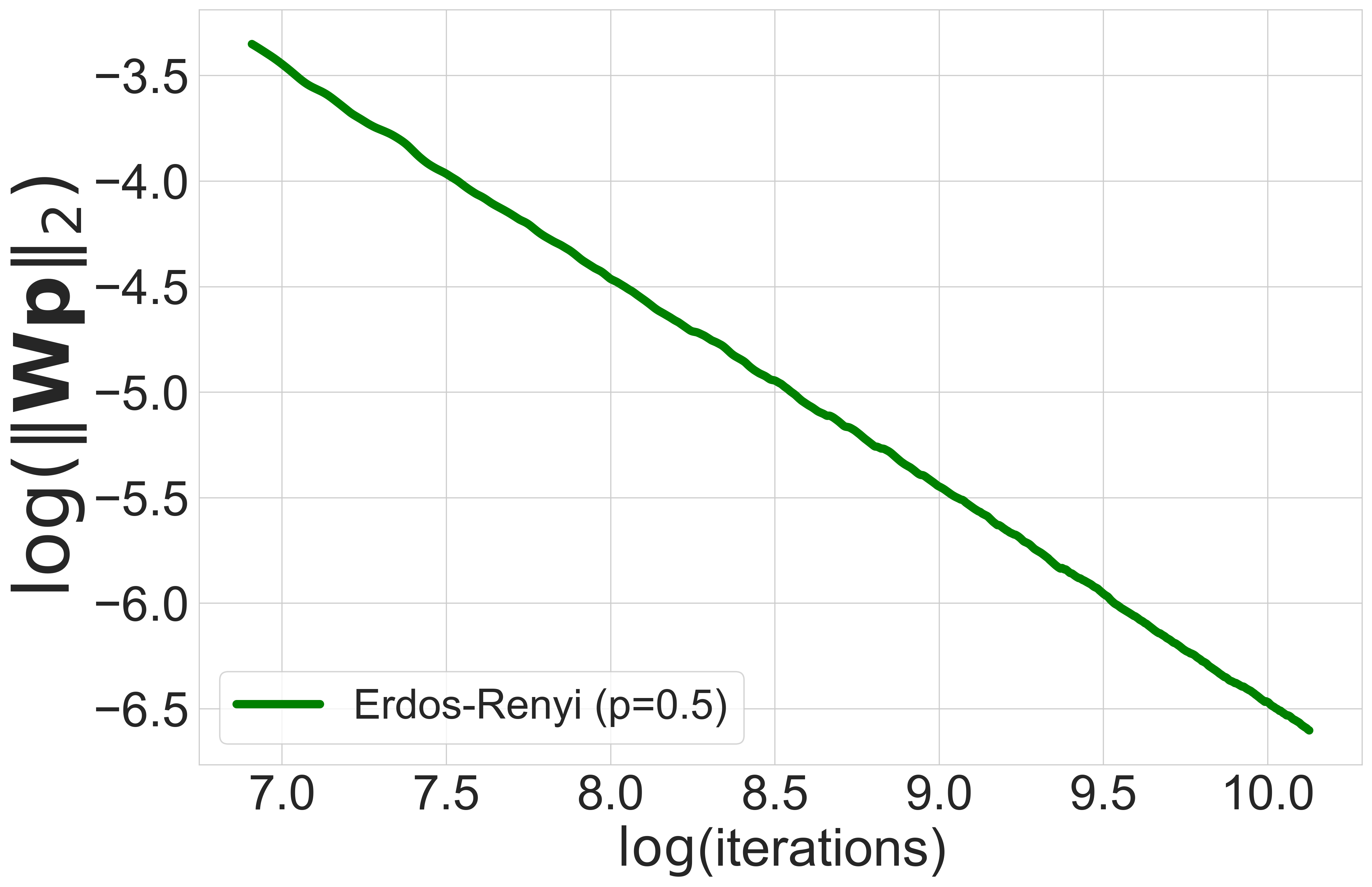}
\end{subfigure}
\vspace{-0.2cm}
\caption{Convergence of Decentralized Mirror-Prox for Wasserstein Barycenters }
\label{fig:comparis_conv1}
\end{figure}

\chapter{Decentralized Algorithms for Stochastic Optimization}\label{ch:decentralized}

This Chapter has interests other than Wasserstein barycenters, it presents the optimal  bounds on the number of communication rounds and dual oracle calls of the gradient of the dual objective per node in the problem of minimizing the sum of  strongly convex and Lipschitz smooth functions. This chapter complements Chapter \ref{ch:dual} for the case of additionally Lipschitz smooth (stochastic) objectives.

We consider  minimizing the average of  functions
in a distributed manner 
\begin{equation}\label{Sum}
\min_{x\in \R^n } \frac{1}{m}\sum_{i=1}^m f_i(x),
\end{equation}
where $f_i(x)$'s
are $\gamma$-strongly convex and  $L$-Lipschitz smooth. We assume that each $f_i(x)$ has the Fenchel--Legendre representation \[f_i(x) = \max\limits_{y \in \R^n}\left\{\langle x,y \rangle - \psi_i(y)\right\}\] with convex $\psi_i(y)$.  The case when $f_i$'s are dual-friendly (have the Fenchel--Legendre representation)  is the case of  the Wasserstein barycenter problem (see  Chapter \ref{ch:dual}).

\section[Dual Approach for Optimization Problem with Affine Constraints]{Dual Approach for Optimization Problem with Affine Constraints}\label{sec:decentr_deter}

Similarly to Chapter \ref{ch:dual}, we
 firstly derive  (stochastic)  dual algorithms for a general minimization problem with affine constrains  where the objective is strongly convex and Lipschitz smooth,  and then we show how to execute these algorithms in a decentralized setting for problem \eqref{Sum}.

We consider a general minimization problem with affine constrains 
\begin{equation}\label{PP}
\min_{Ax=b, ~x \in \R^n} F(x),    
\end{equation}
where $F(x)$ has $L_F$-Lipschitz continuous gradient and $\gamma_F$-strongly convex in the $\ell_2$-norm, ${\rm Ker} A \neq \varnothing$.
Let $x^* = \arg\min\limits_{\substack{Ax=b,\\ x \in \R^n}} F(x).$

\begin{remark}
We notice that turning to the dual problems does not oblige us using the dual oracle. Instead, we can use the primal oracle and the Moreau theorem \cite{Rockafellar2015} with the Fenchel--Legendre representation. The corresponding maximization problem can be solved  using the first-order oracle for the primal objective. However, such approach will not contribute to obtaining the optimal bounds on the  number of primal first-order oracle calls.
\end{remark}

The dual problem (up to a sign) to \eqref{PP} is the following
\begin{align}\label{DP}
\min_{y \in \R^n} \Psi(y) \triangleq \max_{x\in \R^n}\left\{\langle y,Ax-b\rangle - F(x)\right\},
\end{align}
where 
 $\Psi$  is $L_{\Psi}$--Lipschitz smooth  with $L_{\Psi} = \frac{\lambda_{\max}(A^T A)}{\gamma_F}$ and $\gamma_{\Psi}$--strongly convex with $\gamma_{\Psi} = \frac{\lambda_{\min}^{+}(A^T A)}{L_F}$ in the $\ell_2$-norm in $y^0 + ({\rm Ker} A^T)^{\perp}$.\footnote{Since ${\rm Im} A  = ({\rm Ker} A^T)^{\perp}$ we will have that all the points $\tilde{y}^k,z^k,y^k,$ generated by fast gradient method and \dm{methods based on fast gradient method}, belong to $y^0 + ({\rm Ker} A^T)^{\perp}$. That is, from the point of view of estimates this means, that we can consider $\Psi$ to be $\gamma_{\Psi}$-strongly convex everywhere.}  The Lipschitz smoothness of dual objective $\Psi$ follows from strong convexity of  the primal objective $F(x)$ (Theorem \ref{th:primal-dualNes}), the strong convexity of  $\Psi$ follows from
Lipschitz smoothness of $f(x)$
 \cite{kakade2009duality,Rockafellar2015}.

 By Demyanov--Danskin theorem we have
\begin{align}\label{eq:demyanov_danDual}
    \nabla \Psi(y) = Ax(A^\top y)-b,  
\end{align}
where 
\begin{equation}\label{eq:demian_eq_x}
    x(A^T y) = \arg\max\limits_{x \in \R^{n}}  \left\{ \la  y, Ax-b\ra - F(x) \right\}.
\end{equation}
Let us define $y^* = \arg\min\limits_{y \in \R^n} \Psi(y)$.

For strongly convex objective, fast gradient method (see Chapter \ref{ch:dual} for its stochastic version) is not a primal-dual method \cite{nesterov2009primal-dual, nemirovski2010accuracy}, hence, it cannot be used to solve the primal-dual pair of problems \eqref{PP} and \eqref{DP}.  The restart technique \cite{juditsky2014deterministic,nemirovskij1983problem,gasnikov2017modern} cannot be also used because  the radius of a solution is not a distance from a starting point:  $R_y \triangleq \|y^0\|_2 + \|y^0 - y^*\|_2$ (usually we take $y^0 = 0$).
The next theorem provides a method to solve  the primal-dual pair of problems \eqref{PP} and \eqref{DP}.

\begin{theorem}\label{eq:thero1decen}
Let the objective $F(x)$ in \eqref{PP} be $L_F$-Lipschitz smooth and $\gamma_F$-strongly convex in the $\ell_2$-norm.
Let   $y^N$  be an output of the OGM-G algorithm \cite{kim2021optimizing}. Let $R_y$ be such that $\|y^*\|_2 \leq R_y$, where $y^*$ is the solution of \eqref{DP}. Then after 
\begin{equation}\label{eq:oraclefisrtcall}
\widetilde O\left(\sqrt{\frac{L_F}{\gamma_F}\chi(A^\top A) }\right)
\end{equation}
iterations (number of oracle calls of $\nabla \Psi(y)$),
 the following holds for $ x^N = x(A^T y^N)$
\begin{equation}\label{eq:toprove_toget}
F(x^N)-F(x^*)\le \e, \qquad \|Ax^N-b\|_2 \le \e/R_y.
\end{equation}
\end{theorem}
\begin{proof}
Let $\Psi(y)$ be dual function for $F(x)$ defined in \eqref{DP}.  $\Psi$  is $L_{\Psi}$--Lipschitz smooth  with $L_{\Psi} = \frac{\lambda_{\max}(A^T A)}{\gamma_F}$ and $\gamma_{\Psi}$--strongly convex with $\gamma_{\Psi} = \frac{\lambda_{\min}^{+}(A^T A)}{L_F}$ in the $\ell_2$-norm in $y^0 + ({\rm Ker} A^T)^{\perp}$.
Then we have \cite{allen2018make,anikin2017dual,nesterov2012make}
\begin{equation}\label{grad_norm}
F(x(A^T y))- F(x^*) 
\le\langle\nabla \Psi(y), y\rangle = \langle Ax(A^\top y)-b, y\rangle.    
\end{equation}
Let   $x^N$ and $y^N$  be outputs of  an  algorithm solving the pair of primal-dual problems \eqref{PP} and \eqref{DP} and let $R_y$ be such that $\|y^*\|_2 \leq R_y$. Since the dual objective  is strongly convex, the following relation for $x^N$ and $y^N$ holds 
\begin{equation}\label{eqLprduality}
  x^N = x(A^Ty^N).  
\end{equation}
We have
\begin{align*}
F(x^N)-F(x^*)=F(x(A^T y^N))-F(x^*)
\stackrel{\eqref{grad_norm}}{\leq}   
\langle \nabla \Psi(y^N), y^N\rangle &\leq \|\nabla \Psi(y^N)\|_2\|y^N\|_2 \notag\\
&\leq 2R_y\|\nabla \Psi(y^N)\|_2,
\end{align*}
where we used the Cauchy–-Schwarz  inequality and \ga{($\|y^N\|_2\le 2R_y$)}. 
Hence, to get $F(x^N)-F(x^*)\le \e$ from \eqref{eq:toprove_toget} we need to prove
\begin{equation}\label{GN}
\|\nabla \Psi(y^N)\|_2\le \e/(2 R_y).
\end{equation}
Moreover, from \eqref{eq:demyanov_danDual} and \eqref{eqLprduality}  it follows that $\nabla \Psi(y^N) = Ax^N-b$. Thus if we get \eqref{GN}, we prove \eqref{eq:toprove_toget}.

In order to prove this, we refer to a method which converges in term of the norm of the gradient, for instance, OGM-G \cite{kim2021optimizing}. It has the following convergence rate
\[\|\nabla \Psi(y^N)\|_2 = O\left(\frac{L_{\dm{\Psi}}\ga{\|y^0 - y^*\|_2}}{N^2}\right)\dm{ = O\left( \frac{L_{\Psi} \|\nabla \Psi(y^0)\|_2}{\gamma_\Psi N^2} \right)},\]
where we used 
 \[\frac{\gamma}{2}\|y^0-y^*\|^2_2 \leq \Psi(y^0) - \Psi(y^*) \leq \frac{1}{2\gamma_{\Psi}} \|\nabla \Psi(y^0)  \|_2^2.\]
Thus, after  $\bar N =  \widetilde O\left(\sqrt{\frac{L_{\Psi}}{\gamma_{\Psi}}}\right)$ iterations of OGM-G we will have 
\[\|\nabla \Psi(y^{\bar{N}})\|_2 \le \frac{1}{2}\|\nabla \Psi(y^0)\|_2.\]
So after  $l = \log_2\left(\gav{\|}\nabla \Psi(y^{0})\|_2 \frac{R_y}{\e}\right)$ restarts ($y^0 : = y^{\bar{N}}$) we will obtain
\eqref{GN}.  This approach  requires \[O\left(\sqrt{\frac{L_{\Psi}}{\gamma_{\Psi}}}\log\left(\|\nabla \Psi(y^{0})\|_2 \frac{R_y}{\e}\right)\right)\]
number of oracle calls of $\nabla \Psi(y)$ (that is  $Ax(A^T y)-b$ ). 
Using $L_{\Psi} = \frac{\lambda_{\max}(A^T A)}{L}$ and  $\gamma_{\Psi} = \frac{\lambda_{\min}^{+}(A^T A)}{L}$, we obtain
\[
\widetilde O\left( \sqrt{\frac{L_\Psi}{\gamma_\Psi} } \right) = \widetilde O\left( \sqrt{\frac{\lm_{\max}(A^\top A)/\gamma}{\lm_{\min}(A^\top A)/L} } \right) = \widetilde O\left( \sqrt{\frac{L_F}{\gamma_F} \chi(A^\top A)} \right).
\]
\end{proof}

The same result with the replacement \[\dm{\sqrt{\frac{L_{\Psi}}{\gamma_{\Psi}}}}\log\left(\|\nabla \Psi(y^{0})\|_2 \frac{R_y}{\e}\right) \to \dm{\sqrt{\frac{L_{\Psi}}{\gamma_{\Psi}}}}\log\left(2L_{\Psi}^2\frac{R^4_y}{\e^2}\right)\] 
can be obtained by using  fast gradient method for Lipschitz smooth dual objective (but not strongly convex) with \dm{ bound $\dm{\sqrt{\frac{L_{\Psi}}{\gamma_{\Psi}}}}\log \left(\frac{L_\Psi R_y^2}{\e'} \right)$  \cite{nesterov2010introduction}} and desired accuracy $\e' = \frac{\e^2}{2L_{\Psi}R^2_y}$. This follows from 
\[\frac{1}{2L_{\Psi}}\|\nabla \Psi(y^N)\|_2^2 \le \Psi(y^N) - \Psi(y^*) \dm{\leq \e'}.\]

\section[Stochastic Dual Approach for Optimization Problem with Affine Constraints]{Stochastic Dual Approach for Optimization Problem with Affine Constraints}\label{sec:decentr_stoch}

\dm{Now we assume that we are given stochastic oracle $\nabla \Psi(y,\xi)$ with sub-Gaussian variance $\sigma^2_{\Psi}$} \cite{jin2019short}.
\begin{align*}
   &\E \nabla\Psi( y, \xi) = \nabla\Psi (y) \notag \\  
   &\E \exp \left( {\|\nabla\Psi( y, \xi) - \nabla\Psi (y)\|^2_2}/{\sigma_\Psi^2}\right) \leq \exp(1).
\end{align*}
Now  \dm{we} consider a  method form \cite{foster2019complexity} called RRMA+AC-SA$^{2}$  (see also \cite{allen2018make} in the non-accelerated but composite case). This algorithm converges as follows (for simplicity we skip polylogarithmic factors and high probability terminology) 
\[\|\nabla \Psi(y^N)\|_2^2 = \widetilde{O}\left(\frac{L^2_{\Psi}\|y^0 - y^*\|_2^2}{N^4} + \frac{\sigma^2_{\Psi}}{N} \right) = \widetilde{O}\left(\frac{L^2_{\Psi}\|\nabla \Psi(y^0)\|_2^2}{\gamma^2_{\Psi} N^4} + \frac{\sigma^2_{\Psi}}{N} \right).\]
 If we use restart technique of size $\bar{N} = \tilde{O}\left(\sqrt{\frac{L_{\Psi}}{\gamma_{\Psi}}}\right)$   and  batched gradient  with batch size  
$$r_{k+1} = \tilde{O}\left(\frac{\sigma^2_{\Psi}}{\bar{N}\|\nabla \Psi(\bar{y}^{k+1})\|_2^2}\right),$$
where  $\bar{y}^{k}$ is the output from the previous restart,
then after $l = O\left(\log_2 \left(\|\nabla \Psi(y^0)\|_2 \frac{R_y}{\e}\right)\right)$ restarts we will get 
\[\|\nabla \Psi(\bar{y}^l)\|_2 \le \e/R_y.\]
Therefore, the total number of stochastic dual oracle calls will be
\begin{equation}\label{eq:oracle)calldual}
   \widetilde{O}\left(\frac{\sigma^2_{\Psi}R^2_y}{\e^2}\right). 
\end{equation}
Note that the same bound 
takes place in the non-strongly convex case ($\gamma_{\Psi} = 0$). From \cite{allen2018make,jin2019short} it is known that this bound cannot be improved.

\section{Decentralized  Optimization}\label{dec}

Next we apply the results for minimizing the average of the
 functions in a distributed setting
\begin{equation}\tag{P1}\label{P1}
  \min_{x\in  \R^n} F(x) \triangleq \frac{1}{m}\sum_{i=1}^m f_i(x), 
\end{equation}
where  $f_i$'s are  $L$-Lipschitz smooth and $\gamma$-strongly convex. 
We seek to solve \eqref{P1}  on a network of agents in a decentralized  manner. To do so, we similarly to Chapter \ref{ch:dual} equivalently rewrite \eqref{P1} using communication matrix  $\WW$ defined in \eqref{eq:matrixWdef}
as follows  
\begin{equation}\tag{P2}\label{P2}
  \min_{\substack{\sqrt{\WW}\x=0, \\ x_1,\dots, x_m \in \R^n }} F(\x) \triangleq \frac{1}{m}\sum_{i=1}^m f_i(x_i),
\end{equation}
$\x = (x_1^\top, x_2^\top,, ..., x_n^\top)^\top$ is the stack column vector.
We also consider a stochastic version of  problem (P2), whose objectives $f_i$'s are given by their expectations: $f_i(x_i) = \E f_i(x_i,\xi_i)$.
If   $f_i$'s are  dual-friendly  
then we can construct the dual problem to problem \eqref{P2} with dual Lagrangian variable   $\y = [y_1^T \in \R^n,\cdots,y_m^T \in \R^n]^T \in \R^{mn}$  

\begin{equation}\label{D2}\tag{D2}
\min_{\y \in \R^{mn}}\Psi(\y) \triangleq  \frac{1}{m}\sum_{i=1}^{m}
\psi_i(\ga{m}[\sqrt{\WW}\y]_i), 
\end{equation}
where $
    \psi_i(\lm_i) = \max\limits_{x_i \in\R^{n} } \left\{ \la  \lm_i,x_i \ra - f_i(x_i) \right\}$
is the Fenchel--Legendre transform of  $f_i(x_i)$ and  the
vector $[\sqrt{\WW}\x]_i$ represents the $i$-th $n$-dimensional block of $\sqrt{\WW}\x$. 
From the fact that $F(\x)$ is 
$L_{F}$--Lipschitz smooth   and $\gamma_{F}$--strongly convex  it follows that
$\Psi(\y)$ is 
$L_{\Psi}$--Lipschitz smooth  with $L_{\Psi} = \frac{\lambda_{\max}(W)}{\gamma_F}$ and $\gamma_{\Psi}$--strongly convex with $\gamma_{\Psi} = \frac{\lambda_{\min}^{+}(W)}{L_F}$ in the $\ell_2$-norm in $y^0 + ({\rm Ker} A^T)^{\perp}$. Here $L_F = L/m$, $\gamma_F = \gamma/m$. We also consider the stochastic version of  problem \eqref{D2}, whose objectives $\psi_i$'s are given by their expectations $\psi_i(\ga{\lambda}_i) = \E[\psi_i(\ga{\lambda}_i,~ \xi_i)]$. 

We consider the unbiased stochastic dual oracle returns  $\nabla \psi_i(\ga{\lambda}_i, \xi_i)$  under the following $\sigma_{\psi}^2$-sub-Gaussian variance  condition (for all $i=1,...,m$)
\begin{equation*}
\E \exp\left({ \|\nabla \psi_i(\ga{\lambda}_i, \xi_i)- \nabla \psi_i(\ga{\lambda}_i)\|_2^2}/{\sigma_{\psi}^{2}}\right) \le \exp(1).
\end{equation*}
Problem \eqref{D2} can be considered as a particular case of problem \eqref{PP} with with  $A = \sqrt{W}$, $b=0$ and   $\sigma_{\ga{\Psi}}^2 = \ga{\lambda_{\max}(W)m\sigma_{\psi}^2}$ (Lemma \ref{lm:sigma_subgaus}).

Similarly to Chapter \ref{ch:dual} we make the following change of variables
 \[\tilde{\y} :=\sqrt{\WW}\tilde{\y}, \quad
    \mathbf{z} := \sqrt{\WW}\mathbf{z}, \quad
  \y := \sqrt{\WW}\y \]
to present the algorithms of this Chapter solving the pair of primal-dual problems \eqref{P2} and \eqref{D2} in a decentralized manner. We also need to multiply the corresponding steps in the algorithm by $\sqrt{\WW}$.

The bound  \eqref{eq:oraclefisrtcall} 
 for  the pair of decentralized primal-dual problems \eqref{P2} and \eqref{D2}  will change as follows
\[
\widetilde O\left(\sqrt{\frac{L}{\gamma}\chi(W)}\right),
\]
where we used $A = \sqrt{W}$ and the symmetry of $\sqrt{W}$, $L_F = L/m$, $\gamma_F = \gamma/m$ for \eqref{P2}.

The  bound \eqref{eq:oracle)calldual}
   for  the pair of decentralized primal-dual problems \eqref{P2} and \eqref{D2} will change as follows
 \[\widetilde{O}\left(\max\left\{\frac{\sigma^2_{\Psi}R^2_y}{\e^2}, \sqrt{\frac{L}{\gamma}\chi(W)} \right\}\right) = \widetilde O\left(\max\left\{{\frac{M^2\sigma_{\ga{\psi}}^2}{\e^2}\chi(W)}, \sqrt{\frac{L}{\gamma}\chi(W)} \right\}\right),\]
where we used $\sigma_{\ga{\Psi}}^2 = \ga{\lambda_{\max}(W)m\sigma_{\psi}^2}$ (Lemma \ref{lm:sigma_subgaus}) and 
  \cite{lan2017communication} 
\begin{align*}
   \|\lm^*\|^2_2 \leq R^2_\lm 
   &= \frac{\|\nabla F(\x^*)\|_2^2}{\lm^+_{\min}(W)} \leq \frac{ 
   \left\| \frac{1}{m}
\begin{pmatrix}
&\nabla f_1(x^*)\\
&\hspace{0.5cm}\vdots\\
&\nabla f_m(x^*)\\
\end{pmatrix}
\right\|_2^2
   }{\lm^+_{\min}(W)}  =\frac{\sum_{i=1}^m \|\nabla f_i(x^*)\|_2^2}{m^2\lm^+_{\min}(W)} \notag \\
   &\leq \frac{M^2}{m \lm^+_{\min}(W)}. 
\end{align*}

    	\begin{table}[ht!]
\caption {The optimal bounds for dual deterministic oracle}
\label{tab:distrDetCOm}
\begin{center}
{
\begin{tabular}{ lll }
 \toprule
Property of $f_i$ & \makecell[l]{  $\gamma${-strongly convex},\\  $L$-smooth} 
&  \makecell[l]{  $\gamma$-strongly convex, \\ $\|\nabla f_i(x^*)\|_2\le M$ }  \\
 \midrule
\makecell[l]{The number of \\ communication  \\ rounds} & \makecell[l]{ $\widetilde O\left(\sqrt{\frac{L}{\gamma}\chi(W)} \right)$ } &  \makecell[l]{ $O\left(\sqrt{\frac{M^2}{\gamma\e}\chi(W)}\right) $ }  \\
 \midrule
\makecell[l]{  The number of \\ oracle calls of\\ $\nabla \psi_i(\lambda_i)$
 per node $i$} & \makecell[l]{ $\widetilde O\left(\sqrt{\frac{L}{\gamma}\chi(W)} \right)$ } &  \makecell[l]{ $O\left(\sqrt{\frac{M^2}{\gamma\e}\chi(W)}\right) $ }  \\
 \bottomrule
\end{tabular}
}
\end{center}
\end{table}

	\begin{table}[ht!]
\caption{The optimal bounds for dual stochastic (unbiased) oracle}
\label{T:dist_stoch}
{\hspace{-0.4cm}
\begin{tabular}{ lll }
 \toprule
Property of $f_i$ &{\makecell[l]{  $\gamma${-strongly convex},\\  $L$-smooth} }  & \makecell[l]{  $\gamma$-strongly convex,\\ $\|\nabla f_i(x^*)\|_2\le M$ } \\
 \midrule
\makecell[l]{The number of \\ communication \\ rounds} & \makecell[l]{ $\widetilde O\left(\sqrt{\frac{L}{\gamma}\chi(W)} \right)$ }  & \makecell[l]{ $O\left(\sqrt{\frac{\ga{M^2}}{\gamma\e}\chi(W)}\right) $ } \\
 \midrule
\makecell[l]{The number of \\ oracle calls of \\
$\nabla \psi_i(\lambda_i,\xi_i)$\\
 per node $i$}  &  \makecell[l]{$\widetilde O\left(\max\left\{{\frac{M^2\sigma_{\ga{\psi}}^2}{\e^2}\chi(W)}, \sqrt{\frac{L}{\gamma}\chi(W)} \right\}\right)$ }  & \makecell[l]{ $O\left(\max\left\{ \frac{M^2\sigma_{\psi}^2}{\e^2}\chi(W), \sqrt{\frac{ M^2}{\gamma\e}\chi(W)} \right\}\right)$ }   \\
 \bottomrule
\end{tabular}
}

\end{table}
 Tables \ref{tab:distrDetCOm} and \ref{T:dist_stoch} 
summarize these bounds together with bounds from Chapter \ref{ch:decentralized}.
 Note that the bounds on communication steps (rounds) are optimal (up to a logarithmic factor) due to \cite{arjevani2015communication,scaman2017optimal,scaman2018optimal}.
Bounds for the oracle calls per node are probably optimal \ga{in the class of  methods with  optimal number of communication steps} (up to a logarithmic factor) in the deterministic case \ga{\cite{allen2018make,foster2019complexity,woodworth2018graph}} and optimal \g{for the non-smooth stochastic primal oracle and stochastic dual oracle}
for parallel architecture.\footnote{In parallel architecture the bounds on stochastic oracle calls per node of type $\max\{B,D\}$ can be \g{parallel} up to $B/D$ processors.} For stochastic oracle the bounds hold  in terms of high probability deviations (we skip the corresponding logarithmic factor).

The detailed proofs of the statements of this Chapter can be   found in the arXiv preprint \cite{gorbunov2019optimal}.





\addcontentsline{toc}{chapter}{References}

\small
\bibliographystyle{apalike}
\bibliography{ref}

\chapter*{Declaration}


\noindent I declare that I have completed the thesis independently. I have not applied for a doctor’s degree in the doctoral subject elsewhere and do not hold a corresponding doctor’s degree. \\
\vspace{0.5cm}

\noindent  Berlin, 03.05.2021 \hspace{9cm} Darina Dvinskikh

\end{document}